\documentclass[a4paper,11pt]{article}
\usepackage{amssymb,amsthm,amsmath,amscd}
\usepackage{epsfig}
\usepackage{subfigure}
\usepackage{fancyhdr}
\usepackage[top=1in,bottom=1.5in,left=1in,right=1in]{geometry}
\usepackage[all]{xy}
\newtheorem{theorem}{Theorem}[section]

\newtheorem{definition}{Definition}[section]
\newtheorem{lemma}{Lemma}[subsection]
\newtheorem{proposition}{Proposition}[section]
\newtheorem*{theorem*}{Theorem}

\newtheorem{remark}{Remark}[section]
\newcommand{\be}{\begin{eqnarray}}
\newcommand{\ee}{\end{eqnarray}}
\newcommand{\ben}{\begin{eqnarray*}}
\newcommand{\een}{\end{eqnarray*}}

\DeclareMathOperator{\dist}{dist}
\DeclareMathOperator{\qdist}{qdist}
\DeclareMathOperator{\supp}{supp}
\DeclareMathOperator{\diam}{diam}
\DeclareMathOperator{\pds}{{p_{\it d}}sin}
\DeclareMathOperator{\pdms}{{p_{{\it d}-1}sin}}

\DeclareMathOperator{\SCale}{{scale}}
\DeclareMathOperator{\scl}{{sc}}
\def\Supp{{\supp(\mu)}}

\providecommand{\norm}[1]{\lVert#1\rVert}

\newcommand{\reals}{\mathbb R}

\newcommand{\ints}{\mathbb Z}
\newcommand{\nats}{\mathbb N}
\newcommand{\natsz}{\nats_0}

\def\RR{{\mathbb{R}}}
\def\Max{{\mathrm{max}}}

\def\MM{{\mathrm{M}}}
\def\name{{\mathbf{S}}}
\def\HH{{H}}
\newcommand{\di}{\, \mathrm{d}}
\newcommand{\fdi}{\mathrm{d}}

\newcommand{\latop}[2]{\substack{{#1}\\{#2}}} 

\begin{document}

\title{High-Dimensional Menger-Type Curvatures - Part I:\\
Geometric Multipoles and Multiscale Inequalities \thanks{This work
was supported by NSF grant \#0612608 }}

\author{
Gilad Lerman\thanks{lerman@umn.edu} ~and J. Tyler
Whitehouse\thanks{jonathan.t.whitehouse@vanderbilt.edu}}

\maketitle

\begin{abstract}
We define discrete and continuous Menger-type curvatures. The
discrete curvature scales the volume of a $(d+1)$-simplex in a real
separable Hilbert space $\HH$, whereas the continuous curvature
integrates the square of the discrete one according to products of a
given measure (or its restriction to balls). The essence of this
paper is to establish an upper bound on the continuous Menger-type
curvature of an Ahlfors regular measure $\mu$ on $\HH$ in terms of
the Jones-type flatness of $\mu$ (which adds up scaled errors of
approximations of $\mu$ by $d$-planes at different scales and
locations). As a consequence of this result we obtain that uniformly
rectifiable measures satisfy a Carleson-type estimate in terms of
the Menger-type curvature. Our strategy combines discrete and
integral multiscale inequalities for the polar sine with the
``geometric multipoles'' construction, which is a multiway analog of
the well-known method of fast multipoles.
\end{abstract}
%

\section{Introduction}
We introduce Menger-type curvatures of measures and show that they
satisfy a Carleson-type estimate when the underlying measures are
uniformly rectifiable in the sense of David and Semmes~\cite{DS91,
DS93}. The main development of the paper (implying this
Carleson-type estimate) is a careful bound on the Menger-type
curvature of an Ahlfors regular measure in terms of the sizes of
least squares approximations of that measure at different scales and
locations.

Our setting  includes a real separable Hilbert space $H$ with
dimension denoted by $\dim(\HH)$ (possibly infinite), an intrinsic
dimension $d\in\nats$, where $d< \dim(\HH)$, and a $d$-regular
(equivalently, $d$-dimensional Ahlfors regular) measure $\mu$ on
$H$. That is, $\mu$ is locally finite Borel measure on $\HH$ and
there exists a constant $C \geq 1$ such that for all $x\in\Supp$ and
$0 < r\leq\diam(\Supp)$:
\begin{equation}\label{equation:measure-comp}C^{-1}\cdot
r^d\leq\mu(B(x,r))\leq C\cdot r^d.\end{equation}
We denote the smallest constant $C$ satisfying
equation~\eqref{equation:measure-comp} by $C_{\mu},$ and refer to it
as the regularity constant of $\mu$. The estimates developed in this
paper only depend on the intrinsic dimension $d$ and the regularity
constant $C_{\mu}$, and no other parameter of either $\mu$ or $H$.
In particular, they are independent of the dimension of $H$.

Our $d$-dimensional discrete curvature is defined on vectors
$v_1,\ldots ,v_{d+2}\in H$. We denote the diameter of the set
$\{v_1,\ldots ,v_{d+2}\}$ by $\diam (v_1,$ $\ldots ,$ $v_{d+1})$,
and the $(d+1)$-dimensional volume of the parallelotope spanned by
$v_2-v_1,\ldots ,v_{d+2}-v_1$ by ${\rm Vol}_{d+1}(v_1,\ldots
,v_{d+2})$. Equivalently, ${\rm Vol}_{d+1}(v_1,\ldots ,v_{d+2})$ is
$(d+1)!$ times the volume of the simplex (i.e., convex polytope)
with vertices at $v_1,\ldots ,v_{d+2}$. The square of our
$d$-dimensional curvature $c_d(v_1,\ldots ,v_{d+2})$ has the form
$$
c_d^2(v_1,\ldots ,v_{d+2}) = \frac{1}{d+2}\cdot \frac{{\rm
Vol}^2_{d+1}(v_1,\ldots ,v_{d+2})}{\diam (v_1,\ldots
,v_{d+2})^{d\cdot (d+1)}} \,
\sum^{d+2}_{i=1}\frac{1}{\prod^{d+2}_{\latop{j=1}{j\ne i}}\Vert
v_j-v_i\Vert^2_2} \,.
$$

The one-dimensional curvature $c_1(v_1,v_2,v_3)$ is comparable to
the Menger curvature~\cite{M30,MMV96}, $c_M(v_1,v_2,v_3)$, which is
the inverse of the radius of the circle through the points
$v_1,v_2,v_3 \in H$. Indeed, we note that
$$
c^2_M(v_1,v_2,v_3) = \frac{4 \sin^2(v_2-v_1,v_3-v_1)}{\Vert
v_2-v_3\Vert^2}\,,
$$
$$
c^{2}_1(v_1,v_2,v_3) =
\frac{\sin^2(v_2-v_1,v_3-v_1)+\sin^2(v_1-v_2,v_3-v_2)
+\sin^2(v_1-v_3,v_2-v_3)}{3\cdot\diam^2(v_1,v_2,v_3)}\,,
$$
and consequently
\begin{equation*}
\label{eq:compare_menger} \frac{1}{12}\cdot c^2_M(v_1,v_2,v_3) \leq
c^{2}_1(v_1,v_2,v_3) \leq \frac{1}{4}\cdot c^2_M(v_1,v_2,v_3)\,.
\end{equation*}
We thus view the Menger-type curvature as a higher-dimensional
generalization of the Menger curvature. Clearly, one can directly
generalize the Menger curvature to the following function of $v_1$,
$\ldots$, $v_{d+2}$:
\be \label{eq:direct_menger} \frac{{\rm Vol}_{d+1}(v_1,\ldots
,v_{d+2})}{
\prod_{ \latop{i,j=1}{i \neq j}}^{d+2} \Vert v_i-v_j\Vert} \,. \ee
However, the methods developed here do not apply to that curvature
(see Remark~\ref{remark:direct_menger}).

Essentially, this paper shows how the multivariate integrals of the
discrete $d$-dimensional Menger-type curvature can be controlled
from above by $d$-dimensional least squares approximations of $\mu$,
which are used to characterize uniform rectifiability~\cite{DS91,
DS93}.

We first exemplify this in the simplest setting of
approximating $\mu$ by a fixed $d$-dimensional plane at a given
scale and location, indicated by the ball $B = B(x,t)$, for
$x\in{\rm supp} (\mu)$ and $0<t\leq\diam ({\rm supp}(\mu))$. We
denote the scaled least squares error of
approximating $\mu$ at $B = B(x,t)$ by a $d$-plane (i.e.,
$d$-dimensional affine subspace) by
$$
\beta^2_2(x,t) =\beta^2_2(B) =\min_{d-\text{planes } L}\int_B
\left(\frac{{\rm dist}(x,L)}{\diam(B)}\right)^2\frac{\fdi\mu
(x)}{\mu (B)}\, .
$$

Fixing $\lambda>0$ and sampling sufficiently
well separated simplices in $B^{d+2}$, i.e., simplices in the set
\begin{equation}\label{equation:big-U}U_{\lambda}(B)=\left\{(v_1,\ldots ,v_{d+2})\in
B^{d+2}:\min_{1\leq i<j\leq d+2}\|v_i-v_j\|\geq\lambda\cdot
t\right\},\end{equation} one can bound $\beta^2_2(x,t)$ from below
by averages of the squared curvature $c^2_d$ in the following way:
\begin{proposition}\label{proposition:upper-bound-big-scale}There exists a constant
$C_0=C_0(d,C_{\mu})\geq1$ such that
\begin{equation}\label{equation:upper-big-scale}\int_{U_{\lambda}(B(x,t))}c_d^2(X)\di\mu^{d+2}(X)\leq
\frac{C_0}{\lambda^{d(d+1)+4}}\cdot\beta_2^2(x,t)\cdot\mu(B(x,t)),\end{equation}
for all $\lambda>0$,  $x\in\Supp,$ and $t\in\RR$ with
$0<t\leq\diam(\Supp)$.
\end{proposition}

An opposite inequality is established
in~\cite[Theorem~1.1]{LW-part2}. An extension of
Proposition~\ref{proposition:upper-bound-big-scale} to more general
measures and to arbitrary simplices (while slightly modifying the
curvature) will appear in~\cite{LW-volume}.

We next extend the above estimate to multiscale least squares
approximations. For this purpose we first define the Jones-type
flatness~\cite{Jo90,DS91,DS93} of the measure $\mu$ when restricted
to a ball $B\subseteq H$  as follows
\label{eq:def_jones_flat} \be J_d(\mu|_B) = \int^{\diam (B)}_0\int_B
\beta^2_2(x,t)\di \mu (x)\frac{\fdi t}{t}\,. \ee
This quantity measures total flatness or oscillation of $\mu$ around
$B$ by combining the errors of approximating it with $d$-planes at
different scales and locations. The actual weighting of the
$\beta_2$ numbers is designed to capture the uniform rectifiability
of $\mu$ (see Section~\ref{section:uniform_rect}). We also define
the continuous Menger-type curvature of $\mu$ when restricted to
$B$, $c_d(\mu|_B)$, as follows
$$ c_d(\mu|_B) = \sqrt{\int_{B^{d+2}} c_d^2(v_1,\ldots,v_{d+2})
\di \mu(v_1) \ldots \di \mu(v_{d+2})}\,.
$$

 The primary result of this paper bounds the local Jones-type flatness from below by the local Menger-type curvature in the following way.
\begin{theorem}\label{theorem:upper-main}
There exists a constant $C_1=C_1(d,C_{\mu})$ such that
$$c_d^2\left(\mu|_B\right)\leq C_1\cdot
J_d^{\mathcal{D}}\left(\mu|_B\right) \ \text{ for all balls } B
\subseteq H.$$ \end{theorem}

This theorem is relevant for the theory of uniform
rectifiability~\cite{DS91,DS93} (briefly reviewed in
Section~\ref{section:uniform_rect}) and it in fact implies the
following result.
\begin{theorem}
\label{theorem:uniform-rect-menger} If $\HH$ is a real separable
Hilbert space and $\mu$ is a $d$-dimensional uniformly rectifiable
measure then there exists a constant $C_2 = C_2(d,C_{\mu})$ such
that
\be \label{eq:multi_d_cond} c^2_{d}(\mu|_B)\leq C_2 \cdot \mu (B) \
\text{ for all balls } B \subseteq H. \ee
\end{theorem}

In~\cite{LW-part2} we show that the opposite direction of
Theorem~\ref{theorem:upper-main}, and thus also conclude the
opposite direction of Theorem~\ref{theorem:uniform-rect-menger},
that is, the Carleson-type estimate of
Theorem~\ref{theorem:upper-main} implies the uniform rectifiability
of $\mu$.  As such, we obtain  a characterization of uniformly
rectifiable measures by the Carleson-type estimate of
equation~\eqref{eq:multi_d_cond} (extending the one-dimensional
result of~\cite{MMV96}).

When $d=1$ a similar version of Theorem~\ref{theorem:upper-main} was
formulated and proved in~\cite[Theorem~31]{pajot_book} following
unpublished work of Peter Jones~\cite{Jo_unpublished}. The
difference is that~\cite[Theorem~31]{pajot_book} uses the larger
$\beta_\infty$ numbers (i.e., the analogs of the $\beta_2$ numbers
when using the $L_\infty$ norm instead of $L_2$) and restricts the
support of $\mu$ to be contained in a rectifiable curve. The latter
restriction requires only linear growth of $\mu$, i.e., one can
consider Borel measures satisfying only the RHS of
equation~\eqref{equation:measure-comp}.

The proof of Theorem~\ref{theorem:upper-main} when $d>1$ requires
more substantial development than that of~\cite{Jo_unpublished}
and~\cite[Theorem~31]{pajot_book}. This is for a few reasons, a
basic one being the greater complexity of higher-dimensional
simplices vis \'{a} vis the simplicity of triangles. The result is
that many more things can go wrong while trying to control the
curvature $c_d$ for $d>1$, and we are forced to invent strategies
for obtaining the proper control.

A more subtle reason involves the difference between the $L_\infty$
and $L_2$ quantities and the way that these interact with some
pointwise inequalities. In the case for $d=1$ and the $\beta_\infty$
numbers, much of the proof (as recorded in~\cite{pajot_book}) is
driven by the ``triangle inequality'' for the ordinary sine
function, i.e., the subadditivity of the absolute value of the sine
function. The ``robustness'' of this inequality combined with the
simplicity of a basic pointwise inequality between the sine function
and the $\beta_\infty$ numbers results in a relatively simple
integration procedure.  For $d>1$ and the $\beta_2$ numbers this
whole framework breaks down. One such breakdown is that the
``correct'' analog of the triangle inequality holds much more
sparsely (see
Proposition~\ref{proposition:concentration-inequality-1}). Other
reasons will become apparent in the rest of the work.

In principal, there are two kinds of methods in the current work. We
refer to the first as geometric multipoles. It decomposes the
underlying multivariate integral over a set of well-scaled simplices
(i.e., with comparable edge lengths) according to multiscale regions
in $H$, emphasizing in each region approximations of the support of
the measure by $d$-planes.  We view it as a $d$-way analog of the
zero-dimensional fast multipoles method~\cite{GR87}, which
decomposes an integral according to dyadic grids of ${\mathbf R}^n$
and emphasizes near-field interactions. This method, which is rather
implicit in~\cite[Theorem~31]{pajot_book}, can be used to decompose
integrals of many other multivariate functions.

Our second method relies on both discrete and integral multiscale
inequalities for the Menger-type curvature (or more precisely the
polar sine function defined in Section~\ref{sec:geo_prop}). They
allow us to bound $c^2_{d}(\mu|_B)$ by multivariate integrals
restricted to sets of well-scaled simplices, so that geometric
multipoles can then be applied. When $d=1$ both the multiscale
inequality and its application are rather trivial
(see~\cite[Lemma~36]{pajot_book} and the way it is used
in~\cite[Theorem~31]{pajot_book}).



\subsection{Organization of the Paper}
In Section~\ref{section:context} we describe the basic context,
notation, definitions and related elementary propositions. In
Section~\ref{sec:geo_prop} we review some geometric properties of
simplices as well as of the $d$-dimensional polar sine, and we
decompose simplices and correspondingly the Menger-type curvature
according to scales and configurations.
Section~\ref{section:uniform_rect} reviews aspects of the theory of
uniform rectifiability relevant to this work. In
Section~\ref{section:belowed} we establish
Proposition~\ref{proposition:upper-bound-big-scale}, whereas in
Section~\ref{section:major-reduction} we reduce
Theorem~\ref{theorem:upper-main} to three propositions, which we
subsequently prove in
Sections~\ref{section:journe}-\ref{section:later-integrate}.

\section{Basic Notation and Definitions}
\label{section:context}
%
We denote the support of the $d$-regular measure $\mu$ on $H$ by
$\Supp$. The inner product and induced norm on $H$ are denoted by
$<\cdot,\cdot>$ and $\|\cdot\|$. For $m \in \nats$, we denote the
Cartesian product of $m$ copies of $H$ by $H^m$ and the
corresponding product measure by $\mu^m$.

We summarize some notational conventions as follows. We typically
denote scalars larger than $1$ by upper-case plain letters, e.g.,
$C$; arbitrary integers by lowercase letters, e.g., $i,j$, and large
integers by $M$ and $N$; real numbers by lower-case Greek or script
letters, e.g., $\alpha_0$, $r$; subsets of $H^m$, where $m \in
\nats$, by upper-case plain letters, e.g., $A$; families of subsets
(e.g., collections of balls)  by calligraphic letters, e.g., $\cal
B$; subsets of $\nats$ (used for indexing) by capital Greek letters,
e.g., $\Lambda$; and measures on $H$ by Greek lower-case letters,
e.g., $\mu$.

We reserve  $x$, $y$ and $z$ to denote elements of $H$; $X$, $Y$ and
$Z$ to denote elements of $H^m$ for $m \geq 3$;  $L$ for a complete
affine subspace of $H$ (possibly a linear subspace); $V$ to denote a
complete linear subspace of $H$.
 If $x\in {\bf R}$, then we denote the corresponding ceiling and
floor functions by $\lceil x\rceil$ and $\lfloor x\rfloor$.

If $A\subseteq H$, we denote its diameter by $\diam(A)$, its
complement by $A^c$ and the restriction of $\mu$ to it by $\mu|_A$.

We denote the closed ball centered at $x\in\HH$ of radius $r$ by
$B(x,r)$, and  if both the center and radius are indeterminate, we
use the notation $B$ or $Q$. For a ball $B(x,r)$  and $\gamma>0$, we
denote the corresponding blow up by  $\gamma\cdot B(x,r)$, i.e.,
$\gamma\cdot B(x,r)=B(x,\gamma\cdot r)$. If $\mathcal{B}$ is a
family of balls, then we denote the corresponding blow up by $
\gamma\cdot\mathcal{B}=\left\{\gamma\cdot B: B
\in\mathcal{B}\right\}. $

If $L$ is a complete affine subspace of $\HH$ and $x\in\HH$, we
denote the distance between $x$ and $L$ by $\dist(x,L)$, that is,
$\dist(x,L)=\min_{y\in L}\|x-y\|$. If $n\leq\dim(H)$, we at times
use the phrase $n$-plane to refer to an $n$-dimensional affine
subspace of $H$.

If $\mathcal{B}$ is a family of balls, then we denote the union of
its elements  by $\bigcup\mathcal{B}$, and we distinguish the latter
notation from $\bigcup_{n \in\,\mathbb{Z}}\mathcal{B}_n$, which is a
family of balls formed by the countable union of other families.

We fix the constant
\begin{equation}
\label{equation:upper-bound-relaxed-constant}C_{\mathrm{p}}
=C_{\mathrm{p}}(d,C_{\mu})=
\begin{cases}
\displaystyle\frac{\sqrt{5}\cdot\pi^2}{4\cdot\arcsin\left(2^{-(5/2\cdot
d+1)}\cdot C_{\mu}^{-2}\right)}\,,  &\textup{if } \
d > 1;\\
\quad\quad\quad\quad\quad\quad 1, &\textup{if } \ d=1.
\end{cases}
\end{equation}
We also fix the following constant $\alpha_0$ and use its powers to
provide appropriate scales:
\begin{equation}\label{equation:alpha-loose}
\alpha_0 = \alpha_0(d,C_{\mu})=\min\left\{\frac{1}{2\cdot
C_{\mathrm{p}}^2}\,,\left(\frac{1}{4\cdot
C_{\mu}^2}\right)^{1/d}\right\}=\begin{cases}\displaystyle\frac{1}{4\cdot
C_{\mu}^2}\,,&\textup{ if }d=1;\\\displaystyle\frac{1}{2\cdot
C_{\mathrm{p}}^2}\,,&\textup{ if }d>1.\end{cases}
\end{equation}

 Finally, for a fixed $d$-plane $L$ and a ball $B=B(x,t)$, we  let
 $$\beta_2^2(B,L)=\beta_2(x,t,L)=\int_B\left(\frac{\dist(x,L)}{\diam(B)}\right)^2\frac{\di\mu(x)}{\mu(B)},$$
where $\beta_2^2(B,L)=0$ if $\mu(B)=0$.  We note that $$\beta_2^2(B)=\inf_{L}\beta_2^2(B,L).$$

\subsection{Notation Corresponding to Elements of $H^{n+1}$}

Throughout this paper, we frequently refer to $n$-simplices in $\HH$
for $n \geq 2$ where usually $n=d+1$. Rather than formulate our work
with respect to subsets of $\HH$, we  work with ordered
($n+1$)-tuples of the product space, $H^{n+1}$, representing the set
of vertices of the corresponding simplex in $H$. Fixing  $n \geq 2$,
we denote an element of $\HH^{n+1}$ by $X=(x_0,\ldots,x_{n})$, and
for $0 \leq i \leq n$, we let
$$(X)_i=x_i$$ denote the projection of $X$ onto its $i^{\text{th}}$
$H$-valued {\em coordinate}.
We make a clear distinction between the symbol $X_i$, denoting an
indexed element of the product space $\HH^{n+1}$, and the coordinate
$(X)_i$ which is a point in $\HH$. The zeroth coordinate $(X)_0=x_0$
is special in many of our calculations. With some abuse of notation
we refer to $X$ both as an ordered set of $d+2$ vertices and as a
$(d+1)$-simplex.

For $0\leq i\leq n$ and $X=(x_0,\ldots,x_{n})\in\HH^{n+1}$, let
$X(i)$ be the following element of $H^n$:
\begin{equation}\label{equation:def-removal}
\nonumber X(i)=(x_0,\ldots, x_{i-1},x_{i+1},\ldots,x_{n}).
\end{equation}
That is, $X(i)$ is the projection of $X$ onto $\HH^{n}$ that
eliminates its $i^{\text{th}}$ coordinate.

If $X\in H^{n+1}$, $y\in H$ and $1\leq i\leq n$, we form $X(y,i)\in
H^{n+1}$ as follows:
\begin{equation}
\nonumber X(y,i)=(x_0,\ldots,x_{i-1},y,x_{i+1},\ldots,x_n)\,
\end{equation}
That is, $X(y,i)$ is obtained from $X$ by replacing its
$i^{\text{th}}$ coordinate $(X)_i$ with $y$.

We say that $X$ is {\em non-degenerate} if the set
$\left\{x_1-x_0,\ldots, x_{n}-x_0\right\}$ is linearly independent.
For $X\in H^{n+1}$ as above, let $L[X]$ denote the affine subspace
of $H$ of minimal dimension containing set of vertices of $X$,
$\{x_0,\ldots,x_n\}$, and let $V[X]$ be the linear subspace parallel
to $L[X]$.

\section{Simplices, Polar Sines, and Menger-type Curvatures}
\label{sec:geo_prop}

We are only interested in simplices without any coinciding vertices,
that is, simplices represented by elements in the set
\be \label{eq:def_S} S = \{X\in H^{d+2}:{\rm min}(X)>0\} .\ee
We note that the restriction to $S$ is natural since
$\mu^{d+2}\left(H^{d+2}\setminus S\right) = 0$. We describe some
properties and functions of such simplices as follows.

\subsection{Height, Content and Scale}\label{subsection:content}

Fixing $X = (x_0,\ldots ,x_{d+1})\in \HH^{d+2}$ and  $0\leq i\leq
d+1$, we define the height of the simplex $X$ through the vertex
$x_i$ as
\begin{equation}
\label{equation:def-heights}h_{x_i}(X)=\dist
\left(x_i,L[X(i)]\right).
\end{equation}
We denote the minimal height  by
\begin{equation}
\label{equation:def-minimal-height}h(X)= \min_{x_i}\
h_{x_i}(X).\end{equation}

The {\em $n$-content} of $X$, denoted by $\MM_n(X)$, is
\begin{equation}\label{equation:def-content-funct}
\mathrm{M}_{n}(X)= \left(
\mathrm{det}\left[
\big\{\langle x_i-x_0,x_j-x_0\rangle\big\}_{i,j=1}^{n}
\right]
\right)^\frac{1}{2}\, .
\end{equation}
Alternatively, the  $n$-content of $X$ is the $n$-dimensional
Lebesgue measure of a parallelotope generated by the images of the
vertices of $X$ under any isometric embedding of $L[X]$ in $\RR^n$.

For $X$, we denote its largest edge length by $$\diam(X)=\max_{0\leq i<j\leq n}\|x_i-x_j\|,$$
and its minimal edge length by $$\min(X)=\min_{0\leq i<j\leq n}\|x_i-x_j\|.$$

 Given a simplex $X$, we quantify the
disparity between the largest and smallest edges of $X$ at  $x_0$
using the  functions $$
\max\nolimits_{x_0}(X)=\max_{1\leq i\leq d+1}\|x_i-x_0\| \ \textup{ and
}\  \min\nolimits_{x_0}(X)=\min_{1\leq i\leq
d+1}\|x_i-x_0\|, $$ as well as the function
$$
\SCale_{x_0}(X)=\frac{\min\nolimits_{x_0}(X)}{\max\nolimits_{x_0}(X)}.$$
\subsection{ Polar Sines and Elevation Sines}
\label{subsection:polar}

We define the $d$-{\em dimensional polar sine} of the element
$X=(x_0,\ldots,x_{d+1}) \in S$ with respect to the coordinate $x_i$,
$0\leq i\leq d+1$, as
\begin{equation}\label{equation:def-psin}
\pds_{x_i}(X) = \frac{\mathrm{M}_{d+1}(X)}{\prod_{\substack{0\leq j\leq d+1\\
j\not=i}}\|x_j-x_i\|},
\end{equation}
and exemplify it in Figure~\ref{subfig:polar}. If $X \notin S$, we
let $\pds_{x_i}(X)=0$. When $d=1$, the polar sine reduces to the
ordinary sine of the angle between two vectors.
Unlike~\cite{LW-semimetric}, our definition here only allows
non-negative values of the polar sine. We note that
$\pds_{x_i}(X)=0$ for some $0\leq i\leq d+1$ if and only if $X$ is
degenerate.

For $n\geq1$, $X\in S$, and $1\leq i \leq n+1$, we also define the
{\em elevation angle of $x_i-x_0$ with respect to
$V\left[X(i)\right]$}, denoted by $\theta_{i}(X)$, to be the acute
angle such that
\begin{equation}\label{eq:elevationsine}\sin(\theta_i(X)) = \frac{\dist(x_i,L[X(i)])}
{\|x_i-x_0\|}\,.
\end{equation}
We exemplify it in Figure~\ref{subfig:elevation}.

The polar sine has the following
product formula in terms of elevation angles~\cite{LW-semimetric}:
\begin{proposition}\label{proposition:product-sine}If
$X=(x_0,\ldots,x_{d+1})\in\HH^{d+2}$ and $1\leq i\leq d+1$, then
$$
\pds_{x_0}(X)=\sin\left(\theta_i(X)\right) \cdot
\pdms_{x_0}\left(X(i)\right).
$$
\end{proposition}
Iterating this product formula we have the estimate
\begin{equation}
\label{equation:0-1-valued} 0\leq \pds_{x_i}(X)\leq 1, \ \text{ for
all}\ 0\leq i\leq d+1 .
\end{equation}

\begin{figure}[htbp] \centering
\subfigure[\small The induced parallelotope and normalizing edges
for computing $\mathrm{p}_{2}\sin_{x_0}(X)$]
{\label{subfig:polar}
\includegraphics[height=2in,width=3in]
{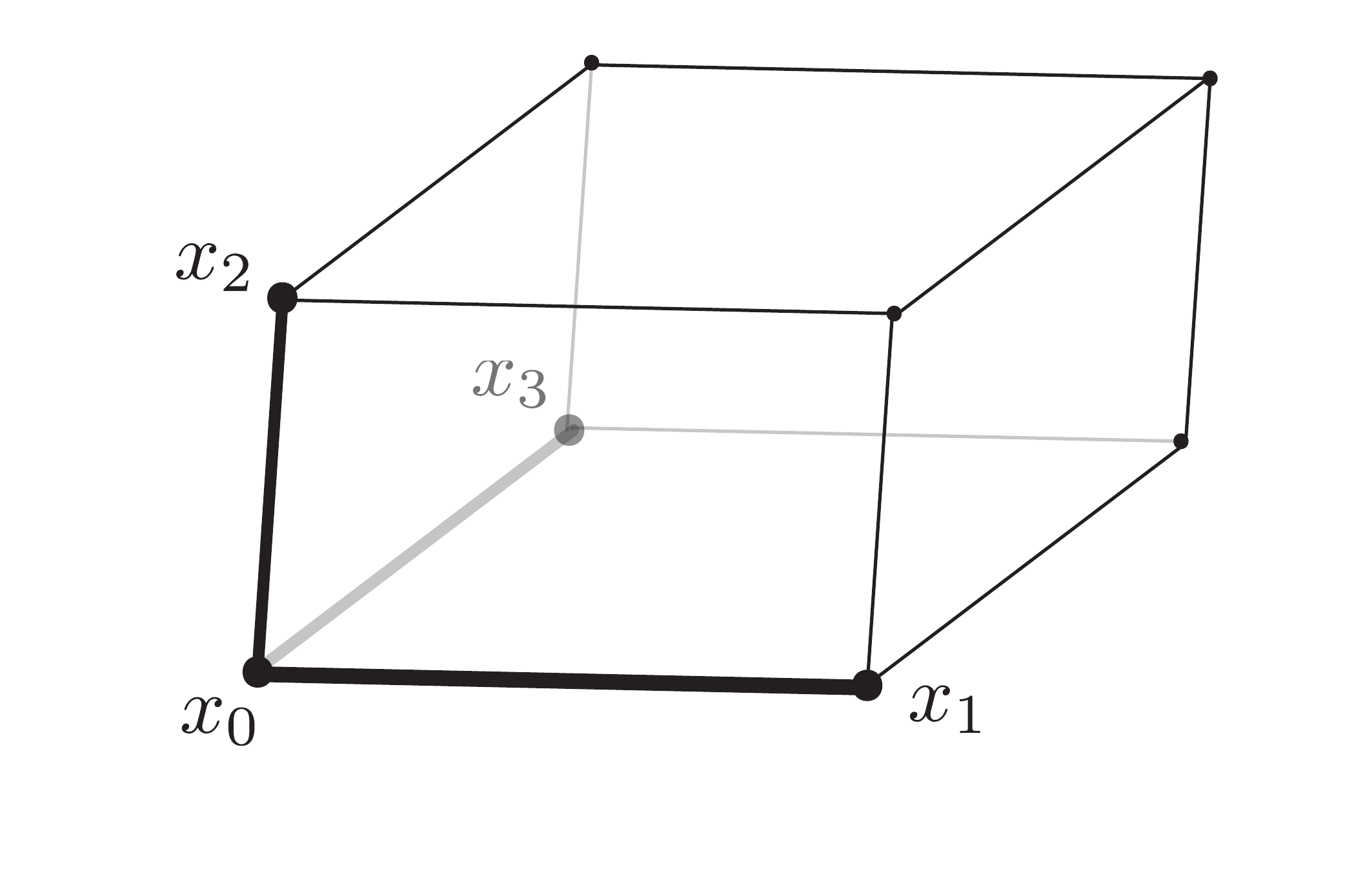}}
\subfigure[\small The elevation angle $\theta_1(X)$] 
{\label{subfig:elevation}
\includegraphics[height=2in,width=3in]
{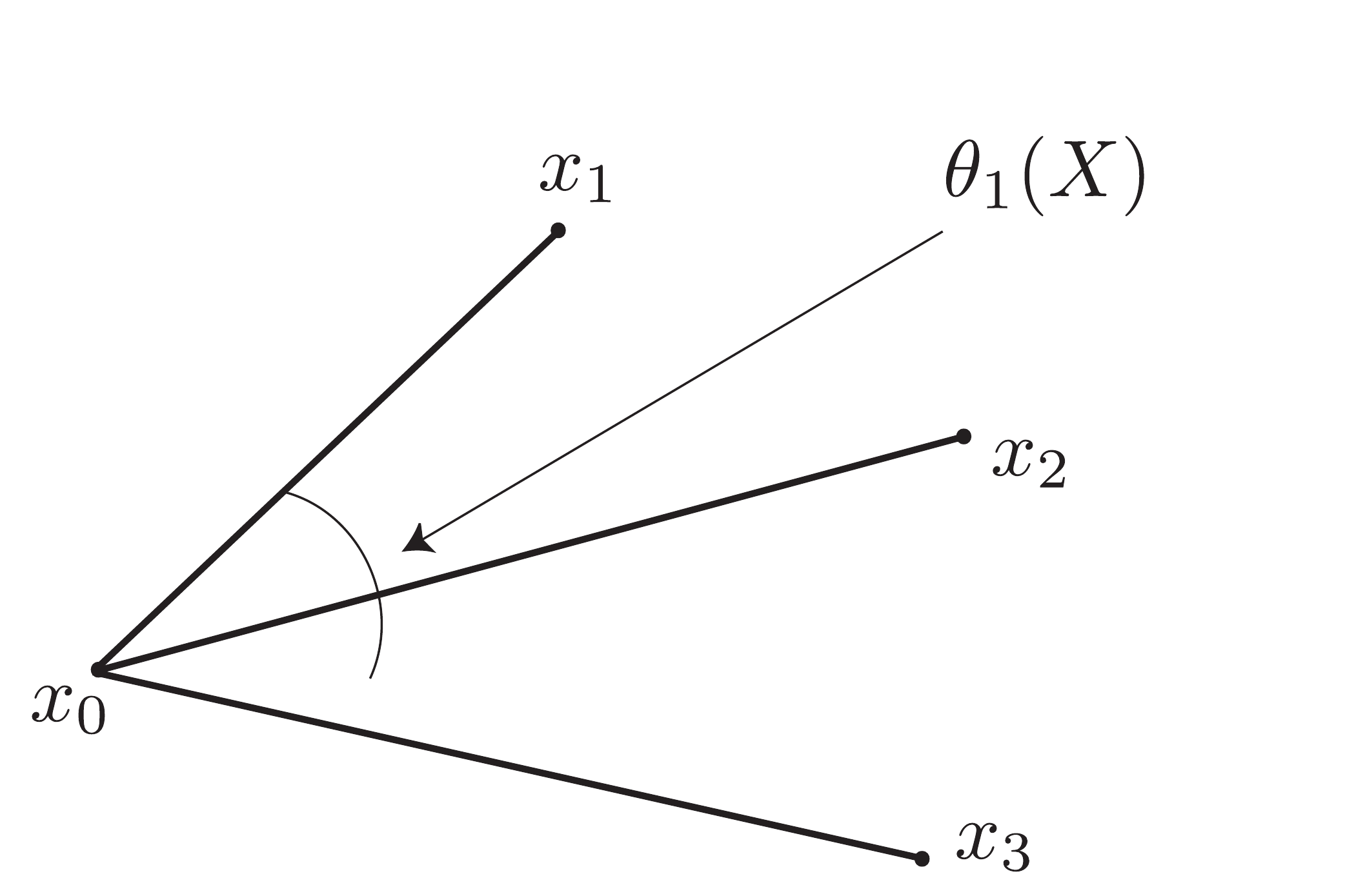}}
%
%
%
\caption{\label{figure:polar} {\small Exemplifying the concepts of
polar sine and elevation angle for tetrahedra of the form
$X=(x_0,x_1,x_2,x_3)$. The polar sine $\mathrm{p}_{2}\sin_{x_0}(X)$
is obtained by dividing the volume of the parallelotope in (a) by
the lengths of the corresponding edges through $x_0$ (indicated in
(a) by the thick lines, where the two in the foreground are dark and
the one in the background is light). The elevation angle
$\theta_1(X)$ in (b) is the angle between the vector $x_1-x_0$ and
its projection onto the face $X(1)$.}}
\end{figure}

\subsubsection{Linear Deviations and Their Use in Bounding the Polar
Sine}\label{subsection:deviations}

Fixing $X \in H^{d+2}$ and $L$ an affine subspace of $\HH$, we
define the  $\ell_2$ deviation of $X$ from $L$, denoted by
$D_2(X,L)$, as follows:
\begin{equation}\label{equation:def-linear-deviation-2}
D_2(X,L)=
\left(\sum_{i=0}^{d+1}\dist^2\left(x_i,L\right)\right)^{1/2}.
\end{equation}
Using this quantity, we get the following upper bound on the polar
sine, which we establish in Appendix~\ref{app:bound-deviate}.
\begin{proposition}\label{proposition:psin-bound-deviations}
If $X\in S$ and $L$ is an arbitrary $d$-plane of $\HH$, then
$$
\pds_{x_0}(X)\leq \frac{\sqrt{2}\cdot(d+1)\cdot
(d+2)}{\SCale_{x_0}(X)}\ \frac{D_2(X,L)}{\diam(X)}\,.
$$
\end{proposition}

\subsubsection{Concentration Inequality for the Polar
Sine}\label{subsection:ineq-polar} For $X\in[\Supp]^{d+2}$,
$C\geq1$, and $1\leq i<j\leq d+1$, we define
\begin{equation}
\label{equation:two-term-set}U_C(X,i,j)=\Big\{y\in\Supp:\pds_{x_0}(X)\leq
C\cdot\left(\pds_{x_0}(X(y,i)) +\pds_{x_0}(X(y,j))\right)\Big\}.
\end{equation}
Using  $C_{\mathrm{p}}$ of
equation~(\ref{equation:upper-bound-relaxed-constant}), we have the
following concentration inequality proved in~\cite{LW-semimetric}.
\begin{proposition}\label{proposition:concentration-inequality-1}If $X\in[\Supp]^{d+2}$ with $x_0 = (X)_0$, and $1 \leq i < j \leq d+1$, then the
following inequality holds for all $r\in\RR$ such that
$0<r\leq\diam(\Supp):$
$$\frac{\mu\left(U_{C_{\mathrm{p}}}(X,i,j)\cap
B(x_0,r)\right)}{\mu(B(x_0,r))} \geq
\begin{cases}\displaystyle 1, &\textup{if } d=1;\\
\displaystyle 0.75, &\textup{if } d>1.
\end{cases}$$

\end{proposition}

\subsection{Categorizing Simplices by Scales and Configurations}\label{subsection:cat-simplices}

We decompose the set  $S$ of equation~\eqref{eq:def_S} according to
the size of $\SCale_{x_0}(X)$ and the configuration of the vertices
of $X$.

\subsubsection{Decomposing the Set of Simplices
$S$ According to Scale}\label{subsubsection:simplex-covers}

Here we categorize simplices according to their scales (defined with
respect to their first coordinate $x_0$) and distinguish between
well-scaled and poorly-scaled simplices (with respect to $x_0$). For
$k \in \natsz$ and $p \in \nats$ we form the following subset of $S$
\begin{equation}
\label{equation:def-S-k}S_{k,p}=\left\{X\in S: \
\alpha_0^{k+p}<\SCale_{x_0}(X) \leq\alpha_0^k\right\}.
\end{equation}
We can then decompose $S$ by $\{S_{k,p}\}_{k,p \in \nats}$.

\begin{figure}[htbp]
\begin{center}
\includegraphics[height=2.858in,width=5in]
{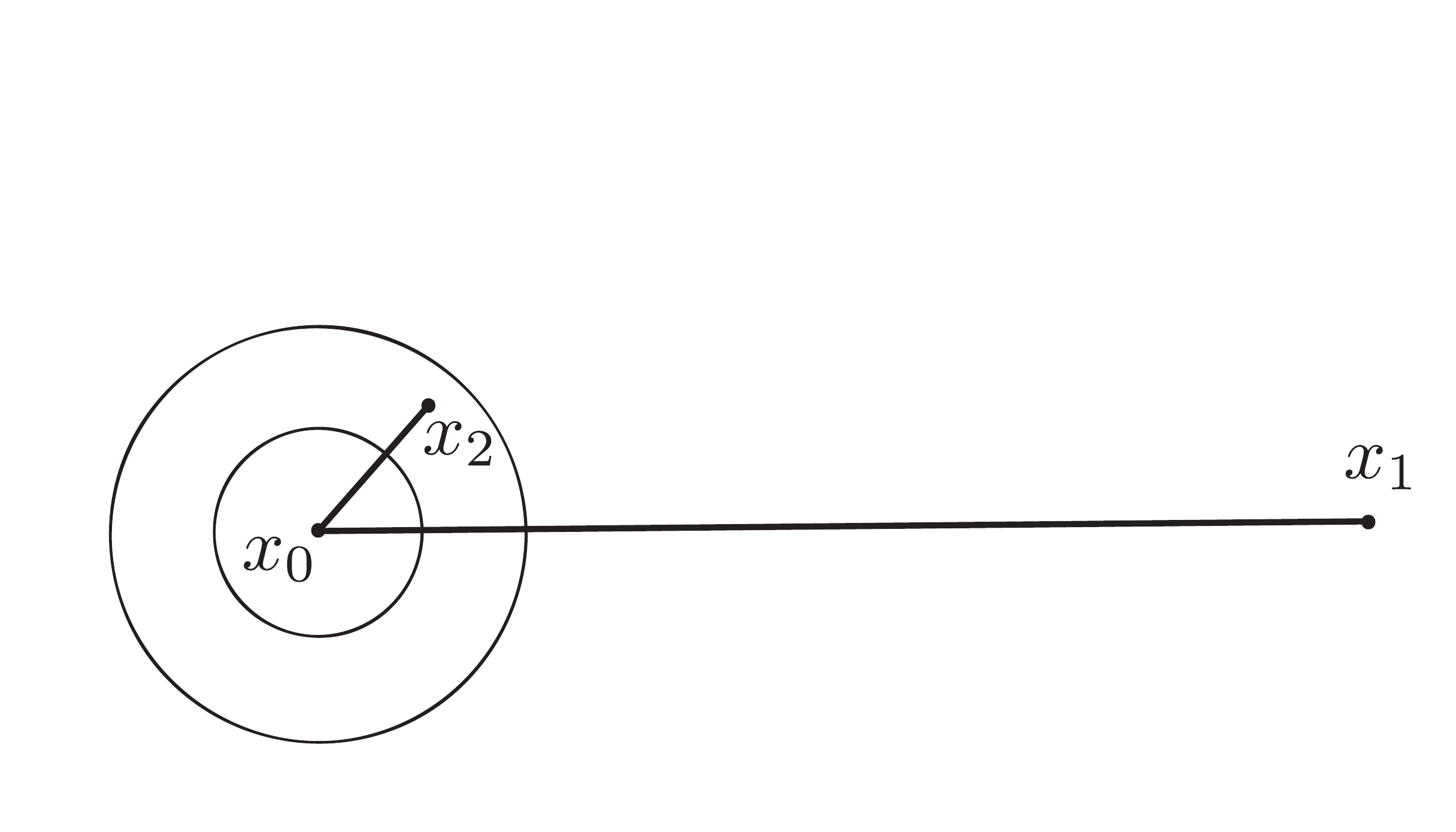}\caption{\label{figure:triangle}{\small
Exemplifying a poorly-scaled triangle (at $x_0$) for $d=1$, $k=3$,
and $p=1$, where the radii of the outer and inner circles are
$\alpha_0^3 \cdot \|x_1-x_0\|$ and $\alpha_0^4 \cdot \|x_1-x_0\|$
respectively. Note that if $x_2$ is relocated outside the two
circles but still closer to $x_0$ than $x_1$, then the modified
triangle is well-scaled. Following the terminology of
Subsection~\ref{section:decomposting} and recalling that $p=1$ and
$k=3$, we note that the edge connecting $x_0$ and $x_1$ is a handle
and the edge connecting $x_2$ and $x_1$ is a tine. Like all
poorly-scaled triangles it is a single-handled rake. }}
\end{center}
\end{figure}

Along these lines, we also denote
the distinguished set of simplices
$$
\widehat{S}= S_{0,3}.
$$
We refer to the elements of $\widehat{S}$ as {\em well-scaled
simplices  at $x_0$}, and the elements of $S \setminus \widehat{S}$
as {\em poorly-scaled simplices at $x_0$}. Quantifiably, $X\in S$ is
well-scaled at $x_0$ if and only if
\be \label{eq:well_scaled_ineq}
\frac{\min\nolimits_{x_0}(X)}{\max\nolimits_{x_0}(X)} >
\alpha_0^3\,. \ee A poorly-scaled triangle is exemplified in
Figure~\ref{figure:triangle}.

For a ball $Q$ in $H$ and $p \in \nats$, we often use localized
versions of the sets  $S$, $\widehat{S}$, and $S_{k,p}$, $k \geq
3$, defined as
\begin{equation}\label{equation:s-k-Q} S(Q)=S\cap Q^{d+2}\,, \ \
\widehat{S}(Q)=\widehat{S}\cap Q^{d+2} \  \text{ and } \
S_{k,p}(Q)=S_{k,p}\cap Q^{d+2} \,.
\end{equation}

In most of the paper it will be sufficient to assume $p=1$, however
in Section~\ref{section:later-integrate} we will need to consider
the case $p=2$, and we thus formulate some corresponding definitions
and propositions in other sections for both $p=1$ and $p=2$. We note
that if $p=1$, then the sets $S_{k,1}$, $k\geq 3$, partition
$S\setminus\widehat{S}$, whereas if $p > 1$, then they cover it.

\subsubsection{Decomposing Simplices in $S_{k,p}$
According to Configuration}\label{section:decomposting} So far we
have decomposed the set of simplices $S$ according to the ratio
between the shortest and longest lengths of edges at $x_0$. Next, we
further decompose simplices according to ratios of lengths of other
edges (at $x_0$) and the length of the largest edge length (at
$x_0$).

We start with some terminology, while fixing arbitrarily $k\geq3$,
$p\in\{1,2\}$, and $X=(x_0,\ldots,x_{d+1})\in S_{k,p}\,$. We say
that an edge connecting $x_0$ and $x_i$, $1 \leq i \leq d+1$, is a
{\em handle} if
\be \label{eq:cond_vertex}
\frac{\|x_0-x_i\|}{\Max_{x_0}(X)}>\alpha_0^k \ee
and a {\em tine} otherwise, i.e.,
\be \nonumber
\alpha_0^{k+p}<\frac{\|x_i-x_0\|}{\max\nolimits_{x_0}(X)}\leq
\alpha_0^k. \ee

This terminology is intended to evoke the image of the
gardening implement known in English as a {\em rake}. We say that
$X$ is a {\em rake}, or a single-handled rake, if it only has one
handle. Similarly, $X$ is an $n$-{\em handled rake} if it has $n$
handles for $1\leq n\leq d$.
 Clearly $X\in S_{k,p}$ has at least one handle (obtaining the
maximal edge length at $x_0$) and at most $d$ of them (excluding the
one of minimal edge length at $x_0$).
 We remark that these notions depend on our fixed choice of
$p$ which will be clear from the context. These notions are
illustrated in Figures~\ref{figure:triangle}
and~\ref{figure:example_eta}.

We partition $S_{k,p}$, $p=1,2$, according to the number of handles
in the elements $X=(x_0,\ldots,x_{d+1})$ at $x_0$.  Formally, for
$1\leq n\leq d$, we define the sets:
\begin{equation}\label{equation:decomposition-union}
S_{k,p}^n=\left\{X =(x_0,\ldots,x_{d+1}) \in S_{k,p}:
\frac{\|x_i-x_0\|} {\max\nolimits_{x_0}(X)}
>
\alpha_0^k\ \textup{ for exactly $n$ vertices}\ x_i
\right\}.
\end{equation}
We note that
\begin{equation}\label{equation:clyde-drexler}
S_{k,p}=\bigcup_{n=1}^d S_{k,p}^n, \ \textup{ and } \ S_{k,p}^n\cap
S_{k,p}^{n'}=\emptyset, \  \textup{ for }  \ 1\leq n\not= n'\leq d.
\end{equation}

\begin{figure}[htbp]
\begin{center}
\includegraphics[height=2.858in,width=5in]
{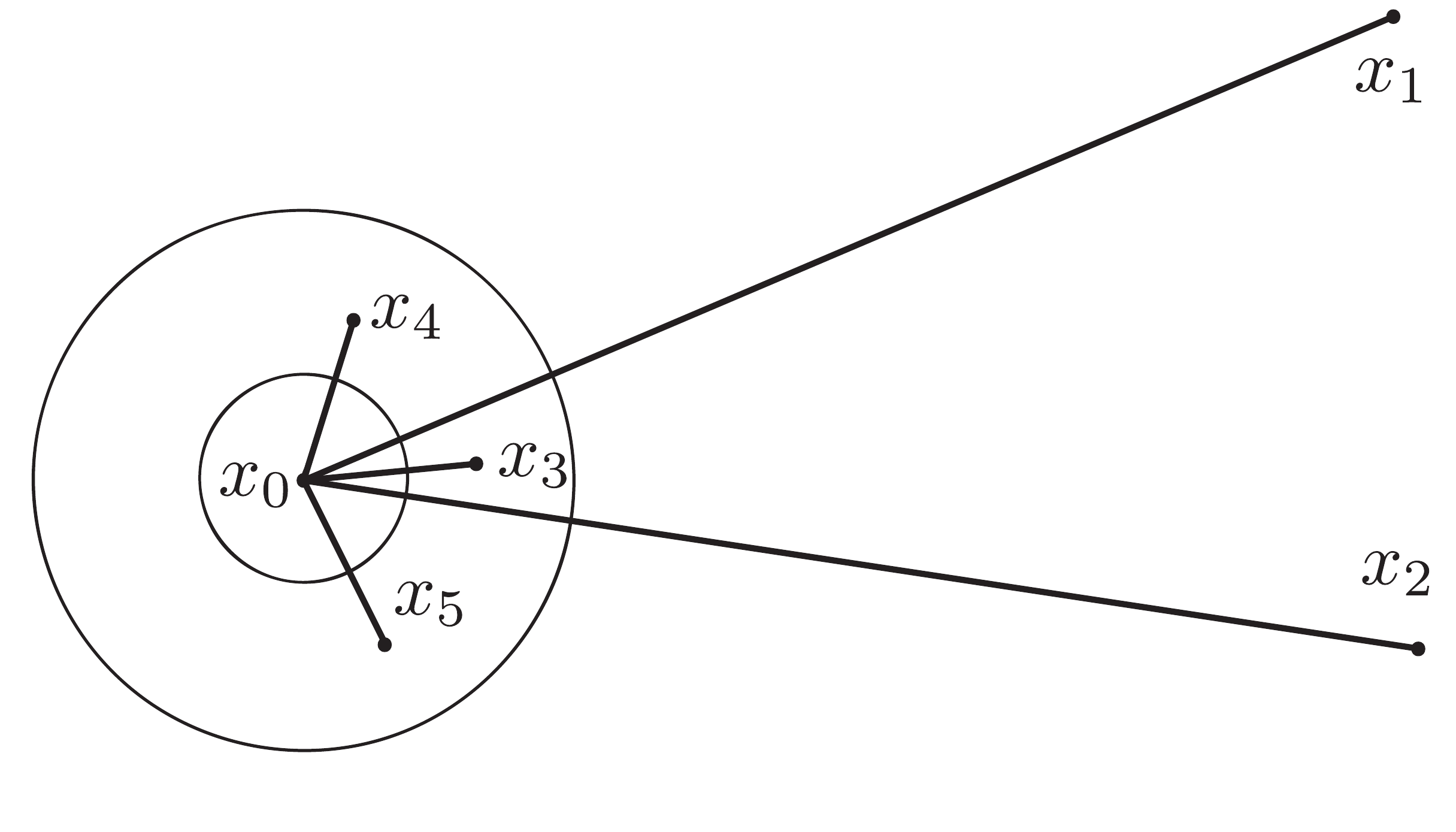}\caption{\label{figure:example_eta}{\small
Example of a simplex $X=(x_0,x_1,x_2,x_2,x_4,x_5)\in \name^2_{3,2}$.
The radii of the outer and inner circles are $\alpha_0^3\|x_1-x_0\|$
and $\alpha_0^5\|x_1-x_0\|$ respectively. }}
\end{center}
\end{figure}

In order to reduce unnecessary information regarding the position of the handles at $x_0$
 we concentrate on the following subset of
$S_{k,p}^n$.
\begin{equation} \nonumber
\name^n_{k,p}=\left\{X\in
S_{k,p}^n:\frac{\left\|(X)_\ell-x_0\right\|}
{\max\nolimits_{x_0}(X)}>\alpha_0^k\ \textup{ for all }\ 1 \leq \ell
\leq n \right\} .
\end{equation}
That is, $\name^n_{k,p}$ is the subset of $S_{k,p}^n$ whose edges
connecting $x_0$ with $x_1,\ldots,x_n$ are handles and whose edges
connecting $x_0$ with $x_{n+1},\ldots,x_{d+1}$ are tines. We
illustrate an element of $\name^2_{3,2}$ where $d=4$ in
Figure~\ref{figure:example_eta}.

Given a ball $Q$ in $H$ we denote the restrictions of the above sets
to $Q^{d+2}$ by $S_{k,p}^n(Q)$ and $\name^n_{k,p}(Q)$ respectively.

%

\subsection{Decomposing the Menger-Type Curvature}\label{section:menger}

We decompose the continuous Menger-type curvature according to the
regions described above (Subsection~\ref{subsection:cat-simplices}).
We start by expressing the discrete and continuous $d$-dimensional
Menger-type curvatures of $X\in S$ and $\mu|_Q$ respectively in
terms of the polar sine as follows:
\begin{equation}\label{equation:def-of-c-d}
c_d(X)=\displaystyle\sqrt{\frac{\sum_{i=0}^{d+1}\pds^2_{x_i}(X)}
{(d+2)\cdot\diam(X)^{d(d+1)}}}\end{equation}
and
\begin{equation}\label{equation:integral-simplification-1}
c_d^2\left(\mu|_Q\right) = \frac{1}{d+2}\,\sum_{i=0}^{d+1}
\int_{Q^{d+2}}\frac{\pds_{x_i}^2(X)}
{\diam(X)^{d(d+1)}}\di\mu^{d+2}(X)
=\int_{Q^{d+2}}\frac{\pds_{x_0}^2(X)}
{\diam(X)^{d(d+1)}}\di\mu^{d+2}(X).
\end{equation}

In order to control the continuous curvature
$c_d^2\left(\mu|_Q\right)$ we break it into ``smaller'' parts. We
first decompose it according to the sets $\{S_{k,1}(Q)\}_{k\geq3}$
and $\widehat{S}(Q)$ of Subsection~\ref{section:decomposting} in the
following way:
\begin{multline}\label{multline:tyleriscrying}
\int_{Q^{d+2}}\frac{\pds^2_{x_0}(X)}
{\diam(X)^{d(d+1)}}\di\mu^{d+2}(X)=\\
\int_{\widehat{S}(Q)}\frac{\pds_{x_0}^2(X)}{\diam(X)^{d(d+1)}}
\di\mu^{d+2}(X) + \sum_{{k \geq 3}}
\int_{S_{k,1}(Q)}\frac{\pds_{x_0}^2(X)}{\diam(X)^{d(d+1)}}
\di\mu^{d+2}(X)\,.
\end{multline}

We further break the terms of the infinite sum in
equation~\eqref{multline:tyleriscrying}  according to the regions
$\name^n_{k,1}(Q)$ of Subsection~\ref{section:decomposting} in the
following way:
\begin{equation}\label{equation:butterfinger}
\int_{S_{k,1}(Q)}\frac{\pds^2_{x_0}(X)}
{\diam(X)^{d(d+1)}}\di\mu^{d+2}(X)=\sum_{n=1}^d\binom{d+1}{n}\,
\int_{\name^n_{k,1}(Q)}\frac{\pds^2_{x_0}(X)}{\diam(X)^{d(d+1)}}\di\mu^{d+2}(X).
\end{equation}

To verify this formula we first decompose $S_{k,1}(Q)$ by
$S_{k,1}^n(Q)$, $n=1,\ldots,d$, according to
equation~\eqref{equation:clyde-drexler}. We then partition each
$S_{k,1}^n(Q)$, $n=1,\ldots,d$, according to the subsets of indices
representing the $n$ handles. Clearly this results in
$\binom{d+1}{n}$ subsets of $S_{k,1}^n(Q)$, where one of these is
associated with $\name^n_{k,p}(Q)$. Now the integral of
$\pds^2_{x_0}(X)/\diam(X)^{d(d+1)}$ over these subsets is the same
(due to the invariance of the polar sine to permutations fixing
$x_0$ and an immediate change of variables), and consequently
equation~\eqref{equation:butterfinger} is established.

Therefore, to control the LHS of
equation~\eqref{multline:tyleriscrying} we only need to concentrate
on the first term of the RHS of
equation~\eqref{multline:tyleriscrying} and the terms of the RHS of
equation~\eqref{equation:butterfinger} for all $k \geq 3$. This is
what we do for the rest of the paper.

\section{Uniform Rectifiability}
\label{section:uniform_rect}
We review here basic notions in the  theory of uniform
rectifiability~\cite{DS91,DS93}. Even though the original theory is
formulated in  finite dimensional Euclidean spaces, the part
presented here generalizes to any separable real Hilbert space.

\subsection{$A_1$ Weights and $\omega$-Regular Surfaces}

Let ${\cal L}_d$ denote the $d$-dimensional Lebesgue measure on
$\mathbb{R}^d$. Given a locally integrable function
$\omega:\mathbb{R}^d\to [0,\infty)$, we induce a measure on Borel
subsets $A$ of $\mathbb{R}^d$ by defining
$$
\omega (A) = \int_{A}\omega (x)\di {\cal L}_d(x) .
$$
We say that  $\omega$ is an $A_1$ weight  if there exists $C\geq1$
such that for any ball $Q$ in $\mathbb{R}^d$,
$$
\frac{\omega (Q)}{|Q|}\leq C\cdot\omega ({x}),\ \ \text{ for } {\cal
L}_d \text{ a.e.}~x\in Q .
$$

We note that the measure induced by $\omega$ is doubling in the
following sense: $\omega (Q) \approx \omega (2 \cdot Q)$ for any
ball $Q$. Consequently, the following function is a quasidistance
(i.e., a symmetric positive definite function satisfying a relaxed
version of the triangle inequality with controlling constant $C\geq
1$ instead of one):
\[
\qdist_\omega({x},{y}) = \sqrt[d]{\omega \left( B \left(
\frac{{x}+{y}}{2},\frac{|{x}-{y}|}{2}\right)\right)},\quad \text{
for all}\ {x},{y}\in \mathbb{R}^d .
\]

Given an $A_1$ weight $\omega$, we define $\omega$-regular
surfaces as follows.
\begin{definition}
\label{def:w_regular_surface} Let $\omega$ be an $A_1$ weight on
$\mathbb{R}^d$. A subset $\Gamma$ of $H$ is called an {\em
$\omega$-regular surface} if there exists a function
$f\colon\mathbb{R}^d\to \HH$ and constants $L$ and $C$ such
that $\Gamma = f(\mathbb{R}^d)$,
\be \label{eq:w_reg_l}
\norm{f({x})-f({y})} \leq L \cdot \qdist_\omega ({x},{y}),\qquad
\textup{ for all } {x},{y} \in \mathbb{R}^d ,
\ee
and
 \be
 \label{eq:flat_d}
 \omega \left(f^{-1} (B({x},r))\right)\leq C\cdot r^d,\qquad
\textup{ for all}\ {x}\in \HH \textup{ and } r>0 . \ee
\end{definition}

\subsection{Two Equivalent Definitions of Uniform Rectifiability}
We provide here two equivalent definitions of uniform
rectifiability. Many other definitions appear
in~\cite{DS91,DS93,MP05,tolsa_plms}.

Given a Borel measure $\mu$ on $H$, we let
\[
\widehat{\HH} =
\begin{cases} \HH \times \reals,\ &\text{if}\
\dim (H)< 2\cdot d;\\
\HH, &\text{otherwise},
\end{cases}
\]and we define the induced measure $\widehat{\mu}$ on
$\widehat{H}$ to be $\widehat{\mu}(A)=\mu(A\cap H),\textup{ for all
Borel sets }A\subseteq\widehat{H}.$

We define $d$-dimensional uniformly rectifiable measures  as
follows:
\begin{definition}
A Borel measure $\mu$ on $\HH$ is said to be  $d$-dimensional
uniformly rectifiable if it is $d$-regular and there exist an $A_1$
weight $\omega$ on $\mathbb{R}^d$ along with an $\omega$-regular
surface $\Gamma\subseteq\widehat{\HH}$ such that $\widehat{\mu}
\left(\widehat{H}\setminus \Gamma\right) = 0$.
\end{definition}

David and Semmes~\cite{DS91,DS93} have shown that the Jones-type
flatness of equation~\eqref{eq:def_jones_flat} can be used to
quantify and thus redefine uniform rectifiability as follows.
\begin{theorem}
\label{theorem_DS} A $d$-regular measure $\mu$ on $H$ is uniformly
rectifiable if and only if there exists a constant $C=C(d,C_{\mu})$
such that
$$
J_d(\mu|_Q)\leq C \cdot \mu (Q) \ \text{ for any ball } Q \text{ in
} H.
$$
\end{theorem}
We note that Theorem~\ref{theorem:uniform-rect-menger} is an
immediate consequence of Theorems~\ref{theorem:upper-main}
and~\ref{theorem_DS}.

\subsection{Multiscale Resolutions and Modified Jones-type Flatness}
\label{section:multiscale-partitions}
Multiscale decomposition of $\Supp$ are common in the theory of
uniform rectifiability~\cite{D91_book, raanan_hlibert}, and they are
often used to compress information from all scales and locations.
Here we also use them to construct covers and partitions of
$S\cap[\Supp]^{d+2}$ and consequently control it by a ``compressed''
Jones-type flatness.

Following the spirit of~\cite{D91_book, raanan_hlibert} we cover
$\Supp$ by balls $\{\mathcal{B}_n\}_{n \in\,\mathbb{Z}}$ which
correspond to the length scales $\{\alpha_0^n\}_{n \in\,
{\mathbb{Z}}}$. We then use them to construct a corresponding
sequence of partitions, $\{\mathcal{P}_n\}_{n \in\,\mathbb{Z}}$ of
$\Supp$.

We say that a collection of points $E_n\subseteq\Supp$ is an {\em
$n$-net} for $\Supp$ if
\begin{enumerate}
\item $\|x-y\|>\alpha_0^n,$ for
all $x$ and $y$ in $E_n$. \item $\displaystyle \Supp \subseteq
\bigcup_{x\in\, E_n}B(x,\alpha_0^n).$
\end{enumerate}
For each $n \in {\mathbb Z}$, we arbitrarily form  an $n$-net,
$E_n$, and among all balls in the family $\{
B(x,4\cdot\alpha_0^n)\}_{x \in\, E_n}$ we fix a subfamily
$\mathcal{B}_n$ such that $\frac{1}{4}\cdot\mathcal{B}_n$ is
maximally mutually disjoint. Since $H$ is separable, $\Supp$ is
separable and $\mathcal{B}_n$ is countable.  Furthermore, we note
that $\mathcal{B}_n$ covers $\Supp$.

We index the elements of ${\cal B}_n$ by $\Lambda_n =
\{1,2,\ldots\}$, which is either finite or $\mathbb{N}$, so that
\begin{equation}\label{equation:lambda-b-good}
{\cal B}_n = \left\{{B}_{n,j}\right\}_{j\in\Lambda_n}.\end{equation}
We define the corresponding {\it multiresolution family} for $\Supp$
to be
\begin{equation}\label{equation:def-family-resolution}
\mathcal{D}=\bigcup_{n \in\,\mathbb{Z}}\mathcal{B}_n.
\end{equation}

These resolutions can be replaced by multiscale partitions of
$\Supp$ in the following manner (see proof in
Appendix~\ref{app:resolution-lemma}).
\begin{lemma}\label{lemma:set-partition}
For any $n \in\,\mathbb{Z}$ there exists a partition of $\Supp$,
$\mathcal{P}_n=\{P_{n,j}\}_{j\in\Lambda_n}$, such that for any
$j\in\Lambda_n$ there exists a unique $B_{n,j}\in {\cal B}_n$ with
$$
\Supp\,\bigcap\,\frac{1}{4} \cdot B_{n,j}\,\subseteq\,
P_{n,j}\,\subseteq\,\Supp\,\bigcap\,\frac{3}{4} \cdot B_{n,j}\, .
$$
\end{lemma}

We typically work with localized resolutions which we define as
follows. For $Q$ a ball in $H$, we let $m(Q)$ be the smallest
integer $m$ such that $\alpha_0^m\leq\diam(Q)$, i.e.,
\begin{equation}\label{equation:integer-for-Q}
m(Q)=\left\lceil\frac{\ln(\diam(Q))}{\ln(\alpha_0)}\right\rceil,\end{equation}
For $n\geq m(Q)$ we define
\begin{equation}\label{equation:localized-restitutionolution}\mathcal{B}_n(Q)=\left\{B_{n,j}\in\mathcal{B}_n:B_{n,j}\cap
Q\not=\emptyset\right\},\end{equation}
and form the {\it local multiresolution family} as follows
\begin{equation}\label{equation:localized-resolution}
\mathcal{D}(Q)=\bigcup_{ n\geq m(Q)}\mathcal{B}_n(Q)
\,.\end{equation}
We also define the set of indices (possibly empty)
$$\displaystyle\Lambda_n(Q)=\{j\in\Lambda_n:P_{n,j}\cap
Q\not=\emptyset\}.$$
%


\subsubsection{Jones-type Flatness via Multiscale Resolutions}

For a ball $Q$ in $H$, and the local multiresolution
$\mathcal{D}(Q)$, we define the corresponding local Jones-type
$d$-flatness as follows:
$$
 J^{\mathcal{D}}_{d}(\mu|_Q) = \sum_{B \in{\cal D}(Q)}\beta^2_2(B) \cdot \mu
(B).
$$
Both quantities of Jones-type flatness, $J_d$ and
$J^{\mathcal{D}}_{d}$, are comparable.  We will only use the
following part of the comparability whose proof practically follows
the same arguments of~\cite[Lemma~3.2]{raanan_hlibert}:
\begin{proposition}\label{proposition:discretization}
There exists a constant $ C_3=C_3(d,C_{\mu})$ such that for any
multiresolution family $\mathcal{D}$ on $\Supp$ :
\be\label{eq:discretization}  J^{\mathcal{D}}_d(\mu|_Q)\leq C_3\cdot
J_d(\mu|_{6\cdot Q}) \ \textup{ for any ball } Q \textup{ in } H\,.
\ee

\end{proposition}

\section{Proof of
Proposition~\ref{proposition:upper-bound-big-scale}}\label{section:belowed}We
first note that  $U_{\lambda}(B(x,t))$ is invariant under any
permutation of the coordinates. Thus, by the same argument producing
equation~(\ref{equation:integral-simplification-1}) we have the
equality
\begin{equation}\label{equation:above-1}\int_{U_{\lambda}(B(x,t))}c_d^2(X)\di\mu^{d+2}(X)=\int_{U_{\lambda}(B(x,t))}
\frac{\pds_{x_0}^2(X)}{\diam(X)^{d(d+1)}}\di\mu^{d+2}(X).\end{equation}
Furthermore,  $\diam(X)\geq\lambda\cdot t$ and
$\SCale_{x_0}(X)\geq\lambda$ for all $ X\in U_{\lambda}(B(x,t)).$
Hence, applying Proposition~\ref{proposition:psin-bound-deviations}
to the RHS of equation~(\ref{equation:above-1}) we get that for any
$d$-plane $L$
\begin{multline}\label{equation:above-2}\int_{U_{\lambda}(B(x,t))}c_d^2(X)\di\mu^{d+2}(X)
\leq\frac{2\cdot(d+1)^2\cdot(d+2)^2}{\lambda^{d(d+1)+4}}\int_{U_{\lambda}(B(x,t))}\frac{D_2^2(X,L)}{t^2}\,
\frac{\di\mu^{d+2}(X)}{t^{d(d+1)}} \
=\\\frac{8\cdot(d+1)^2\cdot(d+2)^2}{\lambda^{d(d+1)+4}}\,\sum_{i=0}^{d+1}
\int_{U_{\lambda}(B(x,t))}\left(\frac{\dist(x_i,L)}{2\cdot
t}\right)^2 \frac{\di\mu^{d+2}(X)}{t^{d(d+1)}}\,.\end{multline}For
$P_i$, $0 \leq i \leq d+1$, the projection of $H^{d+2}$ onto its
$i^{\text{th}}$ coordinate, we have the inclusion
\begin{equation}\label{equation:trivially-coffee}P_i(U_{\lambda}(B(x,t)))\subseteq B(x,t).\end{equation}
Hence, fixing $0\leq i\leq d+1$ and  applying
equation~(\ref{equation:trivially-coffee}) and Fubini's Theorem to
the corresponding term on the RHS of
equation~(\ref{equation:above-2}), we get the
inequality\begin{multline}\label{equation:above-3}\int_{U_{\lambda}(B(x,t))}
\left(\frac{\dist(x_i,L)}{2\cdot t}\right)^2
\frac{\di\mu^{d+2}(X)}{t^{d(d+1)}}\leq
\left(\frac{\mu(B(x,t))}{t^d}\right)^{d+1}
\int_{B(x,t)}\left(\frac{\dist(x_i,L)}{2\cdot
t}\right)^2\di\mu(x_i)=\\\left(\frac{\mu(B(x,t))}{t^d}\right)^{d+1}
\cdot\beta_2^2(x,t,L)\cdot\mu(B(x,t)).\end{multline}

Combining equations~(\ref{equation:above-2})
and~(\ref{equation:above-3}), while summing the RHS of
equation~(\ref{equation:above-3}) over $0\leq i\leq d+1$ as well as
applying the $d$-regularity of $\mu$, we obtain the following bound
on the LHS of equation~\eqref{equation:upper-big-scale}:
\begin{equation}\label{equation:above-4}\int_{U_{\lambda}(B(x,t))}c_d^2(X)\di\mu^{d+2}(X)
\leq\frac{8\cdot(d+1)^2\cdot(d+2)^3}{\lambda^{d(d+1)+4}}\cdot
C_{\mu}^{d+1}\cdot\beta_2^2(x,t,L)\cdot\mu(B(x,t)).\end{equation}
Since $L$ is  arbitrary, taking the infimum over all $L$ on the RHS
of equation~(\ref{equation:above-4}) proves the proposition.\qed

\section{Reduction of
Theorem~\ref{theorem:upper-main}}\label{section:major-reduction}
We reduce Theorem~\ref{theorem:upper-main} by applying the following
decomposition of the multivariate integral
$c_d^2\left(\mu|_Q\right)$, obtained by combining
equations~\eqref{multline:tyleriscrying}
and~\eqref{equation:butterfinger}:
\begin{multline}\label{multline:decomposed}
c^2_d(\mu|_Q) =
\int_{\widehat{S}(Q)}\frac{\pds_{x_0}^2(X)}{\diam(X)^{d(d+1)}}
\di\mu^{d+2}(X) + \\
\sum_{{k \geq 3}} \sum_{n=1}^d\binom{d+1}{n}\,
\int_{\name_{k,1}^n(Q)}\frac{\pds^2_{x_0}(X)}{\diam(X)^{d(d+1)}}\di\mu^{d+2}(X).
\end{multline}

We thus prove Theorem~\ref{theorem:upper-main} by establishing the
following three propositions:
\begin{proposition}\label{proposition:journe}There exists a constant $C_4=C_4(d,C_{\mu})$
such that
\begin{equation}\label{equation:blow-1}\int_{\widehat{S}(Q)}\frac{\pds_{x_0}^2(X)}{\diam(X)^{d(d+1)}}
\di\mu^{d+2}(X)\leq\frac{C_4}{\alpha_0^6} \cdot
J^{\mathcal{D}}_d\left(\mu|_Q\right)\,.\end{equation}\end{proposition}

\begin{proposition}\label{proposition:poorly-scaled-integration-d}
There exists a constant $C_5 = C_5(d,C_{\mu})$ such that for any
ball $Q$ in $H$
%
$$
\int_{\name_{k,1}^1(Q)}\frac{\pds_{x_0}^2(X)}{\diam(X)^{d(d+1)}}\di\mu^{d+2}\leq
C_5\cdot (k\cdot d+1)\cdot\left(\alpha_0^d\cdot
C_{\mathrm{p}}^2\right)^{k\cdot d}\cdot J^{\mathcal{D}}_d(\mu|_Q)\,.
$$
\end{proposition}

\begin{proposition}\label{proposition:poorly-scaled-integration}
If  $1 < n \leq d$, then there exists a constant $C_6 =
C_6(d,C_{\mu})$ such that for any ball $Q$ in $H$
$$
\int_{\name_{k,1}^n(Q)}\frac{\pds_{x_0}^2(X)}
{\diam(X)^{d(d+1)}}\di\mu^{d+2}\leq C_6\cdot(k\cdot
d+1)^3\cdot\left(\alpha_0\cdot C_{\mathrm{p}}^2\right)^{k\cdot
d}\cdot  J^{\mathcal{D}}_d(\mu|_Q)\, .
$$
\end{proposition}

Proposition~\ref{proposition:journe} reduces the integration of the
Menger-type curvature in Theorem~\ref{theorem:upper-main} to
well-scaled simplices. We prove it in Section~\ref{section:journe}
by straightforward decomposition of the multivariate integral which
we refer to as geometric multipoles.

Propositions~\ref{proposition:poorly-scaled-integration-d}
and~\ref{proposition:poorly-scaled-integration} reduce the
integration of the Menger-type curvature in
Theorem~\ref{theorem:upper-main} to poorly-scaled simplices, which
are single-handled rakes in the former proposition and multi-handled
rakes in the latter one. We prove those propositions in
Sections~\ref{section:proof-prop-poorly-d}
and~\ref{section:later-integrate} respectively. Unlike
Proposition~\ref{proposition:journe}, the method of geometric
multipoles is not sufficient for their proof. Our basic idea is in
the spirit of the proof of~\cite{Jo_unpublished} (i.e.,
\cite[Theorem~31]{pajot_book}), and trades any poorly-scaled simplex
for a predictable sequence of well-scaled simplices satisfying an
extended version of the ``triangle inequality'' for the polar sine
(Proposition~\ref{proposition:concentration-inequality-1}). We then
apply the method of geometric multipoles to the well-scaled
simplices of the sequence.

\section{Proof of Proposition~\ref{proposition:journe} via Geometric Multipoles}
\label{section:journe}

Our proof of Proposition~\ref{proposition:journe} generates
approximate decompositions of the Menger-type curvature of $\mu$
according to goodness of approximations by $d$-planes at different
scales and locations. We refer to this strategy as {\em geometric
multipoles} and see it as a geometric analog of the decomposition of
special potentials by near-field interactions at different
locations, as applied in the fast multipoles algorithm~\cite{GR87}.
Unlike fast multipoles, which considers interactions between pairs
of points, geometric multipoles takes into account simultaneous
interactions between $d+2$ points. While fast multipoles  neglects
terms of distant interactions, $d$-dimensional geometric multipoles
may neglect locations and scales well-approximated by $d$-planes.

We first break the integral on the LHS of
equation~(\ref{equation:blow-1}) into a sum of integrals reflecting
different scales and locations in $Q\cap\Supp$. Then we control each
such integral by $\beta_2^2(B) \cdot \mu(B)$ for a unique $B \in
{\cal D}(Q)$.

For fixed $m\geq m(Q)$ (see
equation~(\ref{equation:integer-for-Q})), we define
\begin{equation}\widehat{S}(m)=
\left\{X\in \widehat{S}:\max\nolimits_{x_0}(X)\in
(\alpha_0^{m+1},\alpha_0^m]\right\},
\end{equation}and we let $\widehat{S}(m)(Q)$ denote the restriction of
$\widehat{S}(m)$ to the set $Q^{d+2}$. The argument $m$ indicates
the overall length scale of the elements and the subscript indicates
the relative scaling between the edges at $x_0$.  We note that the
family $\left\{\widehat{S}(m)(Q)\right\}_{m\in\,\mathbb{Z}}$
partitions $\widehat{S}(Q)$, and thus
\begin{equation}\label{equation:d-s-hat-sum}
\int_{\widehat{S}(
Q)}\frac{\pds^2_{x_0}(X)}{\diam(X)^{d(d+1)}}\di\mu^{d+2}(X)=\sum_{m\geq
m(Q)}
\int_{\widehat{S}(m)(Q)}\frac{\pds^2_{x_0}(X)}{\diam(X)^{d(d+1)}}\di\mu^{d+2}(X).
\end{equation}

Next, fixing the length scale determined by $m$, we partition each
$\widehat{S}(m)(Q)$ according to location in $\Supp$ determined by
the partition $\mathcal{P}_m$. For fixed $m\geq m(Q)$,
$j\in\Lambda_m$,  and  $P_{m,j}$ as in
Lemma~(\ref{lemma:set-partition}), let
\begin{equation}
\widehat{P}_{m,j}=\left\{X\in \widehat{S}(m)(Q):x_0\in
P_{m,j}\right\}.
\end{equation}
We thus obtain the inequality
$$
\int_{\widehat{S}(Q)}\frac{\pds^2_{x_0}(X)}{\diam(X)^{d(d+1)}}\di\mu^{d+2}(X)\leq\sum_{m\geq
m(Q)}\sum_{j\in\,\Lambda_m(Q)}\int_{\widehat{P}_{m,j}}\frac{\pds^2_{x_0}(X)}{\diam(X)^{d(d+1)}}\di\mu^{d+2}(X)\,.
$$

Fixing $m$ , $j\in\Lambda_m(Q)$, and $B_{m,j}\in{\cal B}_m(Q)$ such
that $\frac{1}{4}\cdot B_{m,j}\cap\Supp\subseteq
P_{m,j}\subseteq\frac{3}{4}\cdot B_{m,j}$ (see
Lemma~\ref{lemma:set-partition}) we will show that
\be \label{equation:well-scaled-constant}
\displaystyle\int_{\widehat{P}_{m,j}}\frac{\pds^2_{x_0}(X)}{\diam(X)^{d(d+1)}}\di\mu^{d+2}(X)
\leq
\frac{C_4}{\alpha_0^6} \cdot \beta^2_2(B_{m,j})\cdot\mu
(B_{m,j})\,.\ee
Summing over $j\in\Lambda_m(Q)$ and  $m\geq m(Q)$ will then conclude
the proposition.

We establish equation~\eqref{equation:well-scaled-constant} by
following the basic argument behind
Proposition~\ref{proposition:upper-bound-big-scale}. We first note
that for all $ X\in\widehat{P}_{m,j}\subseteq \widehat{S}(m)(Q)$
\begin{equation}\label{equation:goodly-good-tyler} \frac{\alpha_0}{8}\cdot\diam
\left(B_{m,j}\right)=\alpha_0^{m+1}<\max\nolimits_{x_0}(X)\leq\diam(X)\,.
\end{equation}
Then, fixing an arbitrary $d$-plane $L$ in $H$, and combining
Proposition~\ref{proposition:psin-bound-deviations} with
equation~(\ref{equation:goodly-good-tyler}), we obtain the following
inequality for all $X\in\widehat{P}_{m,j}$
\begin{equation}\label{equation:well-scaled-linear-bound}
\frac{\pds^2_{x_0}(X)}{\diam(X)^{d(d+1)}} \leq \frac{2\cdot 8^2
\cdot (d+1)^2\cdot
(d+2)^2}{\alpha^6\cdot\alpha_0^2}\cdot\frac{D^2_2\left(X,L\right)}{\diam^2\left(B_{m,j}\right)}\cdot\frac{1}
{\left(\alpha_0^{d(m+1)}\right)^{d+1}}.
\end{equation}
Furthermore, if $X = (x_0,\ldots,x_{d+1}) \in \widehat{P}_{m,j}$,
then $x_0 \in \frac{3}{4} \cdot B_{m,j}$, and $\|x_i-x_0\| \leq
\alpha_0^m = \diam(B_{m,j})/4$ for all $1\leq i \leq d+1$.
Consequently, $x_i \in B_{m,j}$ for all $0 \leq i \leq d+1$, and
thus
\begin{equation}\label{equation:well-scaled-nnclusion}
\widehat{P}_{m,j}\subseteq \left(B_{m,j}\right)^{d+2}.
\end{equation}
Combining equations~(\ref{equation:well-scaled-linear-bound})
and~(\ref{equation:well-scaled-nnclusion}) we obtain the bound
\begin{multline}\label{equation:individual-bound}
\int_{\widehat{P}_{m,j}}\frac{\pds_{x_0}^2(X)}{\diam(X)^{d(d+1)}}\di\mu^{d+2}(X)
\leq\\ \frac{2\cdot8^2\cdot(d+1)^2\cdot
(d+2)^2}{\alpha^6\cdot\alpha_0^2}\int_{\left( B_{m,j}\right)^{d+2}}
\frac{D^2_2\left(X,L\right)}{\diam^2\left(B_{m,j}\right)} \,
\frac{\fdi\mu^{d+2}(X)}{\left(\alpha_0^{d(m+1)}\right)^{d+1}}.
\end{multline}
Finally, applying the same types of computations to the RHS of
equation~\eqref{equation:individual-bound} that led to
equations~(\ref{equation:above-3}) and~(\ref{equation:above-4})  we
obtain equation~(\ref{equation:well-scaled-constant}) and hence the
proposition.\qed

\section{Proof of
Proposition~\ref{proposition:poorly-scaled-integration-d} via
Multiscale Inequalities} \label{section:proof-prop-poorly-d}

Here we develop a multiscale inequality for the polar sine that
results in a proof of
Proposition~\ref{proposition:poorly-scaled-integration-d}. The basic
idea is to take an $X\in\name_{k,1}$ and carefully construct a
sequence of simplices, $\{X_n\}_{n=1}^{n(X)}$, well-scaled at $x_0$
and depending on $X$ such that
\begin{equation}\label{firstofmany}\pds_{x_0}(X)\lesssim \sum_{j=1}^{n(X)}\pds_{x_0}(X_j),\end{equation}
where neither the length of the sequence, $n(X)$, nor the size of
the comparability constant in the above inequality (also depending
on $n(X)$) are ``too large''.

Such a sequence will have zeroth coordinate $x_0=(X)_0$, and overall
length scales progressing geometrically from $\min_{x_0}(X)$ to
$\max_{x_0}(X)$.  In this way the length of the sequence, $n(X)$,
will be such that $$n(X)\lesssim
\log\left(\frac{\max_{x_0}(X)}{\min_{x_0}(X)}\right)\lesssim k.$$

As we mentioned previously, a similar approach  used in the
one-dimensional case to control $c^2_M(\mu)$ by the $\beta_\infty$
numbers~\cite{Jo_unpublished,pajot_book} was the inspiration for our
approach. However, generalizing it to higher-dimensions, as well as
the $\beta_2$ numbers, required us to go substantially beyond what
was present for $d=1$ in a few ways as shown throughout the
arguments of the proof.

In order to construct these sequences and clearly formulate the
inequalities  we require some development in terms of ideas and
notation. Subsection~\ref{subsection:multiscale} develops the notion
of well-scaled sequences and a related discrete multiscale
inequality for the polar sine.
Subsection~\ref{subsection:background-integration} turns this
inequality into a multiscale integral inequality. Finally, in
Subsection~\ref{subsection:integration-d} we prove
Proposition~\ref{proposition:poorly-scaled-integration-d}.
Throughout this section we arbitrarily fix $k \geq 3$ and take $p=1$
or $p=2$ (depending on the context), where $p=2$ is only needed for
definitions or calculations that will be used in
Section~\ref{section:later-integrate}.

\subsection{From Rakes to Well-Scaled Sequences}\label{subsection:multiscale}We recall that $\name_{k,p}^1$
is the set of rakes whose single handles are obtained at their first
coordinate, such that
$\alpha_0^{k+p}<\SCale_{x_0}(X)\leq\alpha_0^k$.

We explain here how to ``decompose'' elements of $\name_{k,p}^1$
into a sequence of well-scaled simplices satisfying an inequality of
the form in equation~\eqref{firstofmany}.  That is, we  take an
element with one very large edge at $x_0$ and swap it for a
predictable sequence of elements, each with all edges at $x_0$
comparable, such that the sum of the polar sines controls that of
the original simplex.  This involves (regrettably) some technical
definitions and corresponding notation, the first of which is the
following shorthand notation for annuli.

If $n \in \ints$ and $\alpha_0$ is the fixed constant of
equation~\eqref{equation:alpha-loose}, we use the notation
\begin{equation}
\label{eq:def-annulus_n} A_n(x,r)=B(x,\alpha_0^n\cdot r) \setminus
B(x,\alpha_0^{n+1} \cdot r)\,.
\end{equation}

\subsubsection{Well-Scaled Pieces and Augmented Elements}
\label{section:well-scaled}
We define a {\em well-scaled piece} for $X\in \name_{k,p}^1$ to be a
($k\cdot d$)-tuple of the form
%
\begin{equation}\label{equation:def-well-scaled}Y_{X}=\left(y_1,\ldots,y_{k\cdot
d}\right)\in\,\prod_{q=1}^{k\cdot
d}A_{k-\left\lceil\frac{q}{d}\right\rceil}
(x_0,\max\nolimits_{x_0}\left(X\right))\,.\end{equation}
The coordinates of  $Y_X$ are  grouped into $k$ distinct clusters of
$d$ points, with each individual cluster lying in a distinct annulus
centered at $x_0$ (see Figure~\ref{figure:single}).

\begin{figure}[htbp]
     \centering
     \subfigure[simplex $X$ and $Y_X$]{\label{figure:single}
          \includegraphics[height=1.9in,width=1.9in]{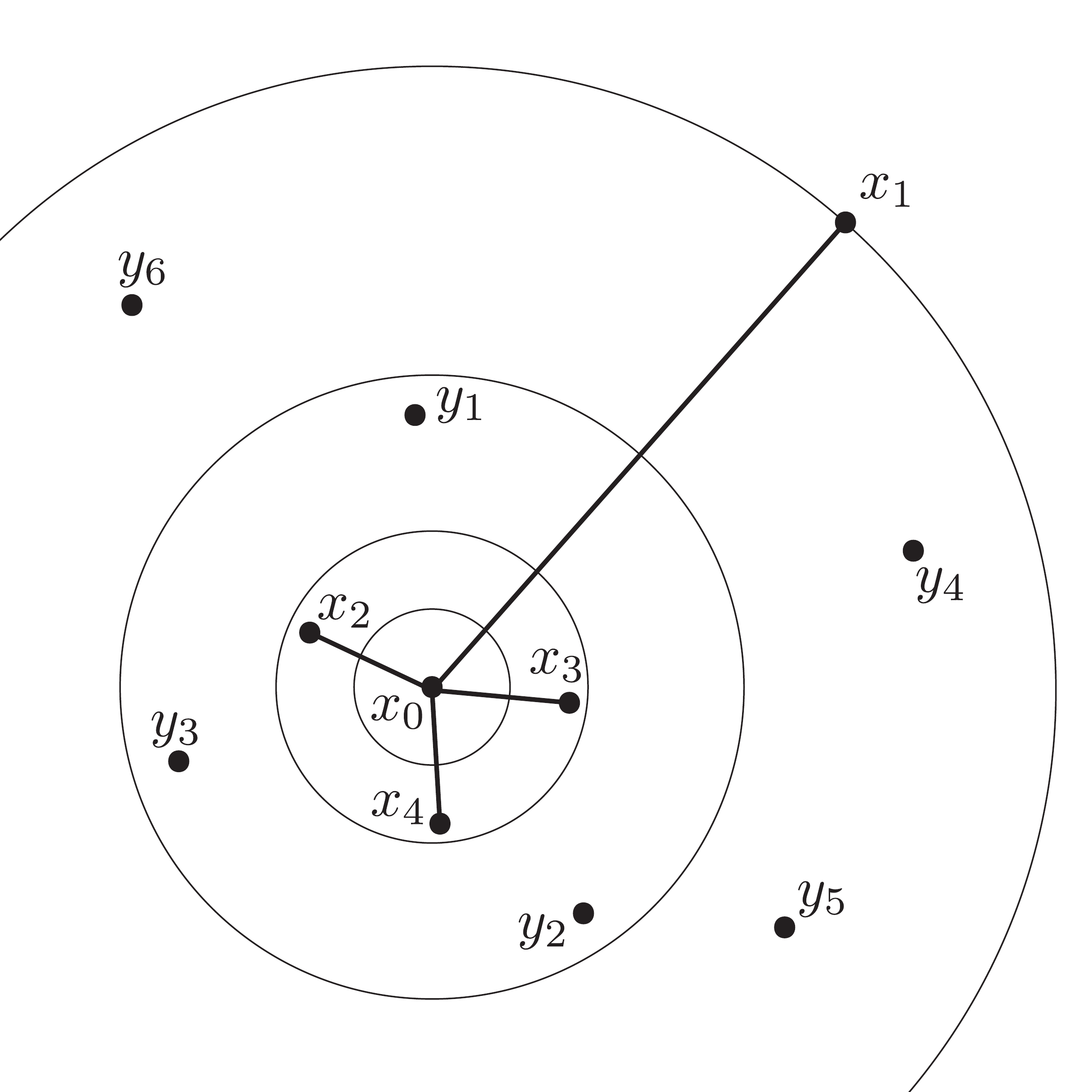}}
     \hspace{.1in}
     \subfigure[auxiliary simplex $\widetilde{X}_1$]{\label{figure:auxiliary}
           \includegraphics[height=1.9in,width=1.9in]{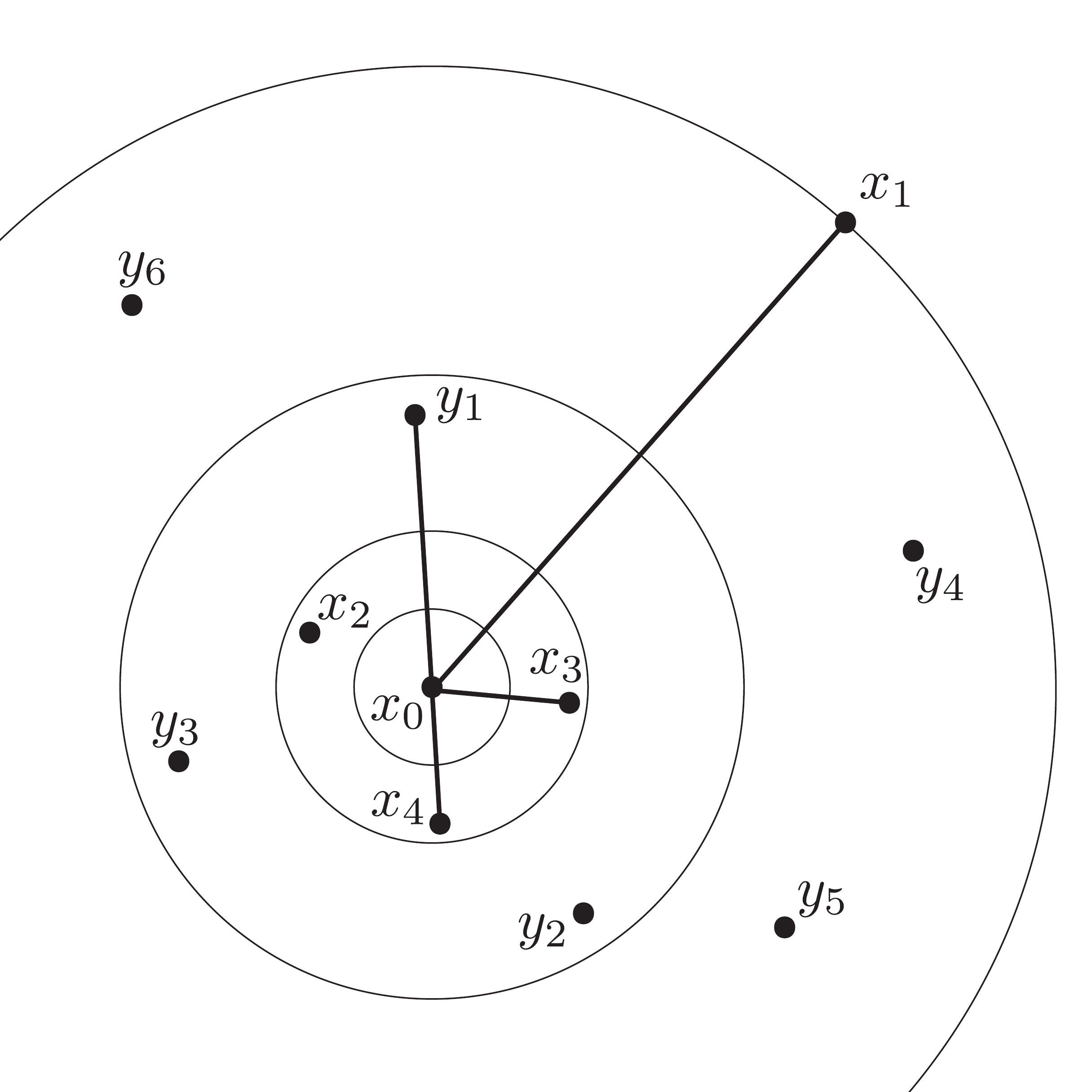}}
     \hspace{.1in}
     \subfigure[well-scaled simplex $X_1$]{\label{figure:shortscale}
           \includegraphics[height=1.9in,width=1.9in]{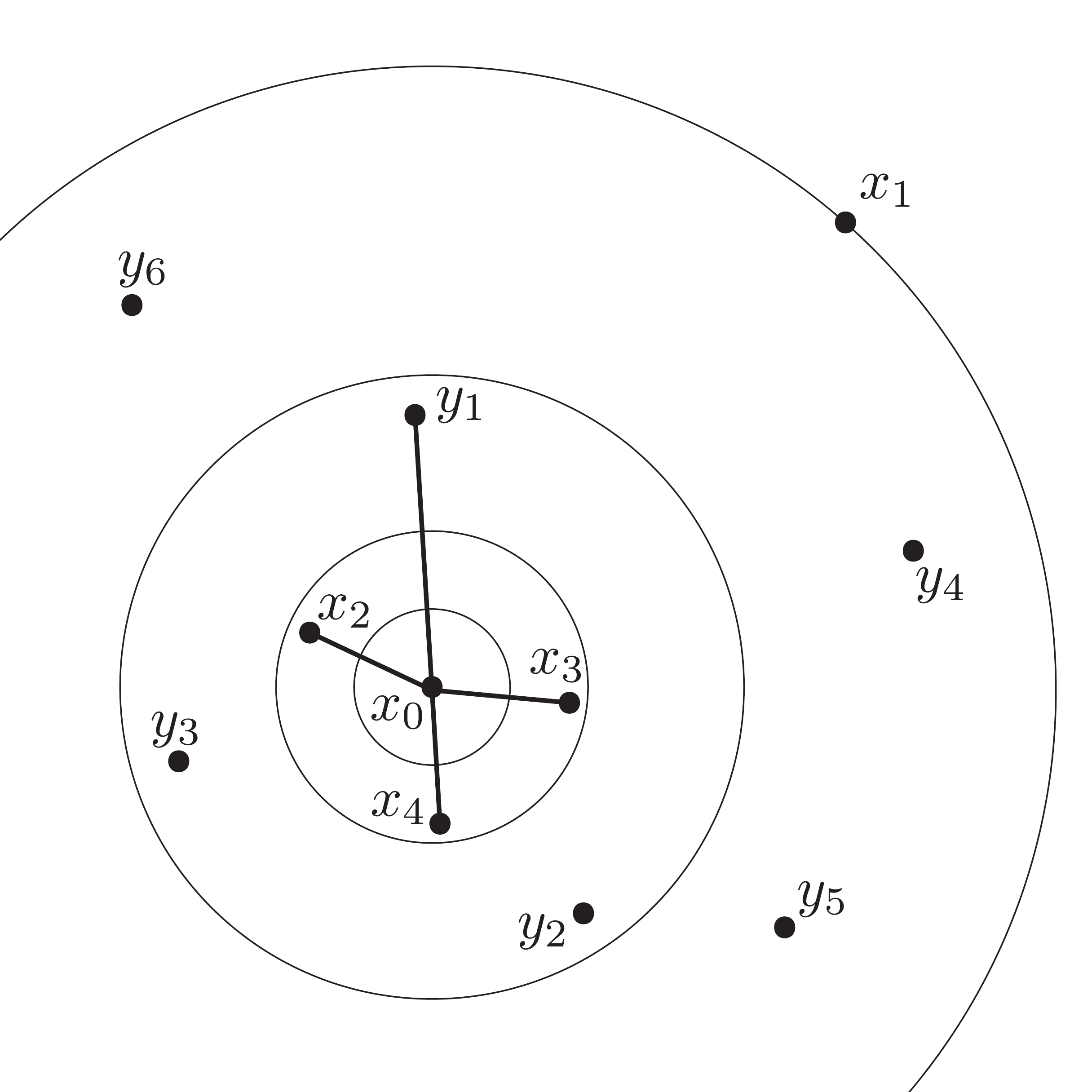}}
     \vspace{.3in}
     \subfigure[well-scaled simplex $X_2$]{\label{figure:shortscale1}
           \includegraphics[height=1.9in,width=1.9in]{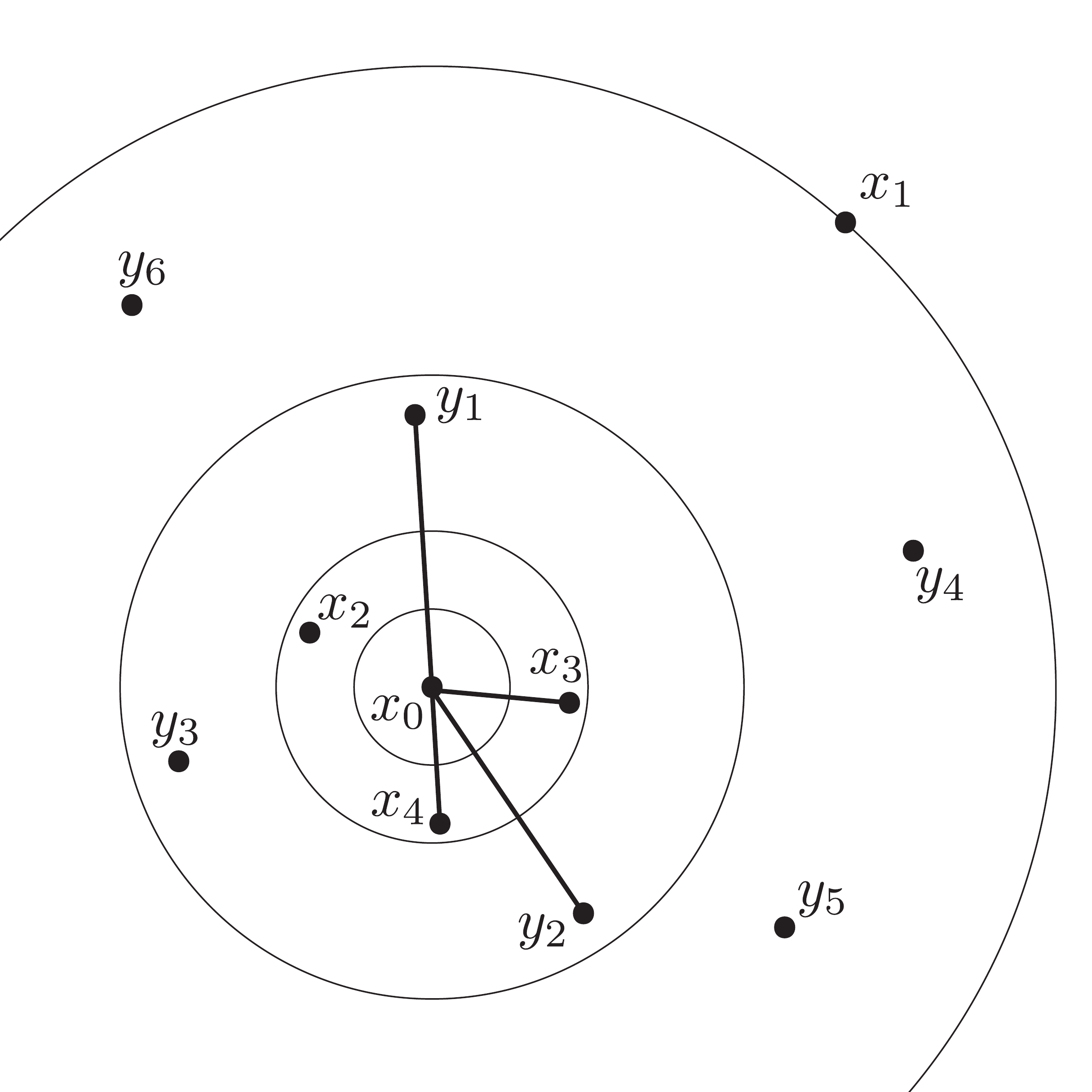}}
     \hspace{.1in}
     \subfigure[well-scaled simplex $X_4$]{\label{figure:shortscale2}
           \includegraphics[height=1.9in,width=1.9in]{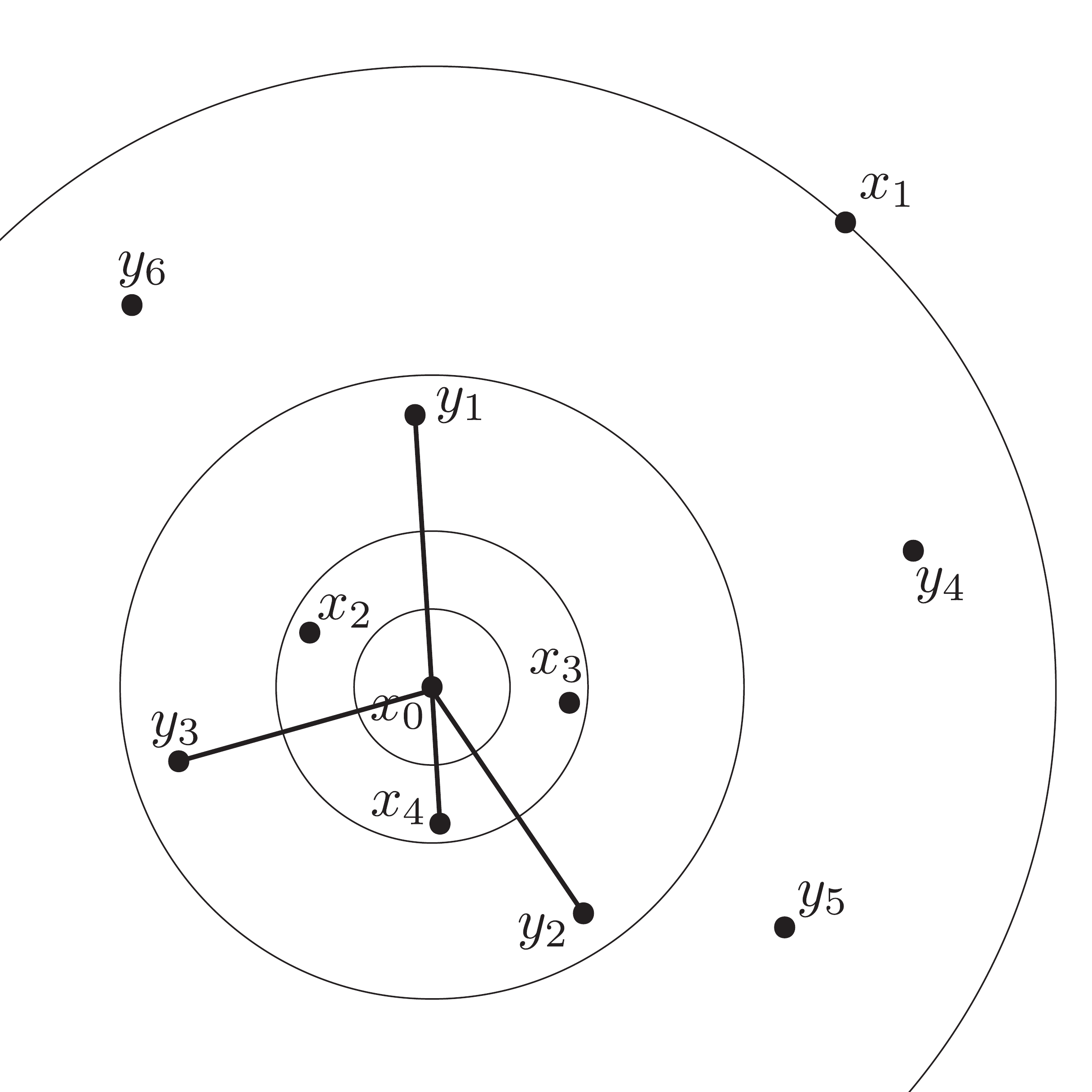}}
     \hspace{.1in}
     \subfigure[well-scaled simplex $X_7$]{\label{figure:shortscale3}
           \includegraphics[height=1.9in,width=1.9in]{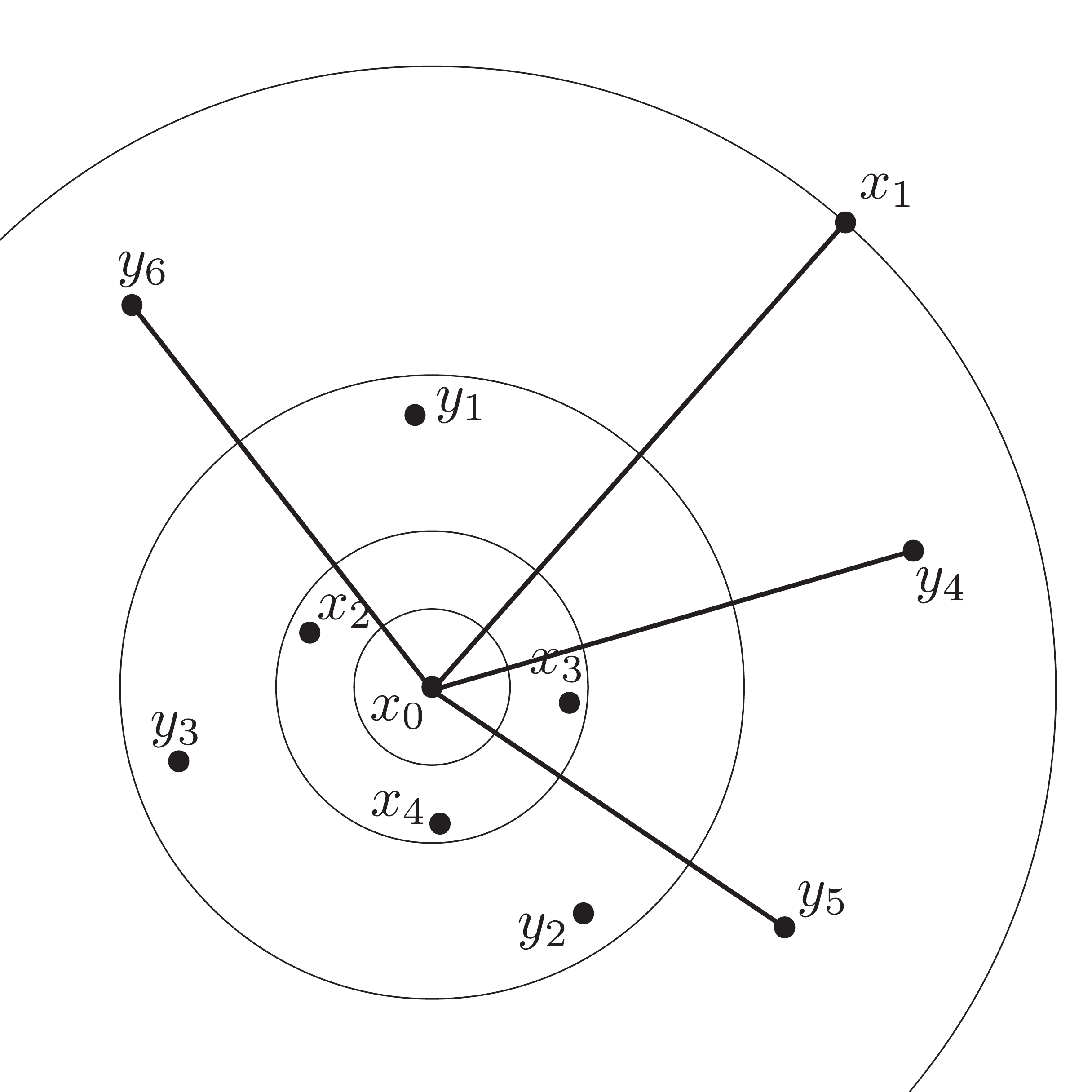}}
       \caption{Illustration of the well-scaled piece and induced sequences.
       We assume $d=3$, $k=2$, and $p=1$ and a simplex
       $X=(x_0,\ldots,x_5)$ (we use the case $k=2$ for convenience of drawing
       even though $X$ itself is well-scaled, but we notice that the construction produces simplices of smaller scales).
       The radii of the circles (ordered from outer to inner) are $\|x_1-x_0\|$,
       $\alpha_0\|x_1-x_0\|$, $\alpha_0^2\|x_1-x_0\|$,
       and $\alpha_0^3\|x_1-x_0\|$.
       The simplex $X$ and the well-scaled piece $Y_X$ are shown in (a). The auxiliary simplex,
       $\widetilde{X}_1$, is exemplified in (b) as an element meditating
       between the well-scaled simplex $X_1$ shown next and the original
       simplex $X$.
       The elements of the well-scaled sequence
       $X_1$, $X_2$, $X_4$ and $X_7$ are shown in~(c)-(f) respectively (note that $X_7$ is not listed in
       the example below).}
\end{figure}

For $X\in \name_{k,p}^1$ and a  well-scaled piece, $Y_X$, we define
the {\em augmentation of $X$ by $Y_X$} as
\begin{equation}\label{equation:def-underline-X}\underline{X}=X\times Y_X=(x_0,\ldots,x_{d+1},y_1,
\ldots,y_{k\cdot d})\in \name_{k,p}^1\times H^{k\cdot
d}.\end{equation}

For a fixed augmented element $\underline{X},$ we define two
sequences in $H^{d+2}$, the {\em auxiliary sequence},
$\widetilde{\Phi}_k\left(\underline{X}\right)
=\left\{\widetilde{X}_q\right\}_{q=0}^{k\cdot d}$, and  the {\em
well-scaled sequence}
$\Phi_k\left(\underline{X}\right)=\left\{X_q\right\}_{q=1}^{k\cdot
d+1}$. They will be used to formulate a multiscale inequality for
the polar sine function. We remark that the auxiliary sequence is
auxiliary in the sense that it is only used to establish the
inequality of equation~\eqref{firstofmany} for the well-scaled
sequence $\Phi(\underline{X})$.

\subsubsection{The Auxiliary Sequence
}\label{subsection:phi} If $a\in\,\mathbb{Z}$, then  let
$\overline{a}\in\{2,\ldots,d+1\}$ denote the unique integer such
that $\overline{a}=a \bmod{d}$.  We form the {\em auxiliary
sequence} $\left\{\widetilde{X}_q\right\}^{k\cdot d}_{q=0}$
recursively as follows.

\begin{definition}\label{definition:tilde-phi} If $X = (x_0,\ldots
,x_{d+1})\in \name_{k,p}^1$ and $\underline{X}=X\times
Y_X=(x_0,\ldots,x_{d+1} ,y_1,\ldots,y_{k\cdot d})$, then let
$\widetilde{\Phi}_k(\underline{X})=\left\{\widetilde{X}_q\right\}_{q=0}^{k\cdot
d}$ be the sequence of elements in $H^{d+2}$ defined recursively as
follows:
$$
\widetilde{X}_0=X,
$$
and
\begin{equation}\label{equation:def-X-q-tilde}\widetilde{X}_q=\widetilde{X}_{q-1}
\left(y_q,\overline{q+1}\right)\ \textup{ for }\ 1\leq q\leq k\cdot
d.
\end{equation}
\end{definition}
For example, if $d=3$, then
\begin{displaymath}\begin{array}{ccc}
\widetilde{X}_1=(x_0,x_1,y_1,x_3,x_4),&\widetilde{X}_2=(x_0,x_1,y_1,y_2,x_4),
&\widetilde{X}_3=(x_0,x_1,y_1,y_2,y_3);\\
\widetilde{X}_4=(x_0,x_1,y_4,y_2,y_3),&
\widetilde{X}_5=(x_0,x_1,y_4,y_5,y_3),&\widetilde{X}_6=(x_0,x_1,y_4,y_5,y_6).
\end{array}
\end{displaymath}
where $\widetilde{X}_1$ is illustrated in
Figure~\ref{figure:auxiliary}.

In general, we note that the elements $\widetilde{X}_q$ have the
following form:
\begin{equation}\label{equation:X-tilde}
\widetilde{X}_q=\begin{cases}(x_0,x_1,y_1,\ldots,y_q,x_{q+2},
\ldots,x_{d+1}),&\textup{if }1\leq q\leq d-1;\\(x_0,x_1,y_{j\cdot
d+1},\ldots,y_q,y_{q-d+1},\ldots,y_{j\cdot d}),&\textup{if } j\cdot
d<q<(j+1)\cdot d \ \textup{ for }1\leq j\leq
k-1;\\(x_0,x_1,y_{(j-1)\cdot d+1},\ldots,y_{j\cdot d}), &\textup{if
} q=j\cdot d \ \textup{ for }1\leq j\leq k.
\end{cases}
\end{equation}
In the special case where $d=1$, the first two cases of
equation~(\ref{equation:X-tilde}) are meaningless and
$\widetilde{X}_q=(x_0,x_1,y_q)$ for all $1\leq q\leq k$.

\subsubsection{The Well-Scaled Sequence
}\label{subsection:phi-1}

We derive the  {\em well-scaled} sequence $\Phi_k(\underline{X})$
from the auxiliary sequence $\widetilde{\Phi}_k(\underline{X})$ as
follows.

\begin{definition}\label{definition:phi} If $X\in \name_{k,p}^1$ and
$\underline{X}=X\times Y_X=(x_0,\ldots,x_{d+1},y_1,\ldots,y_{k\cdot
d})$, then let $\Phi_k(\underline{X})=\{X_q\}_{q=1}^{k\cdot d+1}$ be
the sequence of elements in $H^{d+2}$ such that
\begin{equation}\label{equation:def-X-q}
X_q=\begin{cases}\widetilde{X}_{q-1}\left(y_q,1\right), &\textup{if
}
1\leq q\leq k\cdot d;\\
\widetilde{X}_{k\cdot d},&\textup{if } q=k\cdot d+1.
\end{cases}
\end{equation}
\end{definition}

For example, if $d=3$, then the first six elements of the sequence
are
\begin{displaymath}\begin{array}{ccc}
X_1=(x_0,y_1,x_2,x_3,x_4),&X_2=(x_0,y_2,y_1,x_3,x_4),&X_3=(x_0,y_3,y_1,
y_2,x_4);\\
X_4=(x_0,y_4,y_1,y_2,y_3),&
X_5=(x_0,y_5,y_4,y_2,y_3),&X_6=(x_0,y_6,y_4, y_5,y_3).
\end{array}
\end{displaymath}
We illustrate $X_1$, $X_2$, $X_4$ and even $X_7$ in
Figures~\ref{figure:shortscale}-\ref{figure:shortscale3}.

We also note that in the very special case where $d=1$, then $X_1 =
(x_0,y_1,x_2)$, $X_q = (x_0,y_q,y_{q-1})$, $1<q\leq k\cdot d+1$ and
$X_{k\cdot d+1} = (x_0,x_1,y_{k\cdot d})$.

The following lemma shows that the elements
$X_q\in\,\Phi_k\left(\underline{X}\right)$ , $1\leq q\leq k\cdot
d+1$, are indeed well-scaled at the  vertex $x_0$. It follows
directly from the definition of the well-scaled sequence by checking
that each of the coordinates of the simplices $X_q$, $q=1, \ldots,
k\cdot d+1$, are in the correct annulus centered at $x_0$.
\begin{lemma}\label{lemma:index-function-no-tilde} If
$X\in \name_{k,p}^1$ and $\underline{X}=X\times Y_X$, then each term
of the sequence
$\Phi_k(\underline{X})=\left\{X_q\right\}_{q=1}^{k\cdot d+1}$ is
well-scaled at $x_0$ and we have the following estimates:
\begin{equation}\label{equation:max-bound-q}
\alpha_0^{k+1-\left\lceil\frac{q}{d}\right\rceil}\cdot\max\nolimits_{x_0}(X)<
\max\nolimits_{x_0}(X_q)\leq\alpha_0^{k-\left\lceil
\frac{q}{d}\right\rceil}\cdot\max\nolimits_{x_0}(X),\ \textup{ if }
\ 1\leq q\leq k\cdot d,
\end{equation}
and
\begin{equation}\label{equation:min-max-bound-kd+1}
\alpha_0\cdot\max\nolimits_{x_0}(X)<\min\nolimits_{x_0}(X_q)
\leq\max\nolimits_{x_0}(X_q)=\max\nolimits_{x_0}(X),\ \textup{ if }
\ q=k\cdot d+1.
\end{equation}
\end{lemma}

\subsubsection{Augmented Sets and a Discrete Multiscale Inequality
}

Lemma~\ref{lemma:index-function-no-tilde} assures us that the
elements $X_q$ have the correct structure in terms of relative
scale, but we still have to assure ourselves that we can pick the
sequence $\Phi(\underline{X})=\{X_q\}_{q=1}^{k\cdot d+1}$ so that it
satisfies the inequality of equation~\eqref{firstofmany}.  This is
easy to accomplish as long as we impose some conditions on this
sequence as well as the auxiliary sequence as we are choosing them.
We clarify this as follows.

Using the constant $C_{\mathrm{p}}$ of
Proposition~\ref{proposition:concentration-inequality-1}, we form
the {\em set augmentation} of the set $\name_{k,p}^1$, denoted by
$\underline{\name_{k,p}^1}$, as follows:
\begin{multline}\label{equation:def-big-omega}
\underline{\name_{k,p}^1}=\Bigg\{\underline{X}\in
\name_{k,p}^1\times[\Supp]^{k\cdot d} :\textup{ the sequences
}\widetilde{\Phi}_k(\underline{X})\textup{ and }
\Phi_k(\underline{X})\textup{ satisfy the inequality }\\
\pds_{x_0}\left(\widetilde{X}_q\right)\leq
C_{\mathrm{p}}\cdot\left(\pds_{x_0}\left(X_{q+1}\right)
+\pds_{x_0}\left(\widetilde{X}_{q+1}\right)\right),\ \textup{ for
all } \ 0\leq q< k \cdot d\Bigg\}\,.
\end{multline}
We note that this set has some complicated structure and that it is
not simply a product set.  Essentially it is a set of augmented
elements such that each coordinate is conditioned on the previous
coordinates in such a way as to satisfy a sequence of two-term
inequalities.  In this way the  well-scaled sequence
$\Phi(\underline{X})$ satisfies an inequality of the form in
equation~\eqref{firstofmany} by simply iterating the two-term
inequality of equation~\eqref{equation:def-big-omega}.

In fact, the  sets $\underline{\name_{k,p}^1}$ give rise to the
following multiscale inequality, whose direct proof (which we omit)
is based on a simple iterative argument followed by an application
of the Cauchy-Schwartz inequality.
\begin{lemma}\label{lemma:multiscale-toast}If $\underline{X}\in\underline{\name_{k,p}^1}$, then the
elements of the corresponding well-scaled sequence
$\Phi_k(\underline{X})=\{X_q\}_{q=1}^{k\cdot d+1}$ satisfy the
inequality
\begin{equation*}\pds^2_{x_0}(X)\leq (k\cdot d+1)\cdot
C_{\mathrm{p}}^{2\cdot k\cdot d}\, \sum_{q=1}^{k\cdot d+1}
\pds^2_{x_0} \left(X_q\right).
\end{equation*}
\end{lemma}

At this point, the motivation behind our construction of the
well-scaled and auxiliary sequences may become somewhat more
apparent. Our construction was governed by the iteration of the
two-term inequality of equation~\eqref{equation:def-big-omega}, with
the purpose of iteratively swapping out the simplices with bad
scaling, $\widetilde{X}_q$, for a  well-scaled simplex, $X_q$, and a
simplex $\widetilde{X}_{q+1}$ whose scaling is slightly better than
that of $\widetilde{X}_q$ because {\em fewer} of its edges are
grossly disproportionate. In this way, we gradually move from bad to
better, with each stage leaving us with an acceptable ``remainder'',
i.e., $X_q$.  We must be somewhat careful in this process, because
if we move too quickly,  somewhere down the line we will generate
simplices with worse structure than what we want.  This is the basic
reason that we have $d$ interpolating coordinates, $y_i$, in a given
annulus.  We must make sure that while we are ``growing''  the
interpolating coordinates out from $x_0$, they still remain
concordant with the length scales of the previous step.

\subsection{Estimating the Size of $\underline{\name_{k,p}}^1$}
\label{subsection:background-integration} We show here that for any
$X\in \name_{k,p}^1$, the corresponding ``slice''  in
$\underline{\name_{k,p}^1}$ is uniformly ``quite large'' for each
such $X$, and thus Lemma~\ref{lemma:multiscale-toast} can be applied
somewhat indiscriminately.  Later in Section~\ref{sizematters} we
will use that fact to show that the integral over $\name_{k,1}^1$,
can be bounded in a meaningful way by a corresponding integral over
the set $\underline{\name_{k,1}^1}$.

Once again, for the sake of clarity we need to rely on a bit of
technical notation to account for the structure of the set
$\underline{\name_{k,p}^1}$. While the notation is a bit cumbersome,
we believe that the idea is simple enough.

\subsubsection{Truncations and Projections of $\underline{\name_{k,p}^1}$}\label{subsubsection:trunc-proj-1}
We fix $0\leq q\leq k\cdot d$. If
$\underline{X}=(x_0,\ldots,y_{k\cdot d})\in\underline{
\name_{k,p}^1}$, then we define the {\em $q^{\text{th}}$ truncation
of $\underline{X}$} to be the function $ T_q:\underline{
\name_{k,p}^1}\rightarrow H^{d+2+q}$, where
\begin{equation}
T_q\left(\underline{X}\right)=\begin{cases}\quad\quad\quad\quad
X,&\textup{if }q=0;
\\(x_0,\ldots,x_{d+1},y_1,\ldots,y_q),&\textup{if }1\leq q\leq k\cdot
d.\end{cases}
\end{equation}
The ($d+2+q$)-tuple $ T_q\left(\underline{X}\right)$ is not to be
confused with the projection $(\underline{X})_q\in H$. If
$A\subseteq\underline{\name_{k,p}^1}$, then we denote the image of
$A$ by $T_q(A)$.

For $\underline{X}=(x_0,\ldots,y_{k\cdot
d})\in\underline{\name_{k,p}^1}$, we denote the pre-image of
$T_q\left(\underline{X}\right)=(x_0,\ldots,y_q)$ by
\begin{equation}
T^{-1}_q(x_0,\ldots,y_q)=\left\{\underline{X}'\in
\underline{\name_{k,p}^1}: \
T_q\left(\underline{X}'\right)=T_q\left(\underline{X}\right)=(x_0,
\ldots,y_q)\right\},
\end{equation}
where $(x_0,\ldots,y_q)$ is taken to mean $X$ if $q=0$.


Now, fixing $1\leq q\leq k\cdot d$, we define the $q^{th}$
projection of $\underline{X}$ onto $H$. For
$\underline{X}=(x_0,\ldots,y_{k\cdot
d})\in\underline{\name_{k,p}^1}$, let
$$
\pi_q(\underline{X})=y_q=\left(\underline{X}\right)_{d+1+q}.
$$
The set
$\pi_q\left(T_{q-1}^{-1}\left(x_0,\ldots,y_{q-1}\right)\right)$ is
composed of all possible $q^{th}$ coordinates of the well-scaled
pieces $Y_X=(y_1,\ldots,y_{k\cdot d})$ such that
$\underline{X}'=X\times Y_X\in T^{-1}_{q-1}(x_0, \ldots,y_{q-1})$.

We note that in this way we are giving another clear indication of
how the coordinates $y_q$ of
$\underline{X}\in\underline{\name_{k,p}^1}$ are ``conditioned'' on
the previous coordinates $(x_0,\ldots,x_{d+1},\ldots,y_{q-1})$.

\subsubsection{The ``Size'' of $\underline{\name_{k,p}^1}$}
For any $\underline{X}\in\underline{\name_{k,p}^1}$ and all $1\leq
q\leq k\cdot d$, we define the functions
\begin{equation}\label{equation:def-h-k}
g_{k,q}^1\left(\underline{X}\right)=
\mu\left(\pi_q\left(
T_{q-1}^{-1}\left(T_{q-1}\left(\underline{X}\right)\right)\right)\right).
\end{equation}The double subscript in this
case is not intended to evoke the index $p=1,2$.  It is important to
note that the functions $g^1_{k,q}(\underline{X})$ are independent
of the values of the coordinates $y_\ell$ for $\ell>q$, and as such,
they are well defined on the truncations $T_q(\underline{X})$.  That
is, we always have the equality (with a slight abuse of
notation)$$g^1_{k,q}(\underline{X})=g^1_{k,q}(T_q(\underline{X}))$$

For $1\leq q\leq k\cdot d,$ the following proposition (which is
proved in Appendix~\ref{app:nolyn}) estimates the sizes of the
functions $g^1_{k,q}$ (defined on $\underline{\name_{k,p}^1}$):
\begin{proposition}\label{proposition:multi-scale-good-case}
If $\underline{X}\in\underline{\name_{k,p}^1}$ and $1\leq q\leq
k\cdot d$, then
\begin{equation}
\mu\left(B(x_0,\alpha_0^{k-\left\lceil\frac{q}{d}\right\rceil}
\cdot\max\nolimits_{x_0}(X))\right)\geq
g_{k,q}^1 \left(\underline{X}\right)
\geq\frac{1}{2}
\cdot\mu\left(B(x_0,\alpha_0^{k-\left\lceil\frac{q}{d}\right\rceil}\cdot
\max\nolimits_{x_0}(X))\right).
\end{equation}
\end{proposition}

\subsection{Multiscale Integral Inequality}\label{sizematters}
Taking the  product of the functions $g^1_{k,q}$ over $q$ we have
the strictly positive function
\begin{equation}\label{equation:def-f-k}
f_{k}^1 \left(\underline{X}\right)= \prod_{q=1}^{k\cdot d}
g_{k,q}^1\left(\underline{X}\right).\end{equation}Thus, if we take
the measure induced by
\begin{equation}\label{biggieone}\frac{\di\mu^{d+2+k\cdot d}(\underline{X})}{f_{k}^1 \left(\underline{X}\right)}\Bigg|_{\underline{X}\in\underline{\name_{k,p}^1}}\ ,\end{equation}
then we essentially obtain the measure $\di\mu^{d+2}(X)$ over
$\name_{k,p}^1$ modified by a set of conditional probability
distributions on the coordinates $y_q$ for $1\leq q\leq k\cdot d$.

As such, integrating a function of $X$ according to the measure of
equation~\eqref{biggieone} reduces (after Fubini's Theorem) to
integrating the function with respect to $\mu^{d+2}$.  In this way
we can turn Lemma~\ref{lemma:multiscale-toast} into a meaningful
integral inequality. We formulate such an inequality as follows,
while using the notation
$$ N_k=(k+1)\cdot d+2.
$$

\begin{proposition}\label{proposition:bounding-sum-integral}
If $Q$ is a ball in $H$, $k\geq 3$, and $p=1,2$, then
\be \label{equation:bounding-sum-integral} \int_{\name_{k,p}^1(
Q)}\frac{\pds_{x_0}^2(X)}{\diam(X)^{d(d+1)}}\di\mu^{d+2} \leq
\left(k\cdot d+1\right)\cdot C_{\mathrm{p}}^{2\cdot k\cdot
d}\sum_{q=1}^{k\cdot d+1}\int_{\underline{\name_{k,p}^1( Q)}}
\frac{\pds_{x_0}^2(X_q)}{\diam(X)^{d(d+1)}}
\frac{\fdi\mu^{N_k}(\underline{X})}{f_k^1(\underline{X})}. \ee
\end{proposition}

\begin{proof}
This proposition is a direct consequence of
Lemma~\ref{lemma:multiscale-toast} and the following equation
\begin{equation}\label{equation:differential-identity-1}
\int_{\name_{k,p}^1(Q)}
\frac{\pds^2_{x_0}(X)}{\diam(X)^{d(d+1)}}\di\mu^{d+2}(X)=
\int_{\underline{\name_{k,p}^1(Q)}}
\frac{\pds_{x_0}^2(X)}{\diam(X)^{d(d+1)}}\frac{\fdi\mu^{N_k}(\underline{X})}
{f_{k}^1\left(\underline{X}\right)}, \ \forall  \text{ balls } Q
\subseteq H\,.
\end{equation}

For simplification, we prove
equation~\eqref{equation:differential-identity-1} for $Q=H$,
however, the idea applies to any ball $Q$ in $H$. First, using
Fubini's Theorem we obtain
\begin{multline}
\label{equation:iterated_integral}
\int_{\underline{\name_{k,p}^1}}\frac{\pds_{x_0}^2\left(X\right)}
{\diam\left(X\right)^{d(d+1)}}\frac{\fdi\mu^{N_k}(\underline{X})}{f_k^1\left(\underline{X}\right)}\\
=\int_{\name_{k,p}^1}\frac{\pds^2_{x_0}\left(X\right)}
{\diam\left(X\right)^{d(d+1)}}\left(\int_{\left\{Y_X:X\times
Y_X\in\,\underline{ \name_{k,p}^1}\right\}}\frac{\fdi\mu^{k\cdot
d}(Y_X)}{f_k^1(\underline{X})}\right)\di\mu^{d+2}(X).
\end{multline}

Then, we iterate the inner integral on the RHS of
equation~\eqref{equation:iterated_integral} and get
\begin{multline}\label{equation:dinner-time}
\int_{\left\{Y_X:X\times Y_X\in\,\underline{
\name_{k,p}^1}\right\}}\frac{\fdi\mu^{k\cdot
d}(Y_X)}{f_k^1(\underline{X})}\\
=\int_{\pi_1\left(T_0^{-1}(X)\right)}\cdots \int_{\pi_q\left(
T^{-1}_{q-1}(x_0,\ldots,y_{q-1})\right)} \cdots\int_{\pi_{k\cdot
d}\left( T_{k\cdot d-1}^{-1}(x_0,\ldots,y_{k\cdot
d-1})\right)}\frac{\fdi\mu\left(y_{k\cdot
d}\right)\cdots\di\mu(y_1)}{\prod_{q=1}^{k\cdot d}
g_{k,q}^1 \left(\underline{X}\right)},
\end{multline}
for the sets $\pi_{q}\left(T_{q-1}^{-1}(x_0,\ldots,y_{q-1})\right)$
conditionally defined given
$y_{q-1}\in\pi_{q-1}\left(T_{q-2}^{-1}(x_0,\ldots,y_{q-2})\right)$,
for $2\leq q\leq k\cdot d$.

For any fixed $(x_0,\ldots,y_{q-1}) \in H^{(d+1+q) },$ by the
definition of $g^1_{k,q}$ we have the equality
\begin{equation}\label{equation:nevada}
\int_{\pi_q\left( T_{q-1}^{-1}(x_0,\ldots,y_{q-1})\right)}
\frac{\fdi\mu(y_q)}{g_{k,q}^1 \left(\underline{X}\right)}=
\int_{\pi_q\left( T_{q-1}^{-1}(x_0,\ldots,y_{q-1})\right)}
\frac{\fdi\mu(y_q)}{\mu\left(\pi_q\left(
T_{q-1}^{-1}(x_0,\ldots,y_{q-1})\right)\right)}=1.
\end{equation}
Applying this to the iterated integral on the RHS of
equation~(\ref{equation:dinner-time}) we obtain
\begin{equation}\label{equation:saved-for-later}
\int_{\left\{Y_X:X\times Y_X\in\,\underline{
\name_{k,p}^1}\right\}}\frac{\fdi\mu^{k\cdot
d}(Y_X)}{f_k^1(\underline{X})}=1,\textup{ for all }X\in
\name_{k,p}^1.
\end{equation}
Combining equations~(\ref{equation:iterated_integral})
and~(\ref{equation:saved-for-later}), we conclude
equation~(\ref{equation:differential-identity-1}) and the
proposition.
\end{proof}
\begin{remark}\label{remark:big-shot}Since we do not take the time to establish the measurability of
$f_{k}^1$, one can instead follow  an alternative strategy:
Proposition~\ref{proposition:multi-scale-good-case} implies that
\begin{equation*}\label{equation:good-normal}f_{k}^1 \approx \prod_{q=1}^{k\cdot d}
\mu\left(B(x_0,\alpha_0^{k-\left\lceil\frac{q}{d}\right\rceil}
\cdot\max\nolimits_{x_0}(X))\right),\end{equation*}where the
constant of comparability is at worst $2^{k\cdot d}.$
The latter function is clearly measurable in $\underline{X}$, and
one can thus use it instead of $f_{k}^1$ for
Proposition~\ref{proposition:bounding-sum-integral}. Nevertheless
such a strategy will increase the estimate on the constant $C_1$ of
Theorem~\ref{theorem:upper-main}, and requires a slight change in
the choice of $\alpha_0$ in
equation~(\ref{equation:alpha-loose}).\end{remark}

\subsection{Concluding the Proof of Proposition~\ref{proposition:poorly-scaled-integration-d}}
\label{subsection:integration-d}

The rest of our analysis for
Proposition~\ref{proposition:poorly-scaled-integration-d} consists
of taking each term from the RHS of
equation~\eqref{equation:bounding-sum-integral} and adapting the
argument of Proposition~\ref{proposition:journe}. The basic idea is
to chop the set $\underline{\name_{k,1}^1(Q)}$ in a multiscale
fashion depending on the relative sizes of $X_q$ and $X$. This
requires some careful book keeping, but it allows us to involve the
very small (relative to $\diam(X)$) length scales in the
integration, and this will produce a constant (uniform in $q$) that
decays rapidly enough with $k$.

More specifically, we verify the following proposition, whose
combination with Proposition~\ref{proposition:bounding-sum-integral}
establishes
Proposition~\ref{proposition:poorly-scaled-integration-d}.
\begin{proposition}\label{proposition:gerry-halwell}
If  $1\leq q\leq k\cdot d+1$, then there exists a constant
$C_7=C_7(d,C_{\mu})$ such that
\begin{equation}\label{equation:rapper}
\int_{\underline{\name^1_{k,1}(Q)}}\frac{\pds_{x_0}(X_q)}
{\diam(X)^{d(d+1)}}\frac{\fdi\mu^{N_k}(\underline{X})}{f_k^1
(\underline{X})}\leq C_7\cdot\alpha_0^{k\cdot d^2}\cdot
J_d^{\mathcal{D}}(\mu|_Q)
\end{equation}for any ball $Q\subseteq H$.
\end{proposition}
\begin{proof}[Proof of Proposition~\ref{proposition:gerry-halwell}]
We partition the sets $\underline{\name_{k,1}^1}(Q)$ by the size of
$\max\nolimits_{x_0}(X)$. If $m\geq m(Q)$, then let
\begin{equation}\label{easydecomp}
\underline{\name_{k,1}^1(m)(Q)}=\left\{\underline{X}\in
\underline{S_1^{k,1}(\widetilde{\eta})(Q)}:\max\nolimits_{x_0}(X)\in
(\alpha_0^{m+1},\alpha_0^m]\right\}.
\end{equation}

Throughout the rest of the proof we fix $m \geq m(Q)$ and $1 \leq q
\leq k \cdot d+1$, and we further partition the set
$\underline{\name_{k,p}^1(m)(Q)}$ according to location in $\Supp$
in order to  reflect the quantity $\max\nolimits_{x_0}(X_q)$.
Specifically, we define the {\em scale exponent of $m$ and $q$}
\begin{equation}\label{equation:m-q-is}
\scl(m,q)=\begin{cases}m+k-\left\lceil\frac{q}{d}\right\rceil,&\textup{ if }1\leq q\leq k\cdot d;\\
\quad\quad m,&\textup{ if }q=k\cdot d+1.\end{cases}
\end{equation}The exponent $\scl(m,q)$  indicates the
correct length scale for the decomposition of the set
$\underline{\name_{k,1}^1(m)(Q)}$. Specifically, according to the
estimates of Lemma~\ref{lemma:index-function-no-tilde}, we have that
$$\max\nolimits_{x_0}(X_q)\in\left(\alpha_0^{\scl(m,q)+2},\alpha_0^{\scl(m,q)}\right].$$Thus,
for $j\in\Lambda_{\scl(m,q)}(Q)$, we let
\begin{equation}
\underline{P_{\scl(m,q),j}}=\left\{\underline{X}\in
\underline{\name_{k,1}^1(m)} :x_0\in P_{\scl(m,q),j}\right\},
\end{equation}
and obtain the following cover of $\underline{\name_{k,1}^1(m)(Q)}:$
\begin{equation}\label{easypartition}
\underline{\mathcal{P}_{\scl(m,q)}}(Q)=
\left\{\underline{P_{\scl(m,q),j}}\right\}_{j\in\Lambda_{\scl(m,q)}(Q)}.
\end{equation}

Letting $m\geq m(Q)$ and $j$ vary, the LHS of
equation~(\ref{equation:rapper}) satisfies the following inequality:
\begin{multline}\label{equation:homer}
\int_{\underline{\name_{k,1}^1}(Q)}\frac{\pds_{x_0}(X_q)}
{\diam(X)^{d(d+1)}}\frac{\fdi\mu^{N_k}(\underline{X})}{f_k^1(\underline{X})}\leq\\\sum_{m\geq
m(Q)}
\left[\sum_{j\in\,\Lambda_{\scl(m,q)}(Q)}\int_{\underline{P_{\scl(m,q),j}}}
\frac{\pds_{x_0}(X_q)}{\diam(X)^{d(d+1)}}\frac{\fdi\mu^{N_k}(\underline{X})}{f_k^1(\underline{X})}\right].
\end{multline}

Fixing  $j\in\Lambda_{\scl(m,q)}(Q)$, in addition to the fixed $k$,
$m$ and $q$, we will establish that
\begin{equation}\label{equation:scary-spice}
\int_{\underline{P_{\scl(m,q),j}}}\frac{\pds_{x_0}(X_q)}{\diam(X)^{d(d+1)}}
\frac{\fdi\mu^{N_k}(\underline{X})}{f_k^1(\underline{X})}\leq C_1
\cdot\alpha_0^{k\cdot
d^2}\cdot\beta_2^2\left(B_{\scl(m,q),j}\right)\cdot\mu\left(B_{\scl(m,q),j}\right).
\end{equation}Equations~(\ref{equation:homer}) and~(\ref{equation:scary-spice})
will directly imply Proposition~\ref{proposition:gerry-halwell}.

We prove equation~(\ref{equation:scary-spice}) as follows. Fix an
arbitrary $d$-plane $L$. Since the elements $\{X_q\}_{q=1}^{k\cdot
d+1}$ are well-scaled at $x_0$, by
Proposition~\ref{proposition:psin-bound-deviations} and
equation~\eqref{eq:well_scaled_ineq} we have the following bound on
the LHS of equation~(\ref{equation:scary-spice}):
\begin{multline}\label{equation:dunn-bros-1}
\int_{\underline{P_{\scl(m,q),j}}}\pds^2_{x_0}(X_q)
\frac{\fdi\mu^{N_k}(\underline{X})}{\diam(X)^{d(d+1)}\cdot
f_{k}^1(\underline{X})} \leq \\
\frac{2 \cdot (d+1)^2\cdot
(d+2)^2}{\alpha_0^{6}}\int_{\underline{P_{\scl(m,q),j}}}
\frac{D^2_2(X_q,L)}{\diam(X_q)^2}\cdot
\frac{\fdi\mu^{N_k}(\underline{X})}{\diam(X)^{d(d+1)}\cdot
f_{k}^1(\underline{X})}.
\end{multline}
To bound the RHS of equation~(\ref{equation:dunn-bros-1}) we focus
on the individual terms of
$$
\frac{D^2_2\left(X_q,L\right)}{\diam^2(X_q)}=\sum_{s=0}^{d+1}\frac{\dist^2
\left((X_q)_s,L\right)} {\diam^2(X_q)}.
$$
We arbitrarily fix $0\leq s\leq d+1$ and note the following cases of
possible values of $(X_q)_s$:
\vskip .2cm \noindent Case 1: $(X_q)_s = x_0$.  In this case $q$ has
no restriction, that is, $1\leq q\leq k\cdot d+1$.
\vskip .3cm \noindent Case 2: $(X_q)_s = x_1$. In this case $q =
k\cdot d+1$ (see
equations~\eqref{equation:X-tilde}-\eqref{equation:def-X-q}) and
thus $\scl(m,q) = m$.
\vskip .3cm \noindent Case 3: $(X_q)_s = x_i$, where $2\leq i\leq
d+1$.  In this case $1\leq q\leq d$ (see
equations~\eqref{equation:X-tilde}-\eqref{equation:def-X-q}) and
thus $\scl(m,q) = m+k-1$.
\vskip .3cm \noindent Case 4: $(X_q)_s = y_\ell$, where
$1\leq\ell\leq k\cdot d$.  In this case for each $1\leq q\leq k\cdot
d+1$, we have the following restriction on $\ell$:
$\max\{1,q-d\}\leq\ell\leq q$ and thus
\be \label{equation:restrict_q_l}
\max\left\{1,\left\lceil\frac{q}{d}\right\rceil -1\right\}
\leq\left\lceil\frac{\ell}{d}\right\rceil\leq\left\lceil\frac{q}{d}\right\rceil
. \ee
The calculation of the upper bound varies slightly according to
which case we consider.

Considering the first three cases simultaneously, we let $0\leq
i\leq d+1$ and examine the integrals decomposing the RHS of
equation~(\ref{equation:dunn-bros-1}), each of the form
$$
\int_{\underline{P_{\scl(m,q),j}}}\left(\frac{\dist\left(x_i,L\right)}
{\diam(X_q)}\right)^2\frac{\fdi\mu^{N_k}(\underline{X})}
{f_{k}^1(\underline{X})\cdot\diam(X)^{d(d+1)}}.
$$
Per Fubini's Theorem we obtain the equality
\begin{multline}\label{equation:three-ball}
\int_{\underline{P_{\scl(m,q),j}}}\left(\frac{\dist\left(x_i,L\right)}
{\diam(X_q)}\right)^2\frac{\fdi\mu^{N_k}(\underline{X})}{f_{k}^1(\underline{X})\cdot\diam(X)^{d(d+1)}}
=\\
\int_{ T_0\left(\underline{P_{\scl(m,q),j}}\right)}
\left(\int_{\left\{Y_X:X\times
Y_X\in\underline{P_{\scl(m,q),j}}\right\}} \frac{\fdi\mu^{k\cdot
d}(Y_X)}{\diam^2(X_q)\cdot f_k^1(\underline{X})}\right)
\,\dist^2(x_i,L)\,\frac{\fdi\mu^{d+2}(X)}{\diam(X)^{d(d+1)}}\,.
\end{multline}

To bound the inner integral on the RHS of
equation~(\ref{equation:three-ball}) we first note that
\begin{equation}\label{equation:larry-bird-1}
\diam (X) \geq \max\nolimits_{x_0}(X) =\|x_0-x_1\| > \alpha_0^{m+1}
\ \textup{ for all } \ X\in
T_0\left(\underline{P_{\scl(m,q),j}}\right).
\end{equation}
Combining Lemma~\ref{lemma:index-function-no-tilde} with
equations~\eqref{equation:m-q-is} and~(\ref{equation:larry-bird-1})
we see that
\begin{equation}\label{equation:measure-X-q-bound}
\diam(X_q)\geq\max\nolimits_{x_0}(X_q)>\alpha_0^{\scl(m,q)+1}.
\end{equation}
Since
$\displaystyle\diam\left(B_{\scl(m,q),j}\right)=8\cdot\alpha_0^{\scl(m,q)}$,
we rewrite equation~(\ref{equation:measure-X-q-bound}) as follows:
\begin{equation}\label{equation:diameter-lower-bound-measure}
\diam(X_q)\geq\frac{\alpha_0}{8}\cdot\diam\left(B_{\scl(m,q),j}\right).
\end{equation}
Combining equations~(\ref{equation:saved-for-later})
and~(\ref{equation:diameter-lower-bound-measure}), we obtain that if
$X\in T_0\left(\underline{P_{\scl(m,q),j}}\right)$, then
\begin{equation}\label{equation:dolly}
\int_{\left\{Y_X:X\times Y_X\in\underline{P_{\scl(m,q),j}}\right\}}
\frac{\fdi\mu^{k\cdot
d}(Y_X)}{f_k^1(\underline{X})\cdot\diam^2(X_q)}
\leq\frac{64}{\alpha_0^2}\cdot\frac{1}{\diam^2(B_{\scl(m,q),j})}\,.
\end{equation}

By the definition of $\underline{P_{\scl(m,q),j}}$,
$\underline{\name_{k,1}^1(m)}$ and $\name_{k,1}^1(m)$ we have the
equality
$$
T_0\left(\underline{P_{\scl(m,q),j}}\right)=\bigcup_{x_0\in\,P_{\scl(m,q),j}}
\left[\bigcup_{x_1\in A_m(x_0,1)}\left\{(x_0,x_1)\right\}\times
\left[A_k(x_0,\|x_1-x_0\|)\right]^d\right].
$$From this we trivially obtain the inclusion
$$
T_0\left(\underline{P_{\scl(m,q),j}}\right)\subseteq\bigcup_{x_0\in\,
P_{\scl(m,q),j}}\{x_0\}\times B(x_0,\alpha_0^m)\times
\left[B(x_0,\alpha_0^{m+k})\right]^d.
$$
Applying this together with the inequalities of
equations~(\ref{equation:larry-bird-1}) and~(\ref{equation:dolly})
to the RHS of equation~(\ref{equation:three-ball}) gives the
following inequality for all $0\leq i\leq d+1$:
\begin{multline}\label{multline:s-0-bound}
\int_{\underline{P_{\scl(m,q),j}}}\left(\frac{\dist\left(x_i,L\right)}
{\diam(X_q)}\right)^2\frac{\fdi\mu^{N_k}(\underline{X})}{f_{k}^1(\underline{X})\cdot\diam(X)^{d(d+1)}}
\\
\leq\frac{64}{\alpha_0^{d(d+1)+2}}\,\int_{P_{\scl(m,q),j}}\int_{B(x_0,\alpha_0^m)}
\int_{\left[B(x_0,\alpha_0^{m+k})\right]^d}\left(\frac{\dist\left(x_i,L\right)}
{\diam\left(B_{\scl(m,q),j}\right)}\right)^2\frac{\fdi\mu^{d+2}(X)}
{\left[\alpha_0^{m}\right]^{d(d+1)}}\,.
\end{multline}

Assume Case~1, that is, $i=0$. Then, after iterating the integral on
the RHS of equation~(\ref{multline:s-0-bound}), applying the
defining property of $d$-regular measure and the inclusion
$P_{\scl(m,q),j}\subseteq B_{\scl(m,q),j}$, we see that the term on
the RHS of equation~(\ref{multline:s-0-bound}) has the bound
\begin{equation}\label{multline:vinnie-johnson-1}
\frac{64\cdot
C_{\mu}^{d+1}}{\alpha_0^{d(d+1)+2}}\cdot\alpha_0^{k\cdot
d^2}\cdot\beta_2^2\left(B_{\scl(m,q),j},L\right)\cdot\mu
\left(B_{\scl(m,q),j}\right).
\end{equation}

Assume Case~2, that is, $i=1$, and recall that in this case
$q=k\cdot d+1$ and $\scl(m,q) = m$.  Thus we have the inclusion
$$
B(x_0,\alpha_0^m)\subseteq B_{\scl(m,q),j},\textup{ for all }x_0\in
P_{\scl(m,q),j}.
$$
Hence, iterating the integral on the RHS of
equation~(\ref{multline:s-0-bound}) and then applying similar
arguments to Case~1, we obtain the following bound for the LHS of
equation~(\ref{multline:s-0-bound}):
\begin{equation}\label{equation:bill-walton}\frac{64\cdot4^d\cdot
C_{\mu}^{d+1}}{\alpha_0^{d(d+1)+2}}\cdot\alpha_0^{k\cdot
d^2}\cdot\beta_2^2\left(B_{m,j},L\right)\cdot
\mu\left(B_{m,j}\right).
\end{equation}

Next, assume Case~3, that is, $2\leq i\leq d+1$ and recall that in
this case $1\leq q\leq d$ and $\scl(m,q)=m+k-1$. Using the fact that
$P_{\scl(m,q),j}\subseteq \frac{3}{4}\cdot B_{\scl(m,q),j}$, and the
defining property of $d$-regular measures we have the inequality
$$
\mu\left(P_{\scl(m,q),j}\right)\leq\mu\left(\frac{3}{4}\cdot
B_{\scl(m,q),j}\right) \leq C_{\mu}\cdot\left(3\cdot
\alpha_0^{m+k-1}\right)^d.
$$
Furthermore, we have the inclusion
$$
B\left(x_0,\alpha_0^{m+k}\right)\subseteq B_{\scl(m,q),j},\textup{
for all } x_0\in P_{\scl(m,q),j}.
$$
Iterating the integral as in the previous calculations, the LHS of
equation~(\ref{multline:s-0-bound}) is bounded by
\begin{equation}\label{equation:bimbo-gone}
\frac{64\cdot 3^d \cdot C_{\mu}^{d+1}}{\alpha_0^{d(d+1)+d+2}}
\cdot\alpha_0^{k\cdot d^2}\cdot\beta_2^2\left(B_{\scl(m,q),j},L
\right)\cdot\mu\left(B_{\scl(m,q),j}\right).
\end{equation}
Therefore, taking the maximal coefficient from
equations~(\ref{multline:vinnie-johnson-1}),
~(\ref{equation:bill-walton}) and~(\ref{equation:bimbo-gone}), the
LHS of equation~(\ref{multline:s-0-bound}) has the following uniform
bound for all $0\leq i\leq d+1$:
\begin{equation}\label{equation:multline-s-0-bound-infinite}
\frac{3^d\cdot2^7\cdot C_{\mu}^{d+1}}
{\alpha_0^{d(d+1)+d+2}}\cdot\beta_2^2\left(B_{\scl(m,q),j},L\right)
\cdot\mu\left(B_{\scl(m,q),j}\right).
\end{equation}

At last we consider Case~4, where the terms in the sum comprising
the RHS of equation~(\ref{equation:dunn-bros-1}) are of the form:
$$
\int_{\underline{P_{\scl(m,q),j}}}\left(\frac{\dist\left(y_{\ell},L\right)}
{\diam(X_q)}\right)^2\frac{\fdi\mu^{N_k}(\underline{X})}
{f_k^1(\underline{X})\cdot\diam(X)^{d(d+1)}}\,,
$$
where $1\leq l\leq k\cdot d$.

Iterating the integral and applying
equation~(\ref{equation:diameter-lower-bound-measure}), we obtain
\begin{multline}\label{multline:robert-parish-1}
\int_{\underline{P_{\scl(m,q),j}}}\left(
\frac{\dist\left(y_{\ell},L\right)}{\diam(X_q)}\right)^2
\frac{\fdi\mu^{N_k}(\underline{X})}{f_k^1(\underline{X})
\cdot\diam(X)^{d(d+1)}} \leq\\
\frac{64}{\alpha_0^2} \,
\int_{T_0\left(\underline{P_{\scl(m,q),j}}\right)}
\left(\int_{\left\{Y_X:X\times
Y_X\in\,\underline{\name_{k,p}^1}\right\}}
\left(\frac{\dist\left(y_{\ell},L\right)}{\diam\left(B_{\scl(m,q),j}
\right)}\right)^2\frac{\fdi\mu^{k\cdot
d}(Y_X)}{f_k^1(\underline{X})}\right)
\frac{\fdi\mu^{d+2}(X)}{\diam(X)^{d(d+1)}}\,.
\end{multline}
In order to bound the RHS of
equation~(\ref{multline:robert-parish-1}), we first calculate a
uniform bound  in $1\leq l\leq k\cdot d$ for the interior integral.
Then, completing the integration with respect to $X\in
T_0\left(\underline{P_{\scl(m,q),j}}\right)$ will give the desired
bound in terms of the corresponding $\beta_2$ number.

For fixed $X\in T_0\left(\underline{P_{\scl(m,q),j}}\right)$ and
$1\leq\ell\leq k\cdot d$, after iterating the interior integral on
the RHS of equation~(\ref{multline:robert-parish-1}) and applying
equation~(\ref{equation:nevada}) we have that
\begin{multline}\label{equation:lucky}
\int_{\left\{Y_X:X\times Y_X\in\,\underline{\name_{k,p}^1}\right\}}
\left(\frac{\dist\left(y_{\ell},L\right)}{\diam\left(B_{\scl(m,q),j}\right)}\right)^2\frac{\fdi\mu^{k\cdot
d}(Y_X)}{f_k^1(\underline{X})}=\\\int_{\pi_1\left(
T^{-1}_0(X)\right)}\cdots\int_{\pi_{\ell}\left(
T^{-1}_{\ell-1}(x_0,\ldots,y_{\ell-1})\right)}
\left(\frac{\dist\left(y_{\ell},L\right)}{\diam\left(B_{\scl(m,q),j}\right)}\right)^2\frac{\fdi\mu\left(y_{\ell}\right)
\cdots\di\mu(y_1)}{\prod_{s=1}^{\ell}
%
g_{k,s}^1}\,,
\,.
\end{multline}
where we used the notation $g_{k,s}^1$ defined in
equation~\eqref{equation:def-h-k}.
Given $\underline{X}\in\underline{P_{\scl(m,q),j}}$ we fix
$(x_0,\ldots,y_{\ell-1})=T_{\ell-1}(\underline{X})$ and calculate a
bound for the integral
$$
\int_{\pi_{\ell}\left(
T^{-1}_{\ell-1}(x_0,\ldots,y_{\ell-1})\right)}
\left(\frac{\dist\left(y_{\ell},L\right)}{\diam\left(B_{\scl(m,q),j}\right)}
\right)^2\frac{\fdi\mu\left(y_{\ell}\right)}{g_{k,\ell}^1}\,.
$$
We first obtain an upper bound for
$$
\frac{1}{g_{k,\ell}^1} = \frac{1}{\mu\left(\pi_{\ell}\left(
T^{-1}_{\ell-1}(x_0,\ldots,y_{\ell-1})\right)\right)},
$$
and then complete the integration.

To obtain that bound, we apply
Proposition~\ref{proposition:multi-scale-good-case} to get that for
all $1\leq\ell\leq k \cdot d$:
\be \label{equation:contessa} \mu \left(\pi_{\ell}\left(
T^{-1}_{\ell-1}(x_0,\ldots,y_{\ell-1})\right)\right) \geq
\frac{1}{2} \cdot \mu \left(B
\left(x_0,\alpha_0^{k-\left\lceil\frac{\ell}{d}\right\rceil}\cdot\max\nolimits_{x_0}(X)\right)\right)\,.
\ee
Next, applying equations~\eqref{equation:m-q-is},
\eqref{equation:restrict_q_l} and~(\ref{equation:larry-bird-1}) as
well as the fact that $\alpha_0<1$, we note that
\be \label{eq:alpha_sc_bounds}
\alpha_0^{k-\left\lceil\frac{\ell}{d}\right\rceil}\cdot\max\nolimits_{x_0}(X)\geq
\alpha_0^{k-\left\lceil\frac{\ell}{d}\right\rceil +m+1} \geq
\alpha_0^{k-\left\lceil\frac{q}{d}\right\rceil+m+2} = \alpha_0^{{\rm
sc}(m,q)+2}. \ee
We note that the RHS of equation~\eqref{equation:measure-comp}
extends to all $r>0$ and apply it to obtain the following bound
\be \label{equation:contessa2}
\mu\left(B\left(x_0,\alpha_0^{\scl(m,q)+2}\right)\right)\geq\frac{1}{
C_{\mu}^2}\cdot\left(\frac{\alpha_0^2}{4}\right)^d\cdot\mu
\left(B_{\scl(m,q),j}\right). \ee
Finally, combining equations~\eqref{equation:contessa},
\eqref{eq:alpha_sc_bounds} and~\eqref{equation:contessa2} we
conclude that
\be \label{equation:contessa3} \mu \left(\pi_{\ell}\left(
T^{-1}_{\ell-1}(x_0,\ldots,y_{\ell-1})\right)\right) \geq
\frac{1}{2\cdot
C_{\mu}^2}\cdot\left(\frac{\alpha_0^2}{4}\right)^d\cdot\mu
\left(B_{\scl(m,q),j}\right). \ee

Noting that $\pi_{\ell}\left(
T^{-1}_{\ell-1}(x_0,\ldots,y_{\ell-1})\right)\subseteq
B_{\scl(m,q),j}$ and applying equation~(\ref{equation:contessa3}),
we have the inequality
\begin{multline}\label{equation:garcia}
\int_{\pi_{\ell}\left(
T_{\ell-1}^{-1}(x_0,\ldots,y_{\ell-1})\right)}
\left(\frac{\dist\left(y_{\ell},L\right)}
{\diam\left(B_{\scl(m,q),j}\right)}\right)^2\cdot\frac{\fdi\mu(y_\ell)}
{\mu\left(\pi_{\ell}\left(
T^{-1}_{\ell-1}(x_0,\ldots,y_{\ell-1})\right)\right)}\leq\\\frac{2\cdot
4^d\cdot C^2_{\mu}}{\alpha_0^{2\cdot
d}}\cdot\beta_2^2\left(B_{\scl(m,q),j}\right).
\end{multline}
Then, applying this and equation~(\ref{equation:nevada}) to the RHS
of equation~(\ref{equation:lucky}), we have the following inequality
for all $X\in T_0\left(\underline{P_{\scl(m,q),j}}\right)$:
\begin{equation}\label{equation:fonzi}
\int_{\left\{Y_X:X\times Y_X\in\,\underline{\name_{k,p}^1}\right\}}
\left(\frac{\dist\left(y_{\ell},L\right)}{\diam\left(B_{\scl(m,q),j}\right)}\right)^2\frac{\fdi\mu^{k\cdot
d}(Y_X)}{f_k^1(\underline{X})}\leq \frac{2\cdot 4^d\cdot
C^2_{\mu}}{\alpha_0^{2\cdot
d}}\cdot\beta_2^2\left(B_{\scl(m,q),j}\right).
\end{equation}
Furthermore, noting that
$$
\int_{T_0\left(\underline{P_{\scl(m,q),j}}\right)}\frac{\fdi\mu^{d+2}(X)}{\diam(X)^{d(d+1)}}\leq
\frac{C_{\mu}^{d+1}}{\alpha_0^{d(d+1)}}\cdot\alpha_0^{k\cdot
d^2}\cdot\mu\left(B_{\scl(m,q),j}\right),
$$
per equations~(\ref{multline:robert-parish-1})
and~(\ref{equation:fonzi}), we have the following uniform bound for
all $1\leq \ell\leq k\cdot d$:
\begin{multline}\label{equation:chachi}
\int_{\underline{P_{\scl(m,q),j}}}\left(\frac{\dist\left(y_{\ell},L\right)}
{\diam(X_q)}\right)^2\frac{\fdi\mu^{N_k}(\underline{X})}
{f_k^1(\underline{X})\cdot\diam(X)^{d(d+1)}}\leq\\
\frac{128\cdot4^d\cdot C_{\mu}^{d+3}}{\alpha_0^{d^2+3\cdot
d+2}}\cdot\alpha_0^{k\cdot
d^2}\cdot\beta_2^2(B_{\scl(m,q),j})\cdot\mu\left(B_{\scl(m,q),j}\right).
\end{multline}

Finally, taking largest coefficient from
equations~(\ref{equation:multline-s-0-bound-infinite})
and~(\ref{equation:chachi}), we have the bound
\begin{multline*}
\int_{\underline{P_{\scl(m,q),j}}}\frac{D_2\left(X_q,L\right)}{\diam^2(X_q)}\frac{\fdi\mu^{N_k}(\underline{X})}
{f_{k}^1(\underline{X})\cdot\diam(X)^{d(d+1)}}\leq\\
\frac{(d+2)\cdot128\cdot4^d\cdot C_{\mu}^{d+3}}{\alpha_0^{d^2+3\cdot
d +2}}\cdot\alpha_0^{k\cdot
d^2}\cdot\beta_2^2(B_{\scl(m,q),j},L)\cdot\mu\left(B_{\scl(m,q),j}\right).
\end{multline*}Therefore, taking the infimum over all such
$d$-planes $L$ we obtain the bound \begin{multline*}
\int_{\underline{P_{\scl(m,q),j}}}\frac{D_2\left(X_q,L\right)}{\diam^2(X_q)}\frac{\fdi\mu^{N_k}(\underline{X})}
{f_{k}^1(\underline{X})\cdot\diam(X)^{d(d+1)}}\leq\\
\frac{(d+2)\cdot128\cdot4^d\cdot C_{\mu}^{d+3}}{\alpha_0^{d^2+3\cdot
d +2}}\cdot\alpha_0^{k\cdot
d^2}\cdot\beta_2^2(B_{\scl(m,q),j})\cdot\mu\left(B_{\scl(m,q),j}\right).
\end{multline*}Combining this with equation~(\ref{equation:dunn-bros-1}) establishes the
conclusion of equation~(\ref{equation:scary-spice}).
\end{proof}

\begin{remark}
\label{remark:direct_menger} The interpolation
 procedure applied above verifies the intuitive idea that the small length scales
 ought not to contribute much information when the overall length scale of a simplex is large.
One reason that our argument will not work for the curvature
specified in equation~\eqref{eq:direct_menger} is that such a
curvature is ``resistant'' to our current interpolation procedure.
The gains of small length scales made via the interpolation is
nullified by the very singular behavior of this curvature on
simplices with bad scaling.  In a sense, this curvature incorporates
too much information for our current method to handle, and if the
continuous version of this curvature can be bounded by the
Jones-type flatness (in the spirit of
Theorem~\ref{theorem:upper-main}), then we must use a different
method to show it.
\end{remark}

\section{Proof of
Proposition~\ref{proposition:poorly-scaled-integration}}\label{section:later-integrate}
Due to the similarity of various parts of the proof of
Proposition~\ref{proposition:poorly-scaled-integration} with the
ideas and computations of Section~\ref{section:proof-prop-poorly-d},
this current section focuses mostly on the new ideas required to
prove the proposition, including some technical notation and
statements.  Computations and ideas presented previously are
referenced to as needed.

We define the constant
$$N_n=2^{n-1}-1,$$and we remark that
$N_n$ and $N_k$ (defined in Subsection~\ref{sizematters}) are two
different constants. We also define the
constant$$M_n=d+2+N_n=d+1+2^{n-1}.$$   We recall that
$\name_{k,1}^n$ is the set of multi-handled rakes whose handles
occur at their first $n$ coordinates and whose tines occur at their
last $d+1-n$ coordinates.

Here we adapt the methods of
Section~\ref{section:proof-prop-poorly-d} to deal with the problems
present in integrating the curvature over the regions
$\name_{k,1}^n$. Our adaptation consists of two stages.  First we
split the simplex $X\in\name_{k,1}^n$ into a sequence of {\em
single-handled} rakes using an interpolation procedure similar to
that of section~\ref{section:proof-prop-poorly-d}.  The basic idea
is to use such a procedure to ``break off'' each of the $n$ handles
from the simplex $X$, thereby forming a sequence of single-handled
rakes with elements denoted by $X^s$.   This procedure generates an
``augmentation'' of $\di\mu^{d+2}(X),$ as well as an integral
inequality along the lines of
Proposition~\ref{proposition:bounding-sum-integral} of
Section~\ref{sizematters}.

Once we have made this exchange, we can again apply the methods of
Section~\ref{subsection:multiscale} to the single-handled rakes
obtained from the first step and obtain the proper control in terms
of the $\beta_2$ numbers.

The result is that we exchange integrals over $\di\mu^{d+2}(X)$ for
integrals over a ``doubly augmented'' measure
 (depending on a much larger ``variable'') which allows us to incorporate
the necessary information from small scales. Simply put, we iterate
the methods of Section~\ref{subsection:multiscale} in an highly
adaptive way.  We remark that the type of analysis done in this
section was unnecessary in the case $d=1$ due to the extreme
simplicity of the ``combinatorial structure'' of triangles.  Much of
the work that we had to do revolved around dealing with the problems
of disparities of scale that arise from the more complicated
structure of $(d+1)$-simplices for $d\geq2$.

\subsection{Rake Sequences and  Pre-Multiscale Inequality}
We define a {\em short-scale piece} for $X\in \name_{k,1}^n$ to be
an $N_n$-tuple of the form
\begin{equation}\label{equation:poor-scaled}Z_X=(z_1,\ldots,z_{N_n})\in
 \left[A_k(x_0,\max\nolimits_{x_0}(X))\right]^{N_n},\end{equation}
and illustrate it in Figure~\ref{figure:help}.

For $X\in \name_{k,1}^n$ and  $Z_X$ we define an augmentation of $X$
by $Z_X$ as
\begin{equation}\overline{X}=X\times Z_X=(x_0,\ldots,x_{d+1},z_1,\ldots,z_{N_n})\in
\name_{k,1}^n\times H^{N_n}.\end{equation} We note that
$\overline{X}\in H^{M_n}$ and that all coordinates of $Z_X$ are in
the annulus centered at $x_0$ and determined by  $\max_{x_0}(X)$.

For $\overline{X}$ we construct a sequence of {\em single-handled
rakes}, $\Psi(\overline{X})=\{X^s\}_{s=1}^{2^{n-1}}$, in $H^{d+2}$
(\'{a} la the construction in Section~\ref{section:well-scaled}) and
use them to formulate an inequality for the polar sine on
$\name_{k,1}^n$ \'{a} la equation~\eqref{firstofmany}.  The major
difference is that the $X^s$ are not well-scaled, and the length of
the sequence is determined by $1<n\leq d$, not $k$.  Despite the
fact that they are not necessarily well-scaled, we can  construct
them so that their scaling is better than the original simplex $X$,
i.e., $$\SCale_{x_0}(X^s)\geq\SCale_{x_0}(X).$$

Just as before, in order to get the type of sequence we want, we
must first define an auxiliary sequence which will only be used to
construct the desired type of sequence.
\begin{definition}\label{definition:tilde-Z}If $X\in
\name_{k,1}^n$ and $\overline{X}=X\times
Z_X=(x_0,\ldots,x_{d+1},z_1,\ldots,z_{N_n})$, then let $Z_m^j$,
$j=0, \ldots, n-1$, $m=1, \ldots, 2^j$, be the doubly indexed
sequence of elements of $H^{d+2}$ defined recursively as follows:
$$Z_1^0=X,$$ and for $0 \leq j<n-1$ and $\sigma_j$ denoting the
transposition of $n-j-1$ and $n-j$ (acting on $Z^j_m$ by replacing
its coordinates at those indices)
 \begin{equation}\label{gbba} Z_{2m-1}^{j+1}=
Z^j_m\left(n-j,z_{2^j+(m-1)}\right),\end{equation}and
\begin{equation} Z_{2m}^{j+1}=\sigma_j
\left(
Z^j_m\left(n-j-1,z_{2^j+(m-1)}\right)\right).\end{equation}\end{definition}

Realizing that this is a fairly technical definition, we remark that
its purpose is only to give a sensible and formal framework for
isolating the individual handles of the original simplex $X$.
Furthermore, any such  method must work   under the restrictions
imposed by the two-term inequality for the polar sine.  As such, we
must have some sort of iterative scheme which allows us to swap out
one at a time.  For this reason, we construct the different
generations of the simplices $Z^j_m$, with the simplices of each
successive generation having one less handle than the previous
generation.  The final generation will have only one handle, and
this is the sequence of simplices that we really want to work with.

\begin{figure}[htbp]
     \centering
     \subfigure[simplex $X$ and $Z_X$]{\label{figure:help}
          \includegraphics[height=2in,width=2in]{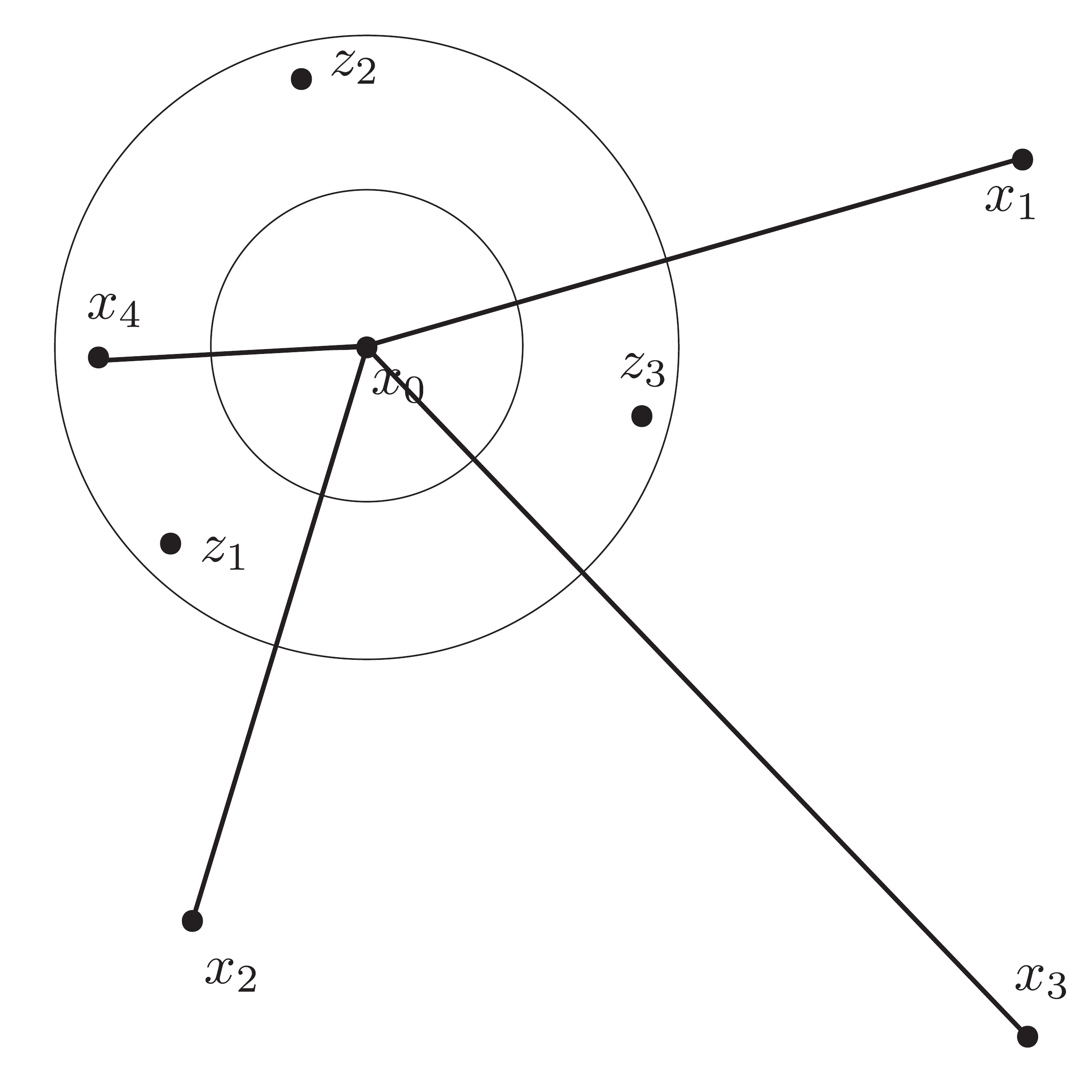}}
     \subfigure[auxiliary simplex $Z^1_2$]{\label{figure:nope}
           \includegraphics[height=2in,width=2in]{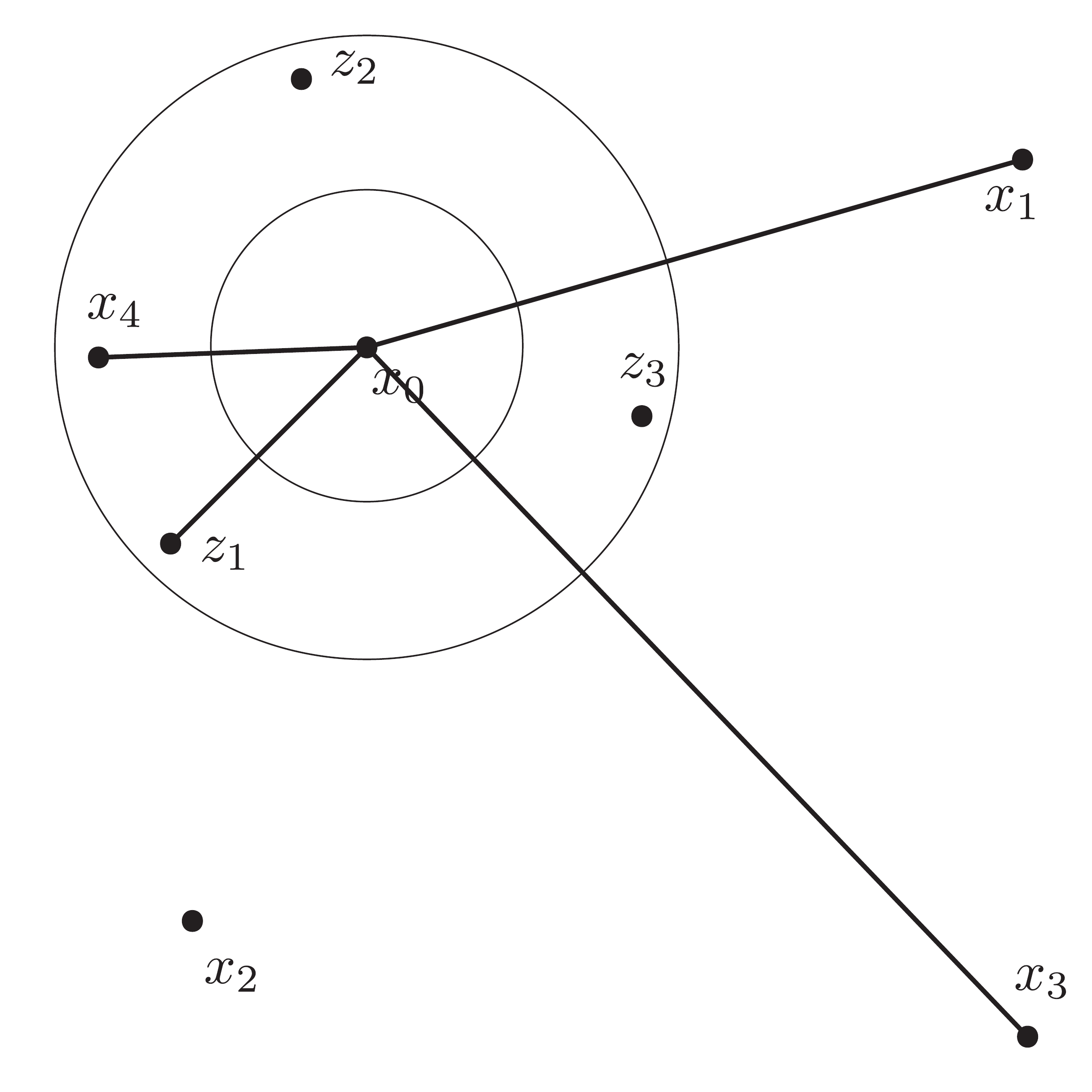}}
     \subfigure[rake simplex $X^2=Z^2_2$]{\label{figure:dope}
           \includegraphics[height=2in,width=2in]{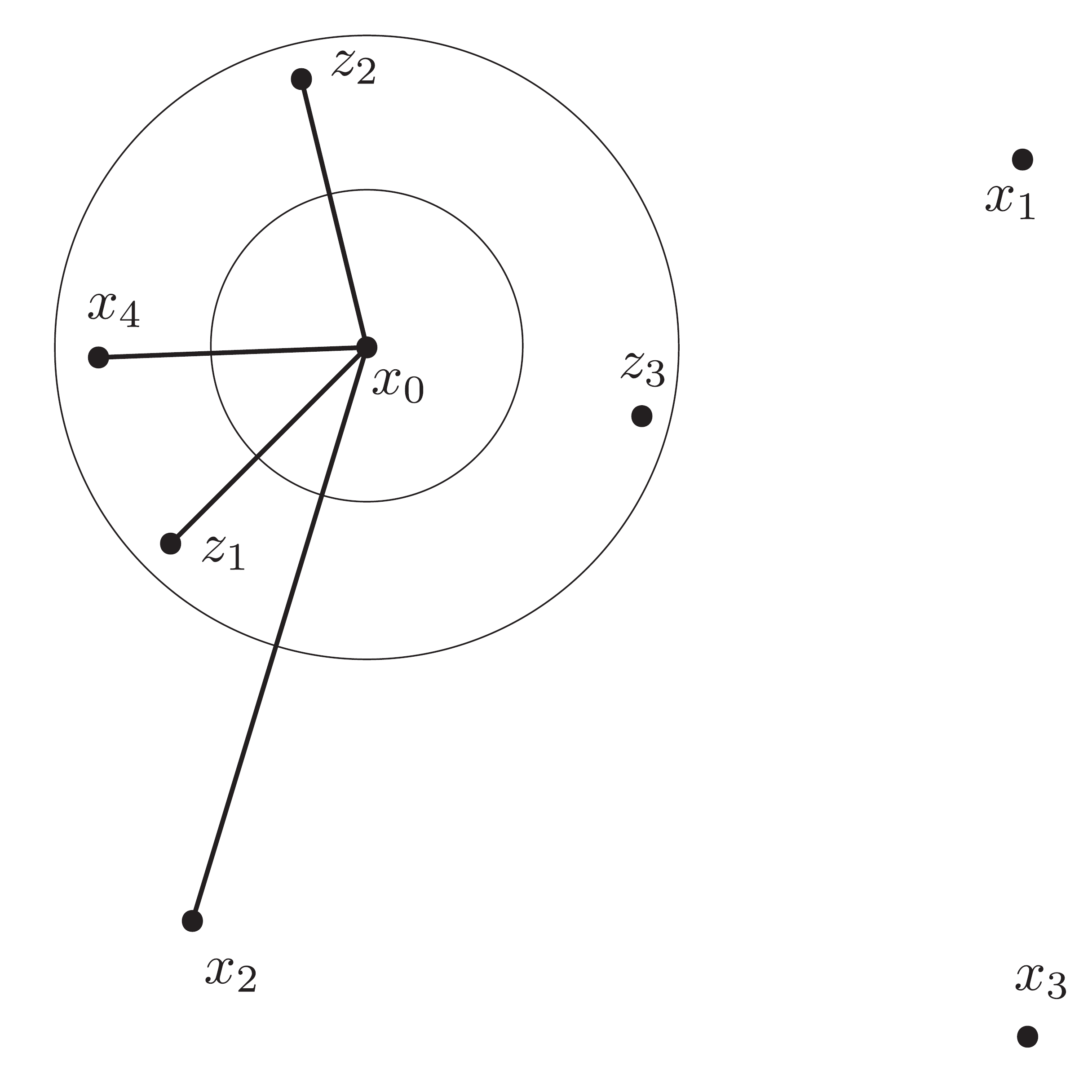}}
     \caption{Illustration of the short-scale piece and elements of the auxiliary and rake sequences.
     Here $X$ is a $3$-handled rake. Note that the coordinates of $Z_X$ are
contained in a single annulus rather than many annuli as was the
case for $Y_X$. The element $Z^1_2=(x_0,x_1,x_3,z_1,x_4)$ of the
auxiliary sequence has two handles rather than three (as $X$). The
element $Z_2^2=(x_0,x_2,z_2,z_1,x_4)$ has one handle and indeed it
is in the rake sequence (denoted by $X^2$).}
     \label{fig:three_noisy_lines}
\end{figure}

Using the $(n-1)^{\text{th}}$-generation of elements of the
auxiliary sequence above we define the {\em rake sequence }
$\Psi_k\left(\overline{X}\right)$ as follows.
\begin{definition}\label{definition:helped}If $X\in
\name_{k,1}^n$, $\overline{X}=X\times Z_X,$ and $Z^j_m$ as above,
then let $\Psi_k\left(\overline{X}\right)=\{X^s\}_{s=1}^{2^{n-1}}$
be the sequence of elements in $H^{d+2}$ such that
\begin{equation}\label{equation:def-X^s} X^s= Z_{s}^{n-1},
\textup{ for }1\leq s\leq 2^{n-1}.\end{equation}\end{definition}
We note that the simplex $X^s$ has exactly three coordinates taken from the original simplex $X$: the ``base'' vertex
$x_0$, the handle vertex $x_{i_s}$ for some index $1\leq i_s\leq n$, and the tines $x_i$ for some $n+1\leq i\leq d+1$.  The
rest of the coordinates are taken from the short scale piece $Z_X.$  This fact is apparent from following the recursive definitions, as well as the fact that $X^s$ can have only one
possible handle, this handled being inherited from one of the handles for $X$.

For example, if $d=3$ and $n=3$, then the simplex $X$ has three
handles located at the first $3$ coordinates, with the
$4^{\mathrm{th}}$ coordinate being the tine. Hence the sequences
$\widetilde{\Psi}_k\left(\overline{X}\right)$ and
$\Psi_k\left(\overline{X}\right)$ fit into the following tree:
\begin{equation*}\label{equation:example-poor}\xymatrix@R.5ex{& &Z_1^2=(x_0,x_1,z_2,z_1,x_4)=X^1\\
 &Z_1^1=(x_0,x_1,x_2,z_1,x_4)\ar[ur]\ar[dr]&\\& &Z_2^2=(x_0,x_2,z_2,z_1,x_4)=X^2\\ X=(x_0,x_1,x_2,x_3,x_4)\ar[uur]\ar[ddr]& &\\&
& Z_3^2=(x_0,x_1,z_3,z_1,x_4)=X^3&
\\& Z_2^1=(x_0,x_1,x_3,z_1,x_4)\ar[dr]\ar[ur]&\\& & Z_4^2=(x_0,x_3,z_3,z_1,x_4)=X^4}\end{equation*}
We illustrate the elements $Z_2^1$ and $Z_2^2$ of this tree in
Figures~\ref{figure:nope} and~\ref{figure:dope} respectively.

We note that the information provided by the original simplex
$X=(x_0,x_1,x_2,x_3,x_4)$ has been ``distributed'' over the new
sequence of simplices, and in a very real sense has been
``decoupled''.  As such, this will allow us to pursue our analysis
of these simplices more or less independently.

The following lemma establishes that all elements of
$\Psi_{k}\left(\overline{X}\right)$ are single-handled rakes, and
follows directly from the definition of the rake sequence.

\begin{lemma}\label{lemma:poor-scaled-property-size}If $X\in
\name_{k,1}^n$, $\overline{X}=X\times Z_X$,  and
$X^s\in\,\Psi_k\left(\overline{X}\right)$ with $1\leq s\leq
2^{n-1}$, then  each $X^s$ is a rake, and for $0\leq k'\leq  k-1$,
we have that
\begin{equation}\label{equation:multi-to-single}X^s\in
S^1_{k',2}.\end{equation}
\end{lemma}
%


We remark that we take the tolerance for the subscript index to be
$p=2$ in Lemma~\ref{lemma:poor-scaled-property-size}, and this is
simply because in our construction of the rakes, $X^s$, we lose a
tiny bit of accuracy in determining the relative lengths of the
smallest and largest edges at $x_0$. This slight change in accuracy
is accounted for by the extra power of $\alpha_0$ in determining the
length of the interval that $\SCale_{x_0}(X^s)$ sits in.

Using the  ideas of Section~\ref{subsection:background-integration}
we  construct a set augmentation of  $\name_{k,1}^n$, denoted by
$\overline{\name_{k,1}^n}$, which has a uniformly large size for a
given $X\in\name_{k,1}^n$ in the sense given by
Proposition~\ref{proposition:multi-scale-good-case}, and is such
that we have a version of equation~\eqref{firstofmany}. Using the
constant $C_{\mathrm{p}}$ of
Proposition~\ref{proposition:concentration-inequality-1},  we define
the set augmentation $\overline{\name_{k,1}^n}$ as
\begin{multline}\label{equation:set-overline}\overline{\name_{k,1}^n}=\Bigg\{\overline{X}\in \name_{k,1}^n\times[\Supp]^{N_n}:
\textup{ for all } 0\leq j<n-1\,\textup{ and } 1\leq m\leq
2^j,\textup{ the sequence }Z^j_m\\\textup{ satisfies the
inequalities: }\pds_{x_0}\left( Z_m^j\right)\leq
C_{\mathrm{p}}\left[\pds_{x_0}\left(
Z_{2m-1}^{j+1}\right)+\pds_{x_0}\left(
Z_{2m}^{j+1}\right)\right]\Bigg\}.\end{multline}

The sequence $\Psi_k\left(\overline{X}\right)$ gives rise to the
following pre-multiscale inequality for the polar sine, which is
analogous to Lemma~\ref{lemma:multiscale-toast} and similarly can be
immediately proved by iterative application of the defining
inequality followed by the Cauchy-Schwartz inequality.
\begin{lemma}\label{lemma:X^s-inequality}If $\overline{X}\in\overline{\name_{k,1}^n}$, then the sequence of single-handled rakes
$\Psi_k(\overline{X})=\{X^s\}_{s=1}^{2^{n-1}}$ satisfy the
inequality$$\pds_{x_0}^2(X)\leq 2^{n-1}\cdot
C_{\mathrm{p}}^{2\cdot(n-1)}\sum_{s=1}^{2^{n-1}}\pds^2_{x_0}(X^s).$$\end{lemma}

Just as in Subsection~\ref{subsection:background-integration}, for
$0\leq s\leq N_n$ we define the truncations, $T_s$, the projections
$\pi_s$, and  the functions
\begin{equation}\label{equation:measure-function-s}g_{k,s}^n\left(\overline{X}\right)=
\mu\left(\pi_s\left(T_{s-1}^{-1}\left(T_{s-1}(\overline{X})\right)\right)\right),\textup{
for all }1\leq s\leq N_n.\end{equation}These functions are again
positive and satisfy similar estimates as before.  Specifically,
adapting the proof of
Proposition~\ref{proposition:multi-scale-good-case} to our current
purposes we can demonstrate the
following.\begin{proposition}\label{proposition:multi-scale-bad-case-tilde}If
$\overline{X}\in\overline{\name_{k,1}^n}$ and $1\leq s\leq N_n$,
then
\begin{equation}\label{equation:left-eye}\mu\left(B(x_0,\alpha_0^k\cdot\max\nolimits_{x_0}(X))\right)
\geq g_{k,s}^n\left(\overline{X}\right)\geq\frac{1}{2}
\cdot\mu\left(B(x_0,\alpha_0^k\cdot\max\nolimits_{x_0}(X))\right).\end{equation}\end{proposition}

Analogous to the normalizing function $f_k^1$ of
Section~\ref{sizematters}, we then define
\begin{equation}\label{equation:lampy}f_k^n\left(\overline{X}\right)=
\prod_{s=1}^{N_n}g^n_{k,s}\left(\overline{X}\right),\end{equation}
and the corresponding ``augmentation'' of the measure
$\di\mu^{d+2}(X)$ (restricted to $X\in\name_{k,1}^n$)
\begin{equation}\label{ooohsono}\frac{\di\mu^{M_n}(\overline{X})}{f_k^n(\overline{X})}\Bigg|_{\overline{\name_{k,1}^n}}.\end{equation}
The following proposition is then established in parallel to
Proposition~\ref{proposition:bounding-sum-integral} (i.e., applying
the inequality of Lemma~\ref{lemma:X^s-inequality} together with a
direct relation between plain and augmented integration as in
equation~\eqref{equation:differential-identity-1}).

\begin{proposition}\label{proposition:multi-scale-integral-bound-tilde-1}If
$Q$ is a ball in $H$, then
$$\int_{\name_{k,1}^n(Q)}\frac{\pds^2_{x_0}(X)}{\diam(X)^{d(d+1)}}\di\mu^{d+2}(X)\leq 2^{n-1}
\cdot C_{\mathrm{p}}^{2\cdot
(n-1)}\sum_{s=1}^{N_n}\int_{\overline{\name_{k,1}^n(Q)}}\frac{\pds^2_{x_0}(X^s)}{\diam(X)^{d(d+1)}}
\frac{\fdi\mu^{M_n}\left(\overline{X}\right)}{f_k^n\left(\overline{X}\right)}.$$\end{proposition}

We now focus on applying the methods of
Section~\ref{section:proof-prop-poorly-d} to the individual terms on
the RHS of the inequality above.

\subsection{Generating Multiscale Discrete and Integral Inequalities}\label{subsection:dog-meat}

The individual integrals on the RHS of
Proposition~\ref{proposition:multi-scale-integral-bound-tilde-1} are
similar to the integral on the LHS of
Proposition~\ref{proposition:poorly-scaled-integration-d}, mainly
because the argument $X^s$, $1\leq s\leq N_n$, is a rake.  In
principle, one would like to  change variables in order to directly
apply Proposition~\ref{proposition:poorly-scaled-integration-d} to
these integrals. We avoid this for various reasons.  The most
immediate is because the region $\overline{\name_{k,1}^n}$, $1 < n
\leq d-1$, is relatively complicated, and any change of variables
would be further obstructed by the normalization of the function
$f_k^n$. Furthermore, for a rake $X^s$ that is poorly-scaled, we
must somehow use the small length scales present in $X^s$, and these
may not be accessible via such a change of variables.  As such, we
find it more straightforward to adapt the methods of
Section~\ref{section:proof-prop-poorly-d} to the individual terms on
the RHS of
Proposition~\ref{proposition:multi-scale-integral-bound-tilde-1}.

Throughout the rest of this section, we  fix $1 < n \leq d-1$,
$k\geq3$, and $1\leq s\leq N_n$.

\subsubsection{The Decomposition of $\overline{\name_{k,1}^n}$
According to $\SCale_{x_0}(X^s)$}We first decompose
$\overline{\name_{k,1}^n}$ according to the variety of the rakes
$\{X^s\}_{s=1}^{N_n}$,  the only meaningful variation being the
length of the handle and size of the function $\SCale_{x_0}(X^s)$.
Furthermore, following the argument leading to
equation~\eqref{equation:butterfinger} we may assume that the handle
occurs at the first coordinate of $X^s$.

The initial task is accounting for the discrepancies of scale for
$X^s$, that is we must decompose the set $\overline{\name_{k,1}^n}$
according to $\SCale_{x_0}(X^s)$.
 Let
\begin{equation*}\widehat{R}^s
=\left\{\overline{X}\in\overline{\name_{k,1}^n}:X^s\textup{ is
well-scaled at }x_0\right\}.\end{equation*} By construction, the
scaling of the simplices $X^s$ is no worse (actually slightly
better) than that of the original simplex $X$, and thus for $2\leq
k'\leq k-1$ we define the sets \begin{equation*}R_{k'}^s=
\left\{\overline{X}\in\overline{\name_{k,1}^n}:X^s\in
\name_{k',2}^1\right\}.\end{equation*} Furthermore, if $Q$ is a ball
in $H$, then we restrict those sets to $Q^{d+2}$ as before to obtain $\widehat{R}^s(Q)$ 
and $R_{k'}^s(Q)$, and we note the following set equality:
\begin{equation}\label{equation:non-disjoint-union}\overline{\name_{k,1}^n(Q)}=\widehat{R}^s(Q)\
\bigcup\, \bigcup_{k'=2}^{k-1}R_{k'}^s(Q).\end{equation}This
decomposition (which is not a partition because the sets
$R_{k'}^s(Q)$ may overlap) yields the following inequality
\begin{multline}\label{equation:sand}
\int_{\overline{\name_{k,1}^n(Q)}}
\frac{\pds_{x_0}^2(X^s)}{\diam(X)^{d(d+1)}}
\frac{\fdi\mu^{M_n}\left(\overline{X}\right)}{f_k^n\left(\overline{X}\right)}
\leq\\
%
%
\int_{\widehat{R}^s(Q)}\frac{\pds_{x_0}^2(X^s)}
{\diam(X)^{d(d+1)}}\frac{\fdi\mu^{M_n}\left(\overline{X}\right)}
{f_k^n\left(\overline{X}\right)}+
\sum_{k'=2}^{k-1}\int_{R_{k'}^s(Q)}\frac{\pds_{x_0}^2(X^s)}
{\diam(X)^{d(d+1)}}\frac{\fdi\mu^{M_n}\left(\overline{X}\right)}
{f_k^n\left(\overline{X}\right)}\,.
\end{multline}Note that the inequality of equation~\eqref{equation:sand} is analogous to the equality of equation~\eqref{multline:decomposed},
with one of the differences being the fact that we now only need to control a finite sum, rather than an infinite one.
The rest of our efforts focus on showing that each term on the RHS
above is ``small'' with respect to the quantity
$J_d^{\mathcal{D}}(\mu|_Q)$, that is, we can control them by
something that looks basically like
 $\alpha_0^{k\cdot d}\cdot J_d^{\mathcal{D}}(\mu|_Q)$.

The first term on the RHS of equation~(\ref{equation:sand}) can be
controlled via geometric multipoles.  In fact, via the well-scaling
of $X^s$ and the small length scales produced by the interpolation
procedure, we get such control  by chopping it
 according to the length scales of $X^s$, and then following the computations of Section~\ref{subsection:integration-d} on these pieces.
The result is the following proposition, whose proof appears in Appendix~\ref{app:wide-hat}
\begin{proposition}\label{proposition:wide-hat-dead}If $Q$ is a ball in $H$,
then there exists a constant $C_8=C_8(d,C_{\mu})$ such that
\be \label{eq:wide-hat-dead}
\int_{\widehat{R}^s(Q)}\frac{\pds_{x_0}^2(X^s)}
{\diam(X)^{d(d+1)}}\frac{\fdi\mu^{M_n}\left(\overline{X}\right)}
{f_k^n\left(\overline{X}\right)}\leq C_8\cdot\alpha_0^{k\cdot d\cdot
(d-n+2)}\cdot J_d^{\mathcal{D}}(\mu|_Q)\,.
\ee
\end{proposition}

We note that the coefficient $\alpha_0^{k\cdot d\cdot (d-n+2)}$ is
larger than   the coefficient $\alpha_0^{k\cdot d^2}$ of
Proposition~\ref{proposition:gerry-halwell}, and this is because the
$n$-handled simplex $X$ has fewer small length scales to work with.
In the previous case we had $d$  ``small'' edges helping us obtain a
sufficiently small coefficient in $k$. In the current case we have
less help because now there are only $d+1-n$ small edges, and the
coefficient is hence slightly larger.

The terms of the finite sum on the RHS of
equation~(\ref{equation:sand}) require further analysis before we
can establish the appropriate bounds,  and we develop this in the
rest of the section.

\subsection{Doubly-Augmented Elements and a Multiscale
Inequality}We fix $2\leq k'\leq k-1$ and concentrate on the set
$R_{k'}^s$. The integral over the augmented region $R_{k'}^s$ can be
exchanged for yet another augmented integral, but  we must  perform
two different types of augmentations. The first element is defined
as follows. If $\overline{X}\in R_{k'}^s$, then we take a
well-scaled piece for the rake $X^s\in \name_{k',2}^1,$
$$Y_{X^s}=(y_1,\ldots,y_{k'\cdot d})\,\in\,\prod_{q=1}^{k'\cdot d}A_{k'-\left\lceil\frac{q}{d}\right\rceil}
(x_0,\max\nolimits_{x_0}\left(X^s\right)),$$ and we form the
``doubly-augmented'' element
$$\overline{X}\times
Y_{X^s}=(x_0,\ldots,x_{d+1},z_1,\ldots,z_{N_n},y_1,\ldots,y_{k'\cdot
d})\in R_{k'}^s\times H^{k'\cdot d}.$$ The ``variable''
$\overline{X}\times Y_{X^s}$ is the underlying piece of information
that controls our process, but the actual simplex driving our
decisions is $X^s$.  As such, we introduce the symbol
$\overline{X}\times Y_{X^s}$ for clarity, and we focus our
development on another type of augmentation.  We clarify this as
follows.

 If $\overline{X}\in
R_{k'}^s$ and $Y_{X^s}$ is a well-scaled piece for $X^s$, then we
form the augmented element$$\underline{X^s}=X^s\times Y_{X^s}
,$$ which is an augmentation of the type introduced in
Subsection~\ref{subsection:multiscale}.  We then form the sequences
$$\widetilde{\Phi}_{k'}(\underline{X^s})=
\left\{\widetilde{X}^s_q\right\}_{q=0}^{k'\cdot d}\ \textup{ and }\
\Phi_{k'}(\underline{X^s})=\left\{X^s_q\right\}_{q=1}^{k'\cdot
d+1}$$ as given in Definitions~\ref{definition:tilde-phi}
and~\ref{definition:phi} of Subsection~\ref{subsection:multiscale},
and we note that these  depend only  on $\underline{X^s}$. Following
this line of reasoning we define the set
augmentation\begin{multline}\label{equation:mystery}\underline{R_{k'}^s}:
=\Bigg\{\overline{X}\times Y_{X^s}:\overline{X}\in R_{k'}^s\textup{
and the sequences }\widetilde{\Phi}_{k'}(\underline{X^s})\textup{
and }\Phi_{k'}(\underline{X^s})\textup{ satisfy the inequality
}\\\pds_{x_0}\left(\widetilde{X}^s_q\right)\leq
C_{\mathrm{p}}\cdot\left[\pds_{x_0}\left(X^s_{q+1}\right)+\pds_{x_0}\left(\widetilde{X}^s_{q+1}\right)\right]\textup{
for all }0\leq q< k'\cdot d\Bigg\}.\end{multline}We have the
following inequality which is a direct application of
Lemma~\ref{lemma:multiscale-toast}.\begin{lemma}\label{lemma:late-date}If
$\overline{X}\times Y_{X^s}\in\underline{R_{k'}^s}$, then the
well-scaled sequence
$\Phi_{k'}(\underline{X^s})=\left\{X^s_q\right\}_{q=1}^{k'\cdot
d+1}$ satisfies the inequality$$\pds^2_{x_0}(X^s)\leq(k'\cdot
d+1)\cdot C_{\mathrm{p}}^{2\cdot k'\cdot d}\ \sum_{q=1}^{k'\cdot
d+1}\pds^2_{x_0}\left(X^s_q\right).$$\end{lemma}

With these definitions, we formulate our ultimate multiscale
integral inequality whose proof is hardly different from the proof
of Proposition~\ref{proposition:multi-scale-integral-bound-tilde-1}.
\begin{lemma}\label{lemma:yet-another-sum}If $Q$ is a ball in $H$ then
\begin{multline*}\int_{R_{k'}^s(Q)}
\frac{\pds_{x_0}^2(X^s)}{\diam(X)^{d(d+1)}}
\frac{\fdi\mu^{M_n}\left(\overline{X}\right)}{f_k^n\left(\overline{X}\right)}\leq\\
C^{2k'\cdot d}_{\mathrm{p}}\cdot(k'\cdot d+1)\sum_{q=1}^{k'\cdot
d+1}\,\int_{\underline{R_{k'}^s(Q)}}
\frac{\pds_{x_0}^2(X^s_q)}{\diam(X)^{d(d+1)}}
\frac{\fdi\mu^{M_n+k'\cdot d}(\overline{X}\times
Y_{X^s})}{f_k^n\left(\overline{X}\right)\cdot
f_{k'}^1\left(\underline{X^s}\right)}.\end{multline*}\end{lemma}

The terms on the RHS of Lemma~\ref{lemma:yet-another-sum} can be
controlled by geometric multipoles as in the proofs of
Propositions~\ref{proposition:poorly-scaled-integration-d}
and~\ref{proposition:wide-hat-dead}. However, the computations here
have more information to take into account because we are dealing
with various length scales at the same time, i.e., those of  the
doubly-indexed well-scaled element $X^s_q$ as well as the original
simplex $X$. We perform these computations in
Appendix~\ref{app:over-under} and conclude the following bounds on
the terms on the RHS of Lemma~\ref{lemma:yet-another-sum} and thus
also Proposition~\ref{proposition:poorly-scaled-integration}.

\begin{proposition}\label{proposition:under-over-dead}If $Q$ is a ball in $H$ and $2\leq k'\leq k-1$, then
there exists a constant $C_9=C_9(d,C_{\mu})$ such that \be
\label{eq:wide-hat-dead2} \int_{\underline{R_{k'}^s(Q)}}
\frac{\pds_{x_0}^2(X^s_q)}{\diam(X)^{d(d+1)}}
\frac{\fdi\mu^{M_n+k'\cdot d}(\overline{X}\times
Y_{X^s})}{f_k^n\left(\overline{X}\right)\cdot
f_{k'}^d\left(\underline{X^s}\right)}\leq C_9\cdot\alpha_0^{k\cdot
d\cdot(d-n+1) }\cdot J_d^{\mathcal{D}}(\mu|_Q)\,. \ee
%
\end{proposition}

We remark that the coefficient  $\alpha_0^{k\cdot d(d-n+1)}$ is
slightly worse than the coefficient produced in analyzing the set
$\widehat{R}^s(Q)$, the reason being that the poor scaling of
$X^s\in R_{k'}^s(Q)$ deprives us (in a small way) of that extra
factor.  

\appendix
\section{Appendix}
\subsection{Proof of Proposition~\ref{proposition:psin-bound-deviations}}\label{app:bound-deviate}Proposition~\ref{proposition:psin-bound-deviations} follows from the
two observations:
\begin{equation}\label{equation:bound-before}
\pds_{x_0}(X)\leq\frac{2\cdot (d+1)}{\SCale_{x_0}(X)}\cdot
\frac{h(X)}{\diam(X)}\,,
\end{equation}
and \be \label{eq:height-deviation}h(X)\leq \sqrt{2} \cdot
\left\lceil\frac{d+1}{2}\right\rceil\cdot D_2(X,L), \ \ \textup{ for
any $d$-plane $L$.} \ee
The two equations follow from elementary geometric estimates as
follows.
\subsubsection{Proof of Equation~\eqref{equation:bound-before}}
\label{app:bound-before}

We first note that
$$
\max_{0\leq i\leq d+1}\,\mathrm{M}_d(X(i))\leq\, (d+1) \cdot
\max_{1\leq i\leq d+1}\,\mathrm{M}_d(X(i)).
$$
Indeed, if the maximum on the LHS of the above equation is obtained
at $1 \leq i \leq d+1$, then the above inequality is trivial. If on
the other hand this maximum is obtained at $i=0$, then the
inequality follows from the fact that the $d$-content of a face of a
($d+1$)-simplex is less than the sum of the $d$-contents of the
other faces (this is since the $d$-content does not increase under
projections and is subadditive on $\reals^d$).

Then, using the fact that the product of any height of a
$(d+1)$-simplex with the $d$-content of the opposite side is a
constant (proportional to the $(d+1)$-content of the simplex), we
obtain that
$$\min_{1\leq i \leq d+1}\,h_{x_i}(X)\cdot\max_{1\leq
i\leq d+1}\,\mathrm{M}_d(X(i))=h(X)\cdot\max_{0\leq i\leq
d+1}\,\mathrm{M}_d(X(i)).$$
Combining the last two equations we deduce the inequality
\begin{equation}\label{equation:inf-height}
\min_{1\leq i\leq d+1}\,h_{x_i}(X)\leq (d+1)\cdot h(X)\,.
\end{equation}
Next, by equation~(\ref{eq:elevationsine}),
Proposition~\ref{proposition:product-sine}, and also
equation~(\ref{equation:0-1-valued}) we obtain that
$$
\pds_{x_0}(X)\leq \min_{1\leq i\leq
d+1}\,\frac{h_{x_i}(X)}{\|x_i-x_0\|}\
\leq\frac{\displaystyle\min_{1\leq i\leq
d+1}\,h_{x_i}(X)}{\min\nolimits_{x_0}(X)}.
$$
Applying equation~(\ref{equation:inf-height}) to the RHS above, we
have that
$$
\pds_{x_0}(X) \leq(d+1)\cdot\frac{h(X)}{\min\nolimits_{x_0}(X)}.
$$
Finally, applying the definition of $\SCale_{x_0}(X)$ as well as the
bound: $ \diam(X)\leq2\cdot\max\nolimits_{x_0}(X) $ to the latter
equation establishes equation~(\ref{equation:bound-before}), and
consequently the current proposition.

\subsubsection{Proof of Equation~\eqref{eq:height-deviation}}
\label{app:control_height}
We may assume that $X$ is non-degenerate, because otherwise $h(X)=0$
and the bound holds trivially. Furthermore, since orthogonal
projection decreases distances and reduces dimension of subspaces,
we may assume that $\dim(H)=d+1$, in particular, $H = \RR^{d+1}$.

Our proof utilizes the comparability of the height of the simplex
represented by $X$ and its width $w(X)$, which is given by the
following infimum over all $d$-planes $L$:
\begin{equation}
\label{eq:def_w}w(X)=2\, \min_{L} \,\,\max_{x\,\in\, \textup{ the
convex hull of }X}\,\dist(x,L).
\end{equation}
Equivalently, the width $w(X)$  is the shortest distance between any
two parallel $d$-planes supporting the convex hull of $X$, i.e., the
convex hull is trapped between them and its boundary touches them.
Gritzman and Lassak~\cite[Lemma~3]{Gritzmann} established the
following bound on $h(X)$ in terms of $w(X)$:
\be \label{equation:height-to-width}
h(X) \leq \left\lceil\frac{d+1}{2}\right\rceil \cdot w(X).
\ee

Equation~\eqref{eq:height-deviation} thus follows from combining
equation~(\ref{equation:height-to-width}) with the following bound
on $w(X)$, which we verify below.
\begin{equation}\label{equation:saturday}w(X)\leq
\sqrt{2}\cdot D_2(X,L) \  \textup{ for an arbitrary }
d\textup{-plane } L\,,
\end{equation}

We verify equation~\eqref{equation:saturday}, and thus conclude
equation~\eqref{eq:height-deviation}, as follows. For a given
$d$-plane $L$, let $L_1$ and $L_2$ be the two unique translates of
$L$ supporting the simplex represented by $X$ and let $w_L(X)$
denote the distance between $L_1$ and $L_2$. Furthermore, let
$x_{L_1}$ be a vertex of $X$ contained in $L_1$ and $x_{L_2}$ be a
vertex contained in $L_2$. The $d$-planes $L_1$ and $L_2$ separate
$\RR^{d+1}$ into three regions, being the two disjoint half spaces
of $\RR^{d+1}$ and the intermediate region bounded by the $d$-planes
whose closure contains the simplex represented by $X$.

If  $L$ is contained in one of the disjoint half spaces described
above, then  we may assume that $L_1$ sits  between $L$ and $L_2$.
We establish equation~\eqref{equation:saturday} in this case as
follows:
$$
w(X)\leq
w_L(X)=\dist\left(x_{L_2},L_1\right)\leq\dist\left(x_{L_2},L\right)\leq
D_2(X,L).
$$

In the second case, where the plane $L$ is contained in the
intermediate region, we obtain equation~\eqref{equation:saturday} in
the following way.
\begin{multline*}
w(X)\leq w_L(X)=\dist\left(x_{L_1},L\right)+\dist
\left(x_{L_2},L\right)\leq\\\sqrt{2}\cdot\left(\dist^2\left(x_{L_1},L\right)+\dist^2
\left(x_{L_2},L\right)\right)^{1/2}\leq\sqrt{2}\cdot D_2(X,L).
\end{multline*}

\subsection{Proof of Lemma~\ref{lemma:set-partition}}\label{app:resolution-lemma}
Given the $n$-net $E_n$, let $\mathcal{B}_n'=\{B(x,4\cdot\alpha_0^n)\}_{x\in E_n}$.  We note that both $\mathcal{B}_n'$ and
$\frac{1}{4}\cdot\mathcal{B}_n'$ cover $\Supp$ since $E_n$ is an $n$-net.  We take $\mathcal{B}_n$ to be a subfamily
of $\mathcal{B}_n'$ such that $\frac{1}{4}\cdot\mathcal{B}_n$ is a maximal, mutually disjoint collection of balls.  In this case,
we note that $\mathcal{B}_n$ still provides a cover of $\Supp$ due to the maximality.

The idea is to categorize the elements
$B'\in\frac{1}{4}\cdot\left[\mathcal{B}'_n \setminus
\mathcal{B}_n\right] $ according to the first element of
$\frac{1}{4}\cdot\mathcal{B}_n$ they intersect, and then use this to
take the ``appropriate'' part of $B'$. Then, once this is done, for
each $j\in \Lambda_n$ the element $P_{n,j}$ is formed by adding
these appropriate pieces to the corresponding ball $\frac{1}{4}
\cdot B_{n,j}$. We clarify this as follows.

If
$\frac{1}{4}\cdot\left[\mathcal{B}'_n\setminus\mathcal{B}_n\right]=\emptyset$,
then we take the partition ${\cal
P}_n=\left\{P_{n,j}\right\}_{j\in\Lambda_m},$ where for fixed
$j\in\Lambda_n$
$$P_{n,j}=\Supp\cap\frac{1}{4}\cdot B_{n,j},\textup{ for }B_{n,j}\in\mathcal{B}_n.$$

We thus assume that
$\frac{1}{4}\cdot\left[\mathcal{B}'_n\setminus\mathcal{B}_n\right]\not=\emptyset$,
and we index the elements of $\frac{1}{4}\cdot\left[\mathcal{B}'_n
\setminus {\cal B}_n\right]$ by the set $\Omega_n = \{1,2,\ldots\}$,
which is either finite or ${\mathbb N}$, i.e.,
$\frac{1}{4}\cdot\left[\mathcal{B}'_n \setminus {\cal B}_n \right]=
\{B_m'\}_{m \in\, \Omega_n}$. From this set of balls, we then
recursively form the following sets. For $m=1$, let
$$\bar B'_1 =
B_1'\cap\left(\bigcup\,\frac{1}{4}\cdot\mathcal{B}_n\right)^c,$$
and for $m\geq 2$, let
$$
\bar B'_m= B_m'\bigcap
\left(\bigcup\,\frac{1}{4}\cdot\mathcal{B}_n\bigcup
\bigcup_{i=1}^{m-1}\bar B'_i\right)^c.
$$
Note that the elements of $\left\{\bar
B'_m\right\}_{m\in\,\Omega_n}$ are mutually disjoint, and that
$\Supp$ is covered by the collection of sets
$$\frac{1}{4}\cdot\mathcal{B}_n\,\bigcup\,\left\{\bar
B'_m\right\}_{m\in\,\Omega}.$$

Let the function $g_n:\frac{1}{4}\cdot\left[\mathcal{B}'_n\setminus
{\cal B}_n\right]\to\Lambda_n$ be defined as follows:
\begin{equation}\label{equation:coffee-ting}
g_n\left(B'\right)=\min \left\{j\in\Lambda_n: \frac{1}{4}\cdot
B_{n,j} \cap B' \not=\emptyset\right\}.
\end{equation}It follows from
the maximality of $\frac{1}{4}\cdot\mathcal{B}_n$ that for every $
B'\in \frac{1}{4}\cdot\mathcal{B}'_n$, there exists a $  B\in
\frac{1}{4}\cdot{\cal B}_n $ such that $ B\cap B'\ne\emptyset$.
Consequently, the   minimum of equation~(\ref{equation:coffee-ting})
is obtained at an element of $\Lambda_n$, and taking
\begin{equation}\label{equation:definition-P-m-j}
P_{n,j}= \Supp\,\bigcap\,\left( \frac{1}{4}\cdot B_{n,j}\ \bigcup\
\bigcup_{g_n\left(B_m'\right)=j}\bar B'_m  \right),
\end{equation}
we note that the sets $P_{n,j}$ are disjoint and cover $\Supp$. The
desired set inclusions follow from the definition of $P_{n,j}$ and
observing that $B'\subseteq\frac{3}{4} \cdot B_{n,j}$ for any
$B'\in\frac{1}{4}\cdot\left[\mathcal{B}'_n\setminus\mathcal{B}_n\right]$
such that $g_n\left(B'\right)=j$.

\subsection{Proof of Proposition~\ref{proposition:multi-scale-good-case}}\label{app:nolyn}
First, for any $X\in \name_{k,p}^1$, along with $Z\in[\Supp]^{d+2}$
such that $(Z)_0=(X)_0=x_0$, as well as $0<r\leq\diam (\Supp)$ and
$1\leq j\leq k\cdot d$, the following estimate holds
\begin{equation}\label{equation:two-term-annulus-thing}
\mu \left( B(x_0,r) \right)\geq\mu
\left( U_{C_{\mathrm{p}}}(Z,1,\overline{j+1})\bigcap
A_0(x_0,r)\right)
\geq\frac{1}{2}\cdot\mu\left( B(x_0,r)\right) >0.
\end{equation}In other words, there is a sufficient amount of
$\Supp$ in the annulus $A_0(x_0,r)$ satisfying the relaxed two-term
inequality defined by the set
$U_{C_{\mathrm{p}}}(Z,1,\overline{j+1})$. Furthermore, for such $X$
we have the inequality $\max\nolimits_{x_0}(X)\leq\diam (X)\leq\diam
(\Supp)$, and we note that the radius
\be \label{eq:radius_choice}
r=\alpha_0^{k-\left\lceil\frac{j}{d}\right\rceil}\cdot
\max\nolimits_{x_0}(X)\ee
is in the above range for all $1 \leq j \leq k \cdot d$, i.e., $0 <
r \leq \diam(\supp(\mu)).$  Substituting this choice of radius into
equation~\eqref{equation:two-term-annulus-thing} we obtain the
estimate
\begin{multline}\label{equation:miami_vice}
\mu\left(B(x_0,\alpha_0^{k-\left\lceil\frac{j}{d}\right\rceil}
\cdot\max\nolimits_{x_0}(X))\right)\geq
\mu\left(U_{C_{\mathrm{p}}}\left(Z,1,\overline{j+1}\right)\,\bigcap\,
A_{k-\left\lceil\frac{j}{d}\right\rceil}(x_0,\max\nolimits_{x_0}(X))\right)\geq\\
\frac{1}{2}\cdot\mu
\left(B(x_0,\alpha_0^{k-\left\lceil\frac{j}{d}\right\rceil}\cdot\max\nolimits_{x_0}(X))\right).
\end{multline}

Next, we will prove that for $\underline{X} \in\underline{\name_{k,p}^1}$ and
$\widetilde{X}_{q-1}$, $1\leq q\leq k\cdot d$,  of the sequence
$\widetilde{\Phi}_k(\underline{X})$:
\begin{equation}\label{equation:miami}
U_{C_{\mathrm{p}}}\left(\widetilde{X}_{q-1},1,\overline{q+1}\right)\,\bigcap\,
A_{k-\left\lceil\frac{q}{d}\right\rceil}(x_0,\max\nolimits_{x_0}(X))=
\pi_q\left(T_{q-1}^{-1}\left(x_0,\ldots,y_{q-1}\right)\right).
\end{equation}
Taking $Z=\widetilde{X}_{q-1}$ in
equation~\eqref{equation:miami_vice} and noting
equation~\eqref{equation:miami} establishes the proposition.

We start by proving
equation~(\ref{equation:two-term-annulus-thing}). The inequality of the LHS of equation~(\ref{equation:two-term-annulus-thing}) is
trivial, and to prove the inequality of
the RHS of
equation~(\ref{equation:two-term-annulus-thing}) we note that
\begin{multline}\label{equation:measure-equal-stuff}\mu\left(U_{C_{\mathrm{p}}}\left(Z,1,\overline{j+1}\right)\bigcap
A_0(x_0,r)\right)=\mu\left(U_{C_{\mathrm{p}}}\left(Z,1,\overline{j+1}\right)\bigcap
B(x_0,r)\right)+\\
\mu\left(A_0(x_0,r)\right)-\mu\left(\left[U_{C_{\mathrm{p}}}\left(Z,1,\overline{j+1}\right)\bigcap
B(x_0,r)\right]\bigcup A_0(x_0,r) \right).
\end{multline}
By formulating lower bounds for the terms on the RHS of the above
equation we can then establish the inequality on the RHS of
equation~(\ref{equation:two-term-annulus-thing}).

With the above assumptions on $X$, $Z$, and $r$, by
Proposition~\ref{proposition:concentration-inequality-1} we have the
inequality
\begin{equation}\label{equation:chicago}
\mu\left(U_{C_{\mathrm{p}}}\left(Z,1,\overline{j+1})\right)\cap
B(x_0,r)\right)\geq\frac{3}{4}\cdot\mu(B(x_0,r)).
\end{equation}
Furthermore, by the  $d$-regularity and the
constant $\alpha_0$ (equation~\eqref{equation:alpha-loose})
we have the inequality
\begin{equation}\label{equation:newyork}
\mu\left(A_0(x_0,r)\right) \geq \left(1-\alpha_0^d\cdot
C^2_\mu\right) \geq\frac{3}{4}\cdot\mu (B(x_0,r)).
\end{equation}
Noting the inclusion
$\left[U_{C_{\mathrm{p}}}\left(Z,1,\overline{j+1}\right)\bigcap
B(x_0,r)\right]\bigcup A_0(x_0,r)\subseteq B(x_0,r), $ we obtain
$$
\mu\left(\left[U_{C_{\mathrm{p}}}\left(Z,1,\overline{j+1}\right)\bigcap
B(x_0,r)\right]\bigcup A_0(x_0,r)\right)\leq \mu(B(x_0,r)).
$$
Finally, applying this and equations~(\ref{equation:chicago})
and~(\ref{equation:newyork}) to the RHS of
equation~(\ref{equation:measure-equal-stuff}) we obtain
\begin{equation}\label{equation:seattle}
\mu\left(U_{C_{\mathrm{p}}}\left(Z,1,\overline{j+1}\right)\bigcap
A_0(x_0,r)\right)\geq \frac{1}{2}\cdot\mu(B(x_0,r)),
\end{equation}
and thus conclude equation~\eqref{equation:two-term-annulus-thing}.

Next, for  a fixed
$(x_0,\ldots,y_{q-1})= T_{q-1}(\underline{X})$, we establish
equation~(\ref{equation:miami}) via the inclusion
\be \label{equation:inclusion_c_p}
U_{C_{\mathrm{p}}}\left(\widetilde{X}_{q-1},1,\overline{q+1}\right)\,\bigcap\,
A_{k-\left\lceil\frac{q}{d}\right\rceil}(x_0,\max\nolimits_{x_0}(X))\subseteq
\pi_q\left( T_{q-1}^{-1} \left(x_0,\ldots,y_{q-1}\right)\right). \ee
The opposite inclusion follows directly from the definitions of the
sets $U_{C_{\mathrm{p}}}(\widetilde{X}_{q-1},1,\overline{q+1})$, $1
\leq q \leq k\cdot d$, and $\underline{\name_{k,p}^1}$ (see
equations~(\ref{equation:two-term-set})
and~(\ref{equation:def-big-omega})).

Our approach to proving equation~\eqref{equation:inclusion_c_p} is
to fix $1\leq q\leq k\cdot d$ and take an arbitrary point
\be \label{eq:fix_y_q} y'_q\in
U_{C_{\mathrm{p}}}\left(\widetilde{X}_{q-1},1,\overline{q+1}\right)\,\bigcap\,
A_{k-\left\lceil\frac{q}{d}\right\rceil}(x_0,\max\nolimits_{x_0}(X))\,.
\ee
We then iteratively use the inequality of
equation~(\ref{equation:two-term-annulus-thing}) to construct an
element
$$\underline{X}'=(x_0,\ldots,y_{q-1},y'_q,\ldots,y'_{k\cdot d})\in
T_{q-1}^{-1} \left(x_0,\ldots,y_{q-1}\right),$$
so that $y_q'=\pi_q(\underline{X}')\in
\pi_q\left(T_q\left(\pi_{q-1}^{-1}(\pi_{q-1}(\underline{X}))\right)\right)$.

Fixing $1\leq q\leq k\cdot d$ and $y'_q$ satisfying
equation~\eqref{eq:fix_y_q}, we recursively form the sequence
$\{y'_i\}_{i=q+1}^{k \cdot d}$ together with additional elements of
an auxiliary sequence $\left\{\widetilde{X}'_i\right\}_{i=q}^{k
\cdot d}$ as follows. First we initialize the auxiliary sequence by
defining
$$
\widetilde{X}'_{q}=\widetilde{X}_{q-1}(y'_q,\overline{q+1})\,.
$$
Next, given $q+1 \leq i \leq k \cdot d$ and assuming that
$\{y'_i\}_{i=q}^{i-1}$ and
$\left\{\widetilde{X}'_i\right\}_{i=q}^{i-1}$ have already been
defined, we fix arbitrarily
\begin{equation} \label{equation:splinter1} y'_i\in
U_{C_{\mathrm{p}}}(\widetilde{X}_{i-1}^\prime,1,\overline{i+1})\cap
A_{k-\left\lceil\frac{i}{d}\right\rceil}(x_0,\max\nolimits_{x_0}(X))\,,\end{equation}
and form
\begin{equation} \nonumber \label{equation:splinter2}
\widetilde{X}_i'=\widetilde{X}_{i-1}'(y'_i,\overline{i+1})\,.
\end{equation}
This procedure is well defined since
equation~(\ref{equation:two-term-annulus-thing}) implies that for
each $q+1\leq j\leq k\cdot d$:
$$\mu\left(U_{C_{\mathrm{p}}}\left(\widetilde{X}'_{j-1},1,\overline{j+1}\right)\,\bigcap\,
A_{k-\left\lceil\frac{j}{d}\right\rceil}(x_0,\max\nolimits_{x_0}(X))\right)>0\,.$$

Finally, forming
$$\underline{X}'=(x_0,\ldots,y_{q-1},y_q',\ldots,y'_{k\cdot
d})=X\times Y_X' \in H^{(k+1)\cdot d+2}\,,$$
we note that $Y_X'$ is a well-scaled element for $X$ and the
elements of the sequences
$\widetilde{\Phi}_k(\underline{X}')=\left\{\widetilde{X}_q\right\}_{q=0}^{k\cdot
d}$ and $\Phi_k(\underline{X}')=\{X_q\}_{q=1}^{k\cdot d+1}$ satisfy
the inequality
$$\pds_{x_0}(\widetilde{X}_{j-1})\leq C_{\mathrm{p}}\left(\pds_{x_0}\left({X}_j\right)+\pds_{x_0}\left(\widetilde{X}_j\right)\right).$$
Therefore, $\underline{X}'\in\underline{\name_{k,p}^1}.$
Furthermore, $ T_{q-1}(\underline{X}')=(x_0,\ldots,y_{q-1}),$ and
thus
$$\underline{X}'\in \pi_{q-1}^{-1}(\pi_{q-1}(x_0,\ldots,y_{q-1})).$$
Since $\pi_q(\underline{X}')=y'_q$,
equation~\eqref{equation:inclusion_c_p} and consequently
equation~(\ref{equation:miami}) are now established.

\subsection{Proof of Proposition~\ref{proposition:wide-hat-dead}}\label{app:wide-hat}
For $m\geq m(Q)$ we define
\begin{equation}\label{equation:again-s-hat-m}\widehat{R}^s(m)(Q)=
\left\{\overline{X}\in\widehat{R}^s(Q):\max\nolimits_{x_0}(X)\in(\alpha_0^{m+1},\alpha_0^m]\right\},\end{equation}and this
gives the following decomposition of the integral
\begin{equation}\label{equation:land-ho}\int_{\widehat{R}^s(Q)}\frac{\pds^2_{x_0}\left(X^s\right)}{\diam(X)^{d(d+1)}}
\frac{\fdi\mu^{M_n}\left(\overline{X}\right)}{f_k^n\left(\overline{X}\right)}
=\sum_{m\geq
m(Q)}\int_{\widehat{R}^s(m)(Q)}\frac{\pds^2_{x_0}\left(X^s\right)}{\diam(X)^{d(d+1)}}
\frac{\fdi\mu^{M_n}\left(\overline{X}\right)}{f_k^n\left(\overline{X}\right)}.\end{equation}
Fixing $m\geq m(Q)$, we partition $\widehat{R}^s(m)(Q)$ according to
 $\max\nolimits_{x_0}\left(X^s\right)$ in the following way. 

If $\overline{X}\in\widehat{R}^s(m)$, then
Lemma~\ref{lemma:poor-scaled-property-size} 
%
%
and the well scaling of $X^s$ imply
\begin{equation}\label{equation:max-s-where}\max\nolimits_{x_0}(X^s)\in\left(\alpha_0^k\cdot
\max\nolimits_{x_0}(X),\alpha_0^{k-3}\cdot\max\nolimits_{x_0}(X)\right]\,.\end{equation}
According to equations~(\ref{equation:again-s-hat-m})
and~(\ref{equation:max-s-where}), we use the following scale
exponent:
\begin{equation}\label{equation:chirp}\scl(m,k)=m+k-3\geq m(Q).\end{equation}

Since
$\left\{P_{\scl(m,k),j}\right\}_{j\in\Lambda_{\scl(m,k)}(Q)}$ covers
$Q\cap\Supp$, we cover $\widehat{R}^s(m)(Q)$ by
\begin{equation}\label{equation:yet-another-partition-m-s}\widehat{P}_{\scl(m,k),j}=
\left\{\overline{X}\in\widehat{R}^s(m):x_0\in
P_{\scl(m,k),j}\right\},\textup{ for all
}j\in\Lambda_{\scl(m,k)(Q)}.\end{equation}
Hence we obtain the inequality
\begin{multline}\label{equation:gotta-finish}\int_{\widehat{R}^s(Q)}\frac{\pds_{x_0}^2(X^s)}
{\diam(X)^{d(d+1)}}\frac{\fdi\mu^{M_n}\left(\overline{X}\right)}{f_k^n\left(\overline{X}\right)}\leq\\\sum_{m\geq
m(Q)}\sum_{j\in\,\Lambda_{\scl(m,k)}(Q)}
\int_{\widehat{P}_{\scl(m,k),j}}\frac{\pds_{x_0}^2(X^s)}
{\diam(X)^{d(d+1)}}\frac{\fdi\mu^{M_n}\left(\overline{X}\right)}{f_k^n\left(\overline{X}\right)}.\end{multline}

Now we fix $m\geq m(Q)$ and $j\in\,\Lambda_{\scl(m,k)}(Q)$ and concentrate on the individual terms of the RHS of equation~\eqref{equation:gotta-finish}.
Obtaining the ``proper'' bounds here will clearly imply the proposition.

We fix an arbitrary $d$-plane $L$. If
$\overline{X}\in\widehat{P}_{\scl(m,k),j}$, then equations~\eqref{equation:again-s-hat-m} and~\eqref{equation:max-s-where} imply
$$\diam(X^s)\geq\alpha_0^{m+k+1}=\frac{\alpha_0^4}{8}\cdot\diam(B_{\scl(m,k),j}).$$Applying this
and Proposition~\ref{proposition:psin-bound-deviations} we get that
\begin{multline}\label{equation:major-thing-1}\int_{\widehat{P}_{\scl(m,k),j}}\frac{\pds_{x_0}^2(X^s)}
{\diam(X)^{d(d+1)}}\frac{\fdi\mu^{M_n}\left(\overline{X}\right)}{f_k^n\left(\overline{X}\right)}\leq\\\frac{2^7\cdot(d+1)^2\cdot(d+2)^2}
{\alpha_0^{14}}\,\int_{\widehat{P}_{\scl(m,k),j}}\frac{D_2^2\left(X^s,L\right)}
{\diam^2(B_{\scl(m,k),j})}\frac{\fdi\mu^{M_n}\left(\overline{X}\right)}{f_k^n\left(\overline{X}\right)\cdot\diam(X)^{d(d+1)}}.\end{multline}
Hence we focus
on the individual terms of
$\displaystyle\frac{D^2_2\left(X^s,L\right)}{\diam\left(B_{\scl(m,k),j}\right)}.$

We arbitrarily fix
$0\leq t\leq d+1$ and note the cases for the possible values of
$\left(X^s\right)_t\,$.  Per
Lemma~\ref{lemma:poor-scaled-property-size} and the construction of the elements $X^s$ we have the following
cases.

\vskip .2cm \noindent Case 1: $t\in\{0,1,d+1\}$.  In this case we note that by our construction $(X^s)_0 = x_0,$ $(X^s)_1=x_{i_s}$ for $1\leq i_s\leq n$, and
 $(X^s)_{d+1}=x_i$, where $n+1 \leq i\leq
d+1$.
\vskip .3cm \noindent Case 2:  $2\leq t\leq d$.  In this case, again by the construction we have that $(X^s)_t = z_\ell$, for exactly $d-1$ distinct indices $\ell$, $1\leq
\ell\leq 2^{n-2}+\left\lceil\frac{s}{2}\right\rceil-1.$

Assume Case~1. Per Fubini's, the
corresponding terms on the RHS of
equation~(\ref{equation:major-thing-1}) are
\begin{multline}\label{equation:basic-first-3}\int_{\widehat{P}_{\scl(m,k),j}}\frac{\dist^2\left(\left(X^s\right)_t,L\right)}
{\diam^2\left(B_{\scl(m,k),j}\right)}\frac{\fdi\mu^{M_n}\left(\overline{X}\right)}
{f_k^n\left(\overline{X}\right)\cdot\diam(X)^{d(d+1)}}=\\\int_{T_0\left(\widehat{P}_{\scl(m,k),j}\right)}\frac{\dist^2\left(\left(X^s\right)_t,L\right)}
{\diam^2\left(B_{\scl(m,k),j}\right)}\frac{\fdi\mu^{d+2}(X)}{\diam(X)^{d(d+1)}}.\end{multline}Equation~(\ref{equation:max-s-where})
and Lemma~\ref{lemma:poor-scaled-property-size} imply the set
inclusion\begin{equation}\label{equation:set-inclusion-doozy}
T_0\left(\widehat{P}_{\scl(m,k),j}\right) \subseteq
\bigcup_{x_0\in\, P_{\scl(m,k),j}}
\{x_0\}\times\prod_{i=1}^{d+1}B(x_0,\alpha_0^{p_i}),\end{equation}
where $$p_i=\begin{cases}m+k-3,&\textup{ if } i=i_s, (\textup{See
Case $1$ above});
\\m+k,&\textup{ if }n+1\leq i\leq d+1;\\
m,&\textup{ otherwise.}\end{cases}$$
Equation~\eqref{equation:set-inclusion-doozy} then yields the
following bound on the RHS of
equation~(\ref{equation:basic-first-3})
\begin{multline}\label{equation:basic-second-3}\int_{T_0\left(\widehat{P}_{\scl(m,k),j}\right)}\frac{\dist^2\left(\left(X^s\right)_t,L\right)}
{\diam^2\left(B_{}\right)}\frac{\fdi\mu^{d+2}(X)}{\diam(X)^{d(d+1)}}\leq\\\frac{1}{\alpha_0^{d(d+1)}}\,\int_{P_{\scl(m,k),j}}
\int_{\prod_{i=1}^{d+1}B(x_0,\alpha_0^{p_i})}
\frac{\dist^2\left(\left(X^s\right)_t,L\right)}
{\diam^2\left(B_{\scl(m,k),j}\right)}\frac{\fdi\mu(x_{d+1})\cdots\fdi\mu(x_0)}{[\alpha_0^m]^{
d(d+1)}}\,.\end{multline}Applying the usual calculations (see the
proofs of Propositions~\ref{proposition:journe}
and~\ref{proposition:poorly-scaled-integration-d}) to the RHS of
equation~(\ref{equation:basic-second-3}), we obtain the following
bound on the RHS of equation~(\ref{equation:basic-first-3})
\begin{equation}\label{equation:easy-maximal}\frac{3^d\cdot
C_{\mu}^{d+1}\cdot\alpha_0^{k\cdot d\cdot
(d-n+2)}}{\alpha_0^{d(d+1)+3\cdot
d}}\cdot\beta^2_2\left(B_{\scl(m,k),j},L\right)\cdot
\mu\left(B_{\scl(m,k),j}\right).\end{equation}

For Case~2,  we iterate the integral to obtain the equality
%
\begin{multline}\label{equation:pour-it}\int_{\widehat{P}_{\scl(m,k),j}}
\frac{\dist^2(z_\ell,L)}{\diam^2(B_{\scl(m,k),j})}
\frac{\fdi\mu^{M_n}\left(\overline{X}\right)}{f_k^n\left(\overline{X}\right)}=\\
\int_{
T_0\left(\widehat{P}_{\scl(m,k),j}\right)}\left(\int_{\left\{Z_X:X\times
Z_X\in\widehat{P}_{\scl(m,k),j}\right\}}\frac{\dist^2(z_\ell,L)}
{\diam^2(B_{\scl(m,k),j})}\frac{\fdi\mu^{N_n}(Z_X)}{f_k^n\left(\overline{X}\right)}\right)\frac{\fdi\mu^{d+2}(X)}{\diam(X)^{d(d+1)}}.\end{multline}

To
uniformly bound for the inner integral, it is
sufficient to uniformly bound
$$\int_{\pi_\ell\left(T^{-1}_{\ell-1}\left(x_0,\ldots,z_{\ell-1}\right)\right)}\left(\frac{\dist\left(z_{\ell},L\right)}
{\diam\left(B_{\scl(m,q),j}\right)}\right)^2\frac{\fdi\mu(z_\ell)}
{\mu\left(\pi_\ell\left(T^{-1}_{\ell-1}\left(x_0,\ldots,z_{\ell-1}\right)\right)\right)}.$$
Proposition~\ref{proposition:multi-scale-bad-case-tilde},
equation~(\ref{equation:again-s-hat-m}), and the $d$-regularity of $\mu$ imply
\begin{equation}\label{equation:lower-blond}\mu\left(\pi_{\ell}\left(T_{\ell-1}^{-1}\left(x_0,\ldots,z_{\ell-1}\right)\right)\right)
\geq\frac{1}{2\cdot
C_{\mu}^2}\cdot\left(\frac{\alpha_0^4}{4}\right)^d\cdot\mu\left(B_{\scl(m,k),j}\right).\end{equation}
Furthermore,
$\pi_\ell\left(T^{-1}_{\ell-1}\left(x_0,\ldots,z_{\ell-1}\right)\right)\subseteq
B_{\scl(m,k),j}$, and thus
\begin{multline}\int_{\pi_{\ell}\left(T_{\ell-1}^{-1}\left(x_0,\ldots,z_{\ell-1}\right)\right)}
\frac{\dist^2\left(z_{\ell},L\right)}{\diam^2(B_{\scl(m,k),j})}
\frac{\fdi\mu(z_{\ell_i})}{\mu\left(\pi_{\ell}\left(T_{\ell-1}^{-1}(x_0,\ldots,z_{\ell-1})\right)\right)}\leq\\\frac{2\cdot4^d\cdot
C_{\mu}^2}{\alpha_0^{4\cdot
d}}\cdot\beta_2^2(B_{\scl(m,k),j},L).\end{multline}
Therefore the inner integral on the RHS of
equation~(\ref{equation:pour-it}) is bounded by
\begin{equation}\int_{\left\{Z_X:X\times
Z_X\in\widehat{P}_{\scl(m,k),j}\right\}}\frac{\dist^2\left(z_{\ell_i},L\right)}
{\diam^2(B_{\scl(m,k),j})}\frac{\fdi\mu^{N_n}(Z_X)}{f_k^n\left(\overline{X}\right)}\leq\frac{2\cdot4^d\cdot
C_{\mu}^2}{\alpha_0^{4\cdot
d}}\cdot\beta_2^2(B_{\scl(m,k),j}).\end{equation}Applying equation~\eqref{equation:set-inclusion-doozy} gives the following upper bound for the RHS of
equation~(\ref{equation:pour-it}):
\begin{equation}\label{equation:deadly-do-wrong}\frac{2\cdot4^d\cdot
C_{\mu}^{d+3}}{\alpha_0^{d(d+1)+6\cdot d}}\cdot\alpha_0^{k\cdot
d\cdot(d-n+2)}\cdot\beta_2^2(B_{\scl(m,k),j},L)\cdot\mu\left(B_{\scl(m,k),j}\right).\end{equation}
Applying equations~(\ref{equation:easy-maximal})
and~(\ref{equation:deadly-do-wrong}) to the RHS of
equation~(\ref{equation:major-thing-1}) and then taking the infimum
over all $d$-planes $L$ establishes  the proposition.

\subsection{Proof of Proposition~\ref{proposition:under-over-dead}}\label{app:over-under}

Fixing $2\leq k'\leq k-1$, $1\leq q\leq
k'\cdot d+1$, and $m\geq m(Q)$, we define
\begin{equation}\underline{R_{k'}^s}(m)(Q)=\left\{\overline{X}\times Y_{X^s}\in
\underline{R_{k'}^s(Q)}:\max\nolimits_{x_0}(X)\in(\alpha_0^{m+1},\alpha_0^m]\right\}.\end{equation}This
gives the following decomposition of the integral on the LHS of
equation~\eqref{eq:wide-hat-dead}
\begin{multline}\label{equation:lappy}\int_{\underline{R_{k'}^s}(Q)}
\frac{\pds_{x_0}^2(X^s_q)}{\diam(X)^{d(d+1)}}
\frac{\fdi\mu^{M_n+k'\cdot d}(\overline{X}\times
Y_{X^s})}{f_k^n\left(\overline{X}\right)\cdot
f_{k'}^d\left(\underline{X^s}\right)}=\\\sum_{m\geq
m(Q)}\int_{\underline{R_{k'}^s}(m)(Q)}
\frac{\pds_{x_0}^2(X^s_q)}{\diam(X)^{d(d+1)}}
\frac{\fdi\mu^{M_n+k'\cdot d}(\overline{X}\times
Y_{X^s})}{f_k^n\left(\overline{X}\right)\cdot
f_{k'}^d\left(\underline{X^s}\right)}.\end{multline}

Fixing $m\geq m(Q)$, we partition $\underline{R_{k'}^s}(m)(Q)$
according to $\max\nolimits_{x_0}\left(X^s_q\right)$ in the following way.
By Lemma~\ref{lemma:poor-scaled-property-size}, for
$\overline{X}\in\underline{R_{k'}^s}$ we have that
\begin{equation}\label{equation:alcantar}\alpha_0^{m+k-k'+2}\leq\max\nolimits_{x_0}(X^s)
<\alpha_0^{m+k-k'-2}.\end{equation} If $k'=k-1$, then
the upper bound on the RHS of equation~(\ref{equation:alcantar}) is
too large since we  always have that
$\max\nolimits_{x_0}(X^s)\leq\alpha_0^m<\alpha_0^{m-1}$. So, we
amend equation~(\ref{equation:alcantar}) in the
following way. Let
\begin{equation}e(m,k')=\begin{cases}m+k-k'-2,&\textup{ if }2\leq
k'\leq k-2;\\\quad\quad m,&\textup{ if
}k'=k-1.\end{cases}\end{equation} We note that $e(m,k')\geq m$ and
\begin{equation}\label{equation:kessler}\alpha_0^{e(m,k')+4}\leq\max\nolimits_{x_0}(X^s)
\leq\alpha_0^{e(m,k')}.\end{equation}Now, combining
Lemma~\ref{lemma:index-function-no-tilde} and
equation~(\ref{equation:kessler}) we have the following estimate
\begin{equation}\label{equation:raposo}\max\nolimits_{x_0}\left(X^s_q\right)\in\begin{cases}
\displaystyle\left(\alpha_0^{k'+e(m,k')-\left\lceil\frac{q}{d}\right\rceil+5},\alpha_0^{k'+e(m,k')-\left\lceil\frac{q}{d}\right\rceil}\right],&
\textup{ if }1\leq q\leq k'\cdot
d;\\\displaystyle\left(\alpha_0^{e(m,k')+1},\alpha_0^{e(m,k')}\right],&
\textup{ if }q=k'\cdot d+1.\end{cases}\end{equation}
Hence we define the scale exponent as follows
\begin{equation}\label{equation:go-skate}\scl(m,k',q)=\scl(m,k,k',q)=
\begin{cases}\displaystyle k'+e(m,k')-\left\lceil\frac{q}{d}\right\rceil,
&\textup{ if }1\leq q\leq k'\cdot d;\\e(m,k'), &\textup{ if
}q=k'\cdot d+1.\end{cases}\end{equation}
We note that the scale exponent is independent of $s$, and
furthermore, we have the inequality
\begin{equation}\label{equation:life-saver}\scl(m,k',q)\,\geq\, e(m,k'),\textup{ for
all }1\leq q\leq k'\cdot d+1.\end{equation}

Next, $\left\{P_{\scl(m,k',q),j}\right\}_{j\in\Lambda_{\scl(m,k',q)}(Q)}$
covers $Q\cap\Supp$, and  so we cover
$\underline{R_{k'}^s}(m)(Q)$ by
\begin{equation}\label{equation:dale}\underline{P_{\scl(m,k',q),j}}=\left\{\overline{X}
\times Y_{X^s}\in\underline{R_{k'}^s}(m):x_0\in
P_{\scl(m,k',q),j}\right\}, \textup{ for
}j\in\Lambda_{\scl(m,k',q)(Q)}.\end{equation}Letting $m\geq m(Q)$
and $j$ vary we obtain the inequality
\begin{multline}\label{equation:chu}\int_{\underline{R_{k'}^s}} \frac{\pds_{x_0}^2(X^s_q)}{\diam(X)^{d(d+1)}}
\frac{\fdi\mu^{M_n+k'\cdot d}(\overline{X}\times
Y_{X^s})}{f_k^n\left(\overline{X}\right)\cdot
f_{k'}^d\left(\underline{X^s}\right)}\leq\\\sum_{m\geq
m(Q)}\sum_{j\in\Lambda_{\scl(m,k',q)}(Q)}\int_{\underline{P_{\scl(m,k',q),j}}}
\frac{\pds_{x_0}^2(X^s_q)}{\diam(X)^{d(d+1)}}
\frac{\fdi\mu^{M_n+k'\cdot d}(\overline{X}\times
Y_{X^s})}{f_k^n\left(\overline{X}\right)\cdot
f_{k'}^d\left(\underline{X^s}\right)}.\end{multline} Fixing $m\geq
m(Q)$ and $j\in\Lambda_{\scl(m,k',q)}(Q)$,  we now concentrate on
the terms of the RHS of equation~\eqref{equation:chu}.  We note that
the ``proper'' control on these terms implies the proposition.

Let $L$ be an arbitrary  $d$-plane. If $\overline{X}\times
Y_{X^s}\in\underline{P_{\scl(m,k',q),j}}$, then by
equations~(\ref{equation:raposo}) and~(\ref{equation:go-skate})
$$\diam\left({X^s}_q\right)\geq \alpha_0^{\scl(m,k',q)+5}=\frac{\alpha_0^5}{8}\cdot\diam\left(B_{\scl(m,k',q),j}\right).$$

Applying this and
Proposition~\ref{proposition:psin-bound-deviations} we obtain the
inequality
\begin{multline}\label{equation:chavez}\int_{\underline{P_{\scl(m,k',q),j}}}
\pds_{x_0}^2\left(X^s_q\right) \frac{\fdi\mu^{M_n+k'\cdot
d}(\overline{X}\times Y_{X^s})}{\diam(X)^{d(d+1)}\cdot
f_k^n\left(\overline{X}\right)\cdot
f_{k'}^d\left(\underline{X^s}\right)}\\\leq\frac{2^7
\cdot(d+1)^2\cdot(d+2)^2}{\alpha_0^{16}}\int_{\underline{P_{\scl(m,k',q),j}}}
\frac{D^2_2\left(X^s_q,L\right)}
{\diam^2\left(B_{\scl(m,k',q),j}\right)} \frac{\fdi\mu^{M_n+k'\cdot
d}(\overline{X}\times Y_{X^s})}{\diam(X)^{d(d+1)}\cdot
f_k^n\left(\overline{X}\right)\cdot
f_{k'}^d\left(\underline{X^s}\right)}.\end{multline}

To bound the RHS of equation~(\ref{equation:chavez}), we focus on
the individual terms of
$\displaystyle\frac{D^2_2\left(X^s_q,L\right)}{\diam^2\left(B_{\scl(m,k',q)}\right)}.$
Fixing $0\leq t\leq d+1$, per
equations~\eqref{equation:X-tilde}-\eqref{equation:def-X-q} and Lemma~\ref{lemma:poor-scaled-property-size},
we have the following
cases:

\vskip .2cm \noindent Case 1: $\left(X^s_q\right)_t = x_0$. In this
case $q$ has no restriction, that is, $1\leq q\leq k'\cdot d$.
\vskip .3cm \noindent Case 2: $\left(X^s_q\right)_t = x_{i_s}$. In this case $q=k'\cdot d+1$ because $x_{i_s}$ is the handle of $X^s$, and
only the last element of the sequence has this handle. Hence
$\scl(m,k',q)=k'+e(m,k')$ by
equation~\eqref{equation:go-skate}.
\vskip .3cm \noindent Case 3: $\left(X^s_q\right)_t = x_i$, for
$n+1\leq i\leq d+1$. In this case $1\leq q\leq d$ since only the first $d$ elements of the sequence contain the tines  $X^s$.
\vskip .3cm \noindent Case 4: $\left(X^s_q\right)_t = z_\ell$, where
$1\leq \ell\leq 2^{n-2}+\left\lceil\frac{s}{2}\right\rceil-1.$ We again have $1\leq q\leq d$ as in case 3.
\vskip .3cm \noindent Case 5: $\left(X^s_q\right)_t=y_{i}$, where
$1\leq i \leq k'\cdot d.$ In this case, for each $1\leq q\leq
k'\cdot d+1$, we have the following restriction on the quantity
$\left\lceil\frac{i}{d}\right\rceil,$ just as in
equation~(\ref{equation:restrict_q_l}):\begin{equation}\label{equation:coldy-feet}
\max\left\{1,\,\left\lceil\frac{q}{d}\right\rceil-1\right\}
\leq\left\lceil\frac{i}{d}\right\rceil\leq\left\lceil\frac{q}{d}\right\rceil.\end{equation}

For the first three cases, per Fubini's the corresponding terms of
equation~(\ref{equation:chavez}) reduce
to\begin{equation}\label{equation:loonatic}\int_{T_0\left(\underline{P_{\scl(m,k',q),j}}\right)}\frac{\dist^2\left(\left(X^s_q\right)_t,L\right)}
{\diam^2\left(B_{\scl(m,k',q),j}\right)}\frac{\fdi\mu^{d+2}(X)}{\diam(X)^{d(d+1)}}.\end{equation}
For the set
$\underline{P_{\scl(m,k',q),j}},$
equation~(\ref{equation:alcantar}) and
Lemma~\ref{lemma:poor-scaled-property-size} imply the
set inclusion
\begin{equation}\label{equation:agro-inclusion}T_0\left(\underline{P_{\scl(m,k',q),j}}\right)\subseteq\bigcup_{x_0\in
P_{\scl(m,k',q),j}}
\{x_0\}\times\prod_{i=1}^{d+2}B(x_0,\alpha_0^{p_i}),\end{equation}where\begin{equation}p_i=
\begin{cases}\quad m+k,&\textup{ if }n+1\leq i\leq d+1;\\e(m,k'),&\textup{ if }i=i_s;\\\quad\quad m,&\textup{ otherwise}.\end{cases}\end{equation}
Via the usual computations and noting the values
of $\scl(m,k',q)$ and $e(m,k')$,  the RHS of
equation~(\ref{equation:loonatic}) has the bound
\begin{equation}\label{equation:maximal-honest}\frac{3^d\cdot C_{\mu}^{d+1}}{\alpha_0^{d(d+1)+3\cdot
d}}\cdot\alpha_0^{k\cdot d\cdot
(d-n+1)}\cdot\beta_2^2\left(B_{\scl(m,k',q),j},L\right)
\cdot\mu\left(B_{\scl(m,k',q),j}\right).\end{equation}

Assume Case~4. Fix $1\leq \ell\leq
2^{n-2}+\left\lceil\frac{s}{2}\right\rceil-1$ and iterate the
integral to obtain
\begin{multline}\label{equation:rutger}\int_{\underline{P_{\scl(m,k',q),j}}}\left(\frac{\dist\left(z_{\ell},L\right)}
{\diam\left(B_{\scl(m,k',q),j}\right)}\right)^2\frac{\fdi\mu^{M_n+k'\cdot
d}\left(\overline{X}\times Y_{X^s}\right)}{\diam(X)^{d(d+1)}\cdot
f_k^n\left(\overline{X}\right)\cdot
f_{k'}^d\left(\underline{X^s}\right)}=\\\int_{T_0\left(\underline{P_{\scl(m,k',q),j}}\right)}\left(\int_{\left\{Z_X:X\times
Z_X\in R_{k'}^s\right\}}\left(\frac{\dist\left(z_{\ell},L\right)}
{\diam\left(B_{\scl(m,k',q),j}\right)}\right)^2\frac{\fdi\mu^{N_n}\left(Z_X\right)}
{f_k^n(\overline{X})}\right)\frac{\fdi\mu^{d+2}(X)}{\diam(X)^{d(d+1)}}.\end{multline}
Again we want to control the inner integral, and this clearly reduces to  controlling  the quantity
\begin{equation}\label{equation:llorar}\int_{\pi_{\ell}\left(T^{-1}_{\ell-1}(x_0,\ldots,z_{\ell-1})\right)}
\left(\frac{\dist\left(z_{\ell},L\right)}
{\diam\left(B_{\scl(m,k',q),j}\right)}\right)^2\frac{\fdi\mu(z_{\ell})}
{\mu\left(\pi_\ell\left(T^{-1}_{\ell-1}(x_0,\ldots,z_{\ell-1})\right)\right)}.
\end{equation}
To do this, we use the
definitions of  $\scl(m,k',q)$ and $e(m,k')$ obtaining
$$m+k-3\leq\scl(m,k',q)\leq m+k-2.$$Hence,
Proposition~\ref{proposition:multi-scale-bad-case-tilde} and the
$d$-regularity of $\mu$ imply the following  for all
$1\leq \ell\leq2^{n-2}+\left\lceil\frac{s}{2}\right\rceil-1$:
\be\label{equation:futbal}\mu\left(\pi_{\ell}\left(T^{-1}_{\ell-1}(x_0,\ldots,z_{\ell-1})\right)\right)
\geq
\frac{\alpha_0^{4\cdot d}}{2\cdot4^d\cdot
C_{\mu}^2}\cdot\mu\left(B_{\scl(m,k',q),j}\right).\ee
Moreover, for fixed $(x_0,\ldots,z_{\ell-1})\in
T_{\ell-1}\left(\underline{P_{\scl(m,k',q),j}}\right),$ we have that
\begin{equation}\label{equation:football}\pi_{\ell}\left(T^{-1}_{\ell-1}(x_0,\ldots,z_{\ell-1})\right)\subseteq
B\left(x_0,\alpha_0^k\cdot\max\nolimits_{x_0}(X)\right)\subseteq
\frac{3}{4}\cdot B_{\scl(m,k',q),j}.\end{equation}
Applying equations~(\ref{equation:futbal})
and~(\ref{equation:football}) to equation~\eqref{equation:llorar}, we see that equation~\eqref{equation:llorar}, and
hence the inner integral of equation~\eqref{equation:rutger} is bounded (uniformly in $X\in T_0\left(\underline{P_{\scl(m,k',q),j}}\right)$) by
\begin{equation}\label{eq:nanana}\frac{2\cdot4^d\cdot
C_{\mu}^2}{\alpha_0^{4\cdot
d}}\cdot\beta_2^2\left(B_{\scl(m,k',q),j},L\right).\end{equation}
Applying equations~\eqref{eq:nanana},~\eqref{equation:agro-inclusion}, and the usual computations to the RHS of
equation~(\ref{equation:rutger}) gives
\begin{multline}\label{equation:john-smith}\int_{\underline{P_{\scl(m,k',q),j}}}\left(\frac{\dist\left(z_{\ell},L\right)}
{\diam\left(B_{\scl(m,k',q),j}\right)}\right)^2\frac{\fdi\mu^{M_n}\left(\overline{X}\times
Y_{X^s}\right)}{\diam(X)^{d(d+1)}\cdot
f_k^n\left(\overline{X}\right)\cdot
f_{k'}^d\left(\underline{X^s}\right)}\leq\\\frac{2\cdot4^d\cdot
C_{\mu}^{d+3}}{\alpha_0^{d(d+1)+4\cdot d}}\cdot\alpha_0^{k\cdot
d\cdot
(d-n+1)}\cdot\beta_2^2\left(B_{\scl(m,k',q),j},L\right)\cdot\mu\left(B_{\scl(m,k',q),j}\right).\end{multline}

At last we consider Case~5. Here we must be a little bit careful in how we use notation, in the sense that we must remember the pertinent ``variables''.  In this case we make the following harmless abuse of notation for the truncation $T_0$, taking $$T_0(\underline{X^s})=T_0(\overline{X}\times Y_{x^s})=\overline{X},$$
instead of $T_0(\underline{X^s})=X^s$ as we originally defined the notion in Section~\ref{subsubsection:trunc-proj-1}.

Then, via the usual computations, for  $1\leq i\leq k'\cdot d$ and
$\overline{X}\in
T_0\left(\underline{P_{\scl(m,k',q),j}}\right)$ we have
\begin{multline}\label{equation:slowly-dying}\int_{\left\{Y_{X^s}:\overline{X}\times
Y_{X^s}\in\underline{P_{\scl(m,k',q),j}}\right\}}\left(\frac{\dist\left(y_{i},L\right)}
{\diam\left(B_{\scl(m,k',q),j}\right)}\right)^2\frac{\fdi\mu^{k'\cdot
d}\left(Y_{X^s}\right)}{f_{k'}^d\left(\underline{X^s}\right)}\leq\\\frac{2\cdot
4^d\cdot C_{\mu}^2}{\alpha_0^{6\cdot d}}\cdot
\beta_2^2\left(B_{\scl(m,k',q),j},L\right).\end{multline} Hence, iterating the integral over
$\underline{P_{\scl(m,k',q),j}}$ gives the inequality
\begin{multline}\label{equation:arnold}\int_{\underline{P_{\scl(m,k',q),j}}}\left(\frac{\dist\left(y_{i},L\right)}
{\diam\left(B_{\scl(m,k',q),j}\right)}\right)^2\frac{\fdi\mu^{M_n}\left(\overline{X}\times
Y_{X^s}\right)}{\diam(X)^{d(d+1)}\cdot
f_k^n\left(\overline{X}\right)\cdot
f_{k'}^d\left(\underline{X^s}\right)}\leq\\\frac{2\cdot 4^d\cdot
C_{\mu}^{d+3}}{\alpha_0^{d(d+1)+6\cdot d}}\cdot\alpha_0^{k\cdot
d\cdot(d-
n+1)}\cdot\beta_2^2\left(B_{\scl(m,k',q),j},L\right)\cdot\mu\left(B_{\scl(m,k',q),j}\right).\end{multline}

At long last, applying
equations~(\ref{equation:maximal-honest}),~(\ref{equation:john-smith}),
and~(\ref{equation:arnold}) to the RHS of
equation~(\ref{equation:chavez}) and taking the infimum over
all $d$-planes $L$ proves the desired proposition.


\section*{Acknowledgement}
Thanks to
Immo Hahlomaa
for careful reading of an earlier version of this manuscript and his
helpful and detailed comments, and Peter Jones as well as Pertti
Mattila for various constructive suggestions and advice. This work
has been supported by NSF grant \#0612608.

\bibliographystyle{plain}
\bibliography{menger-less-betas_revista}

\end{document}